\documentclass[12pt]{article}

\usepackage{amsfonts,amssymb,amsmath,amsthm}
\usepackage{fullpage}
\usepackage{bm}
\usepackage{dsfont}
\usepackage[pdftex]{graphicx}
\usepackage{xcolor}
\usepackage{epstopdf}
\usepackage[multiple]{footmisc}
\usepackage[colorlinks,
			pdffitwindow=false,
      plainpages=false,
      pdfpagelabels=true,
     	pdfpagemode=UseOutlines,
      pdfpagelayout=SinglePage,
			bookmarks=false,
			colorlinks=true,
      hyperfootnotes=false,
			linkcolor=blue,
			citecolor=green!50!black]{hyperref}
\usepackage{tikz}
\usepackage{enumerate}
\usepackage{titlesec}
\usepackage{floatrow}
\usepackage{sidecap}
\usepackage{calc}
\usepackage{ifthen}
\usepackage{fancyvrb}
\usepackage{listings}
\usepackage{color}

\definecolor{mygreen}{rgb}{0,0.6,0}
\definecolor{mygray}{rgb}{0.5,0.5,0.5}
\definecolor{mymauve}{rgb}{0.58,0,0.82}

\newtheoremstyle{custom}{3pt}{3pt}{}{}{\bfseries}{:}{.5em}{}
\theoremstyle{custom}
\newtheorem{example}    {Example}

\newtheorem{theorem}    [example]{Theorem}
\newtheorem{lemma}      [example]{Lemma}

\newtheorem{remark} 		[example]{Remark}

\newcommand{\R}{\mathbb{R}}

\newcommand{\N}{\mathbb{N}}

\newcommand{\X}{\mathbb{X}}
\newcommand{\dd}{\mathrm{d}}

\newcommand{\set}[1]{\mathbb{#1}}		
\newcommand{\mat}[1]{\mathbf{#1}}		

\newcommand{\eps}{\varepsilon}

\renewcommand{\P}{\mathcal{P}}
\newcommand{\T}{\mathcal{T}}
\newcommand{\F}{\mathcal{F}}
\newcommand{\U}{{U}}
\newcommand{\Dep}{D_{\eps}}
\newcommand{\cDep}{\mathcal{D}_{\eps}}

\def \prob {\mathsf{P}}
\def \expect {\mathsf{E}}

\numberwithin{equation}{section}

\begin{document}

\title{
Understanding the geometry of transport:\\
diffusion maps for Lagrangian trajectory data unravel coherent sets
}
\author{Ralf~Banisch\thanks{School of Mathematics, University of Edinburgh, Edinburgh EH9 3FD, UK. E-mail: ralf.banisch@ed.ac.uk}
\and P\'eter~Koltai\thanks{Institute of Mathematics, Freie Universit\"at Berlin, 14195 Berlin, Germany. E-mail: peter.koltai{@}fu-berlin.de} }

\date{}

\lstset{language=Matlab}  


\maketitle

\begin{abstract}
Dynamical systems often exhibit the emergence of long-lived coherent sets, which are regions in state space that keep their geometric integrity to a high extent and thus play an important role in transport. In this article, we provide a method for extracting coherent sets from possibly sparse Lagrangian trajectory data. Our method can be seen as an extension of diffusion maps to trajectory space, and it allows us to construct ``dynamical coordinates'' which reveal the intrinsic low-dimensional organization of the data with respect to transport. The only a priori knowledge about the dynamics that we require is a locally valid notion of distance, which renders our method highly suitable for automated data analysis. We show convergence of our method to the analytic transfer operator framework of coherence in the infinite data limit, and illustrate its potential on several two- and three-dimensional examples as well as real world data.
\end{abstract}


\section{Introduction} \label{sec:intro}

The term \emph{coherent sets}, as used here, was coined in recent studies~\cite{FrLlQu10, FrLlSa10, FrSaMo10}. They are understood to be sets (one at each time point), in the state space of a flow governed by a possibly non-autonomous (time-variant) system, which keep their geometric integrity to a high extent, and allow little transport in and out of themselves. Natural examples are moving vortices in atmospheric~\cite{RypEtAl07,FrSaMo10}, oceanographic~\cite{TBBM03,DFHPG09,FrEtAl15}, and plasma flows~\cite{PHJJ07}.

Dynamical systems techniques have been developed for the qualitative and quantitative study of transport problems. Most of these are either geometric or transfer operator based (probabilistic) methods, but topological~\cite{AlTh12} and ergodicity-based~\cite{BuMe12} methods appeared recently as well.
Geometric approaches are mainly aiming at detecting transport barriers (\emph{Lagrangian coherent structures}), and include studying invariant manifolds, lobe dynamics~\cite{RKWi90}, finite-time material lines~\cite{Hal00}, geodesics~\cite{HaBeV12} and surfaces~\cite{Hal01}. The notions of shape coherence~\cite{MaBo14} and flux optimizing curves~\cite{BaFrSa14}  are also of geometric nature. Transfer operator based methods aim at detecting sets (i.e.\ full-dimensional objects in contrast to codimension one transport barriers), and consider almost-invariant~\cite{DeJu99, DeFr03} and coherent sets~\cite{FrLlQu10, FrLlSa10, FrSaMo10}. Efforts have been made to compare geometric and probabilistic methods and understand the connection between them~\cite{FrPa09, FrPa12, Froyland2015, AlPe15}.

Increasing computational and storage capacities, just as improving measurement techniques supply us with large amounts of data. Even if a tractable computational model is not available, analysis of this data can reveal much of the desired properties of the system at hand. Recently, different approaches emerged that compute coherent sets and coherent structures based on \emph{Lagrangian trajectory data}, such as GPS coordinates from ocean drifters: Ser-Giacomi et al~\cite{SGetal15} use graph theoretical tools to perform a geometric analysis of transport, the works~\cite{FrPa15,HaEtAl15} introduce dynamical distances and clustering to extract coherent sets as tight bundles of trajectories in space-time, while Williams et al~\cite{WRR15} use a meshfree collocation-based approach for a transfer operator based classification. 

Here, we introduce a method based on Lagrangian trajectory data, which (i) uses only \emph{local} distances between the data points, and which (ii) can be shown to ``converge'' to the \emph{analytical} transfer operator based framework of Froyland~\cite{Froyland2013} in the infinite-data limit; hence it can be viewed as a natural extension of the functional analytic framework to the sparse data case. Moreover, our approach provides \emph{dynamical coordinates} which shed light on the connectivity of coherent sets, and reveal how transport is occurring. One key ingredient here is to use \emph{diffusion maps}~\cite{coifman2006,nadler2006diffusion,lafon2006diffusion}, which were successfully applied to extract intrinsic geometrical properties from high-dimensional data. The basic idea there is to introduce a diffusion operator on the data points, whose eigenvectors will give a good low-dimensional parametrization of the data set, if this is possible.

This paper is organized as follows. In section~\ref{sec:cohset} we introduce the analytic transfer operator based framework of coherent sets. In section~\ref{sec:dmaps}, we first review the construction of diffusion maps. This is followed by our main result: The extension of diffusion maps to trajectory data and Theorem \ref{thm:main}, which shows that our method coincides with the analytic transfer operator approach in the rich data limit. We also discuss algorithmic aspects and show how to extract coherent sets. In~\cite{Froyland2013} Froyland draws a connection between the analytic transfer operator approach and geometric properties of coherent sets, which he formalizes in~\cite{Froyland2015}. In section~\ref{sec:epsfree} we seek \emph{direct} connections between our data-based framework and this geometry-oriented construction. Finally, section~\ref{sec:numerics} demonstrates our method for diverse numerical examples.

In this paper, we denote sets by double-stroke symbols~(e.g.~$\set{A}$), matrices whose size is compatible with the data by upper case bold face symbols~(e.g.~$\mat{P}$), and operators on (weighted)~$L^2$-spaces by calligraphic symbols~(e.g.~$\mathcal{P}$). $\|\cdot \|$ is always the Euclidean norm on~$\R^n$, for some~$n\in\N$, unless stated otherwise.


\begin{figure}[h]
\centering

\includegraphics[width = 0.49\textwidth]{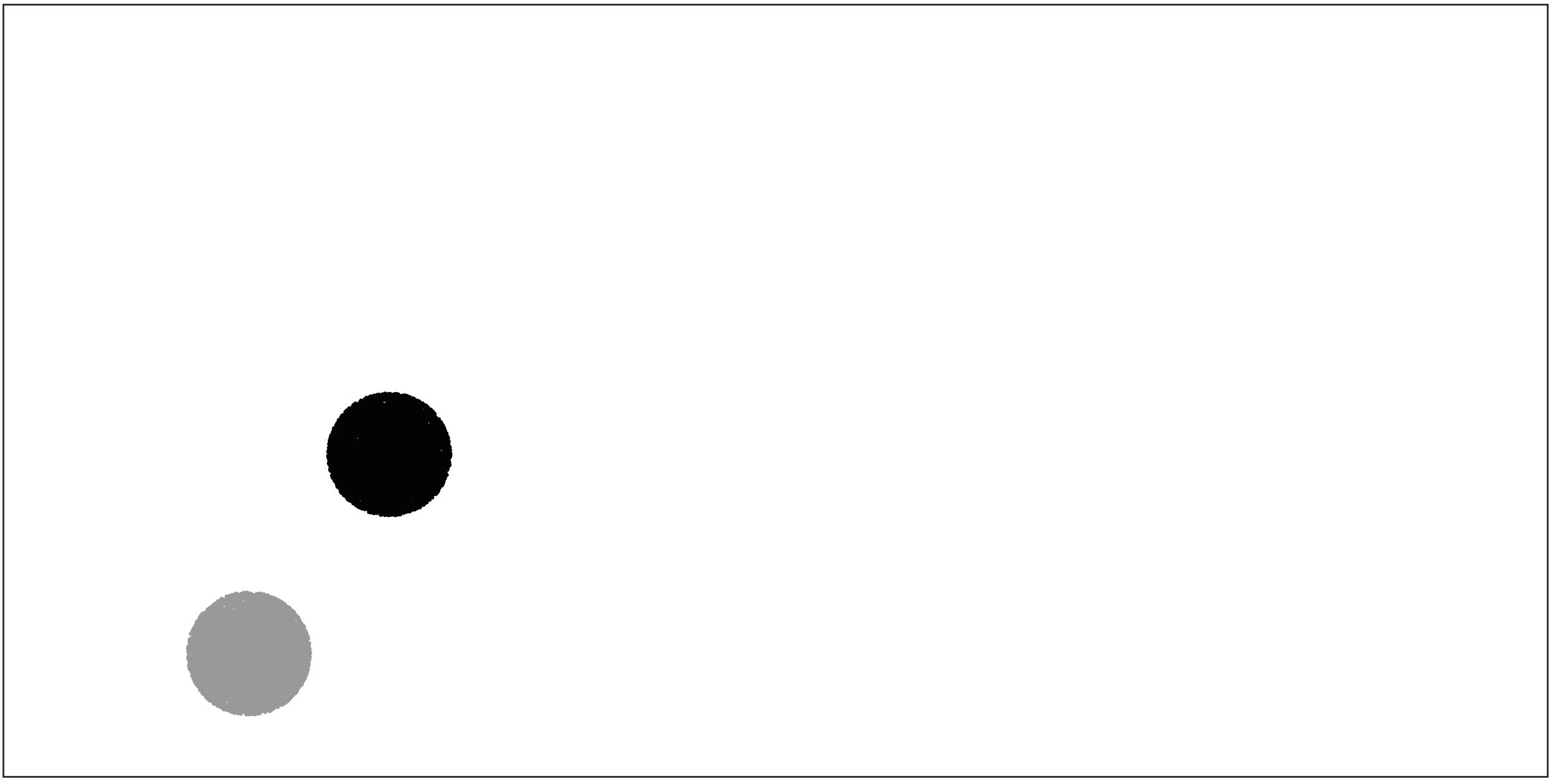}
\hfill
\includegraphics[width = 0.49\textwidth]{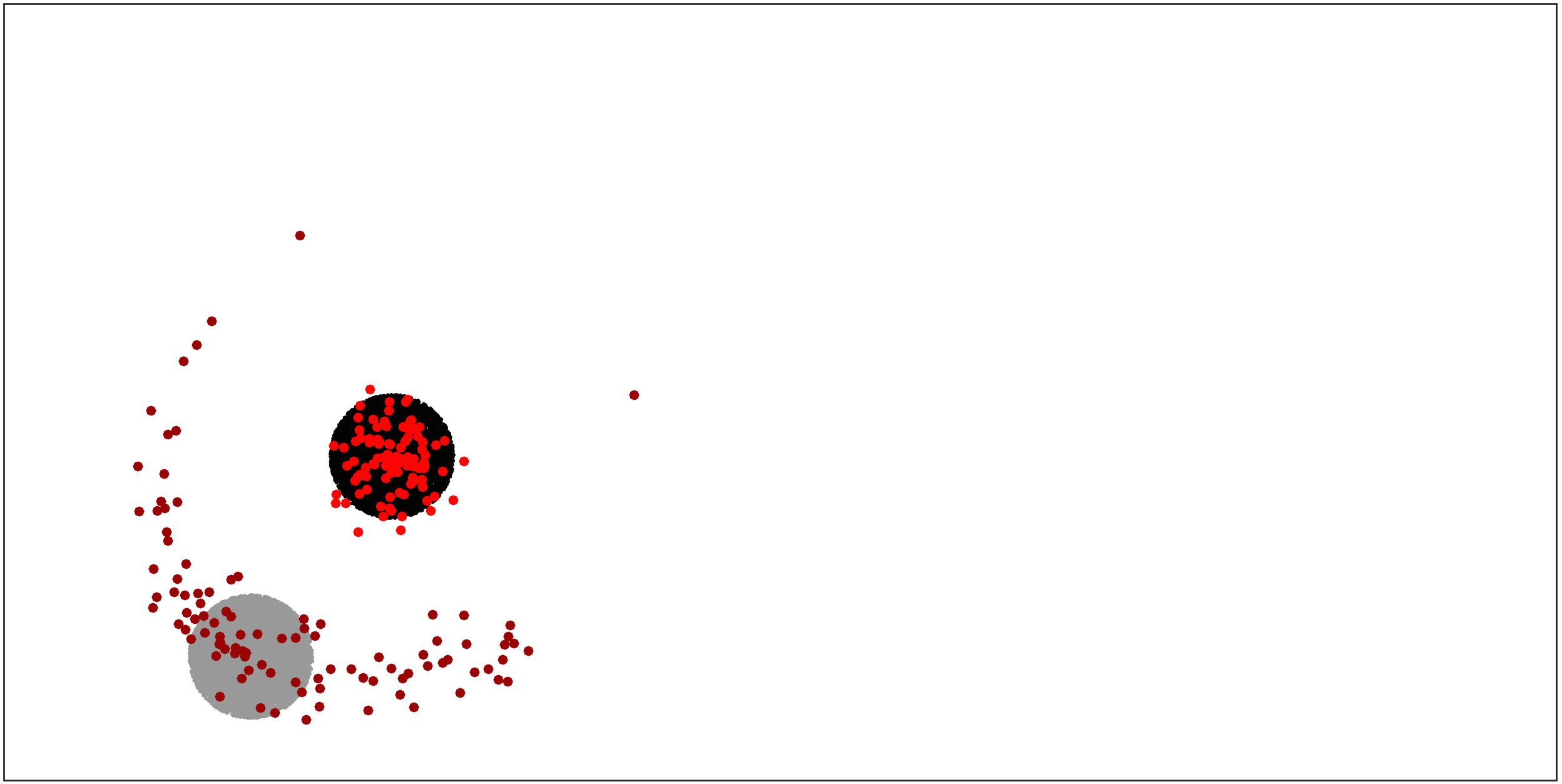}

\includegraphics[width = 0.49\textwidth]{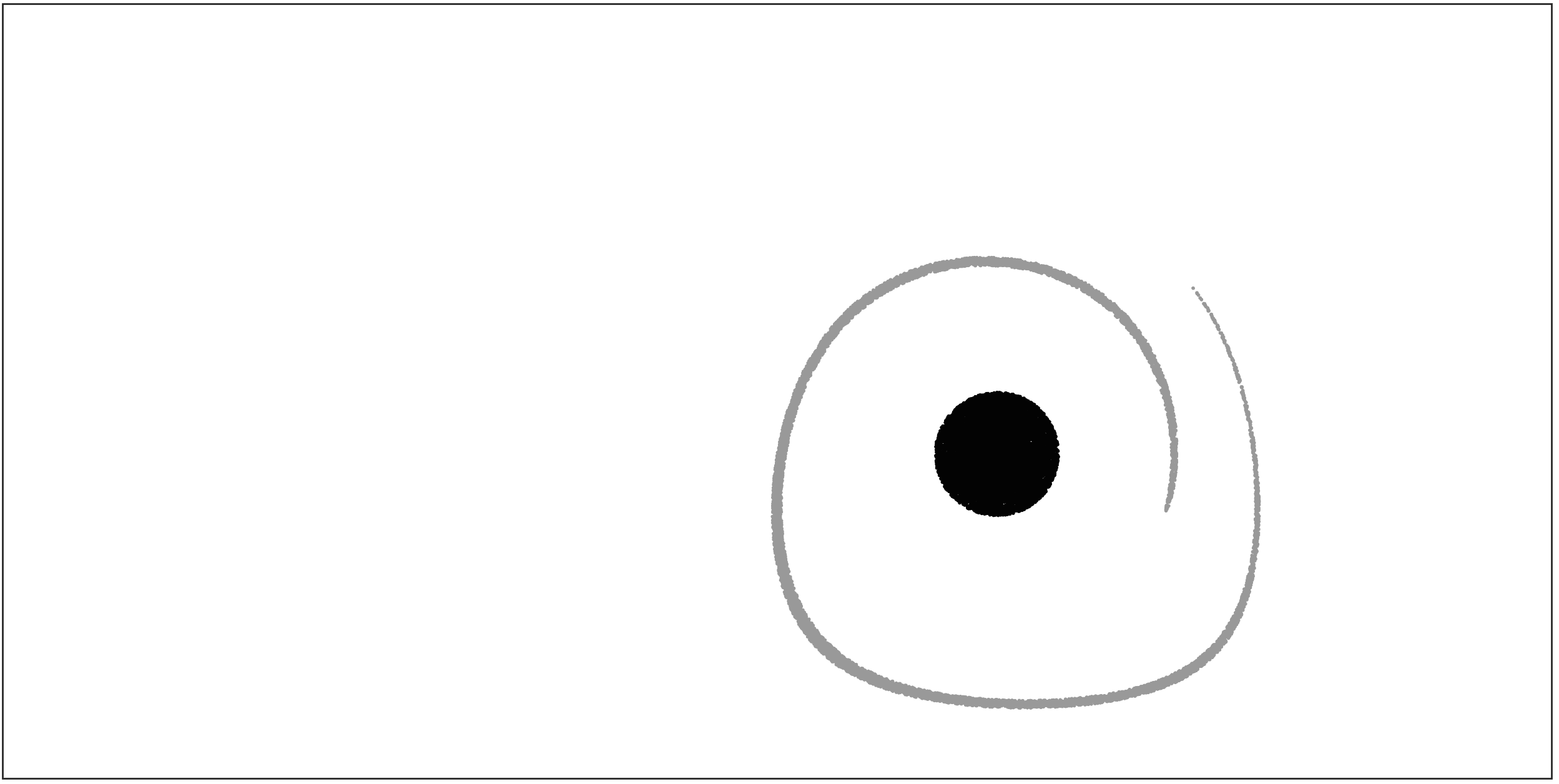}
\hfill
\includegraphics[width = 0.49\textwidth]{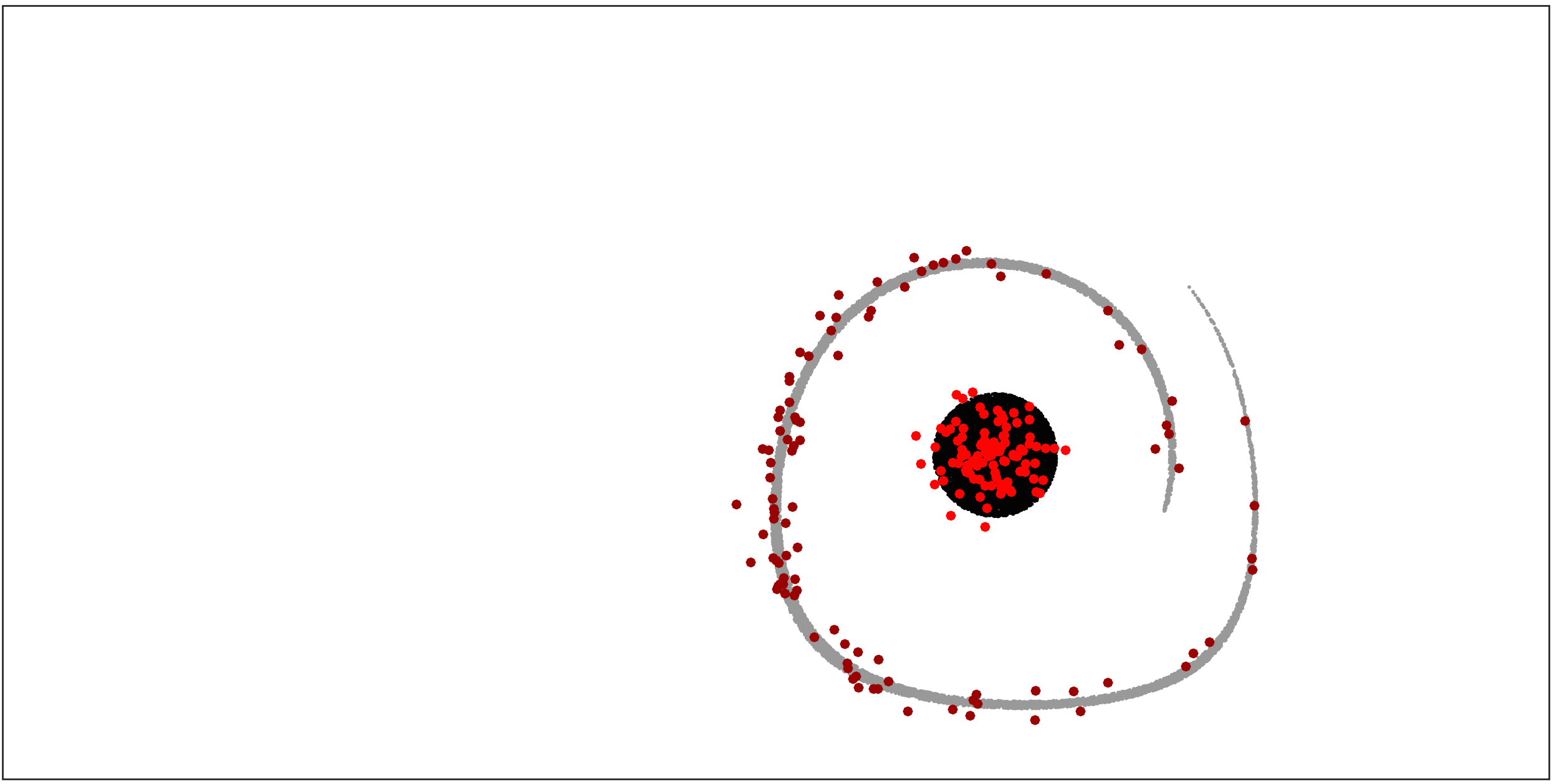}
\caption{Coherent pairs are robust under small perturbations. Top left: two sets at initial time. Bottom left: the image of these two sets under the dynamics. Bottom right: the images of 100 random test points (taken in the respective sets at initial time) under the dynamics, perturbed by a small additive random noise. Top right: preimages of the perturbed image points under the dynamics, which reveal that the black set and its image under the dynamics form a coherent pair, whereas the grey set and its image do not.}
\label{fig:forward_backward_points}
\end{figure}

\section{The analytic framework for coherent sets} \label{sec:cohset}

In this section we quantitatively formalize what we mean by a coherent set. To this end we review Froyand's analytic framework~\cite{Froyland2013} for \emph{coherent pairs}. In particular, we make a small simplifying modification to this, which we will comment on in (d) below.


Let a map~$\Phi:\set{X}\to \set{Y}$ be given, describing the evolution of states under the dynamics from some initial to some final time. We assume~$\set{X},\set{Y}\subset\R^d$, $d\in\N$, to be bounded sets. Consider a pair of sets,~$\set{A}\subset \set{X}$ at initial and $\set{B}\subset \set{Y}$ at final time. In order for $\set{A}$ and $\set{B}$ to form a \emph{coherent pair}, we must have $\Phi\set{A}\approx \set{B}$. This is however not enough, as Figure \ref{fig:forward_backward_points} readily suggests: If we are to distinguish between sets that keep their geometric integrity under the dynamics to a high degree and sets which do not, then we additionally need a robustness property under small perturbations.

Let~$\Phi_{\eps}:\set{X}\to \set{Y}$ denote a small random perturbation of~$\Phi$. The meaning of this is made precise below, for now, one may think of~$\Phi_{\eps} x$ as~$\Phi x$ plus some zero-mean noise with variance~$\eps$, where~$\eps$ is small. Now, a coherent pair has to satisfy~``$\Phi_{\eps}\set{A}\approx \set{B}$'', and~``$\Phi_{\eps}^{-1}\set{B} \approx \set{A}$'' in a suitable sense. In particular,~$\set{A}\subset \set{X}$ can be part of a coherent pair only if~``$\Phi_{\eps}^{-1}\left(\Phi_{\eps} \set{A}\right)\approx \set{A}$''. Note that exactly this is depicted in Figure~\ref{fig:forward_backward_points}, if we start at the top left image and proceed counterclockwise: applying forward dynamics, then diffusion, then the backward dynamics to the points of a coherent set, most of these points should return to the set.

In formalizing the expression~$\Phi_{\eps}^{-1}\left(\Phi_{\eps} \set{A}\right)\approx \set{A}$, \emph{randomness} plays an important role.
We define a non-deterministic dynamical system~$\Psi:\set{X}\to \set{Y}$ by its \emph{transition density function} $k\in L^2(\set{X}\times \set{Y},\mu\times\ell)$. Here~$L^2$ denotes the usual space of square-integrable functions in the Lebesgue sense,~$\mu$ is a probability measure (some reference measure of interest), and~$\ell$ is the Lebesgue measure. We have for the probability that~$\Psi x\in \set{S}$ for some Lebesgue-measurable set~$\set{S}$, that
\begin{equation}
\prob[\Psi x\in \set{S}] = \int_{\set{S}} k(x,y)\,dy\,.
\label{eq:transdens}
\end{equation}
In particular, (i)~$k\ge 0$ almost everywhere; and (ii)~$\int k(\cdot,y)dy = \mathbf{1}$, the constant one function.\footnote{If the range of an integral is not specified, then it is meant to be the whole domain of the integrand.}
From~\eqref{eq:transdens} we can compute that if~$\Psi x = \Phi_{\eps} x := \Phi x  + \sqrt{\eps}\bm\eta$, where~$\eps>0$, and~$\bm\eta$ is a random variable with density~$h$ with respect to~$\ell$, then~$k(x,y) = h\left(\eps^{-1/2}\left(\Phi x -y\right)\right)$.

We introduce the \emph{forward operator}~$\F:L^2(\set{X},\mu)\to L^2(\set{Y},\ell)$ associated with~$\Psi$, by~$\F f = \int k(x,\cdot)f(x)d\mu(x)$.
The operator~$\F$ describes how an ensemble of states which has distribution~$f(x)d\mu(x)$ is mapped by the dynamics; i.e.~$\F f$ is the distribution (given as a density with respect to~$\ell$) of the ensemble after it has been mapped state-by-state by~$\Psi$.

It is not necessary for the initial distribution to be stationary, thus we normalize our transfer operator.\footnote{We call every operator, transporting some object (a distribution, or an observable) by the dynamics, a \emph{transfer operator}.} Let $q_{\nu}:=\F \mathbf{1}$ be the image density of the initial distribution, defining a measure~$\nu$ through~$d\nu(x) = q_{\nu}(x)dx$. The normalized forward operator~$\T:L^2(\set{X},\mu)\to L^2(\set{Y},\nu)$ is then defined as
\begin{equation}
\T f = (\F\mathbf{1})^{-1}\F f = \int \frac{k(x,\cdot)}{q_{\nu}(\cdot)}f(x)d\mu(x)\,,
\label{eq:Tf}
\end{equation}
and its adjoint~$\T^*:L^2(\set{Y},\nu)\to L^2(\set{X},\mu)$ turns out to be
\begin{equation}
\T^* g = \int \frac{k(\cdot,y)}{q_{\nu}(y)}g(y)\,d\nu(y) = \int k(\cdot,y)g(y)\,dy \,,
\label{eq:T*f}
\end{equation}
i.e.\ $\langle \T f,g\rangle_{\nu} = \langle f,\T^* g\rangle_{\mu}$ for every $f\in L^2(\set{X},\mu)$, $g\in L^2(\set{Y},\nu)$, where~$\langle\cdot,\cdot\rangle_{\mu}$ and~$\langle\cdot,\cdot\rangle_{\nu}$ are the usual inner products in the respective spaces. Note that $\T \mathbf{1} = \mathbf{1}$, which just encodes the fact that the initial reference distribution~$\mu$ is mapped onto the final distribution~$\nu$ by the dynamics. Now, if~$\T$ is associated with~$\Phi_{\eps}$, then~$\set{A}$ and~$\set{B}$ being a coherent pair reads as~$\T\mathbf{1}_{\set{A}} \approx \mathbf{1}_{\set{B}}$. Note that this approximation can be made quantitative, since~$\tfrac{1}{\mu(\set{A})}\langle\T\mathbf{1}_{\set{A}},\mathbf{1}_{\set{B}}\rangle_{\nu}$ is the probability that a~$\mu$-distributed initial state from~$\set{A}$ gets mapped by~$\Psi$ into~$\set{B}$. Furthermore, $\T^*$ is the forward operator of the \emph{time-reversed} dynamics (see Appendix \ref{sec:adjoint} for a short proof). 

Froyland~\cite{Froyland2013} extracts coherent pairs from the left and right singular vectors of~$\T$ for dominant singular values. Right singular vectors of~$\T$ are eigenvectors of~$\T^*\T$.
Moreover,~$\T^*\T$ is the transfer operator of the ``forward-backward system''. Here, the ``backward system'' denotes the time-reversed forward system, and the forward system is described by the forward operator~$\T$. Coherent sets are those sets which are hard to exit under the forward-backward dynamics, i.e.~$\langle\T^*\T\tfrac{\mathbf{1}_{\set{A}}}{\mu(\set{A})},\mathbf{1}_{\set{A}}\rangle_{\mu}\approx 1$, which is the probability that the forward-backward system ends up in~$\set{A}$, provided it started there. This statement is a quantitative version of~``$\Phi_{\eps}^{-1}\left(\Phi_{\eps}\set{A}\right)\approx \set{A}$''.
Thus, the method described in Ref.~\cite{Froyland2013} is a \emph{spectral clustering}~\cite{von2007tutorial, Lambiotte09} of the forward-backward system.

A few remarks are in order:

(a). We see from~\eqref{eq:Tf} and~\eqref{eq:T*f} that
\begin{equation}
\T^*\T f(z) = \int f(x) \underbrace{\int \frac{k(z,y)k(x,y)}{q_{\nu}(y)}\,dy}_{=: \kappa(x,z)}\,d\mu(x)\,.
\label{eq:T*T}
\end{equation}
The kernel~$\kappa$ is trivially symmetric, but also doubly stochastic: $\int \kappa(x,\cdot)\,d\mu(x) = \linebreak \int \kappa(\cdot,z)\,d\mu(z) = \mathbf{1}$. Symmetry of~$\kappa$ implies that the forward-backward process is \emph{reversible} with respect to~$\mu$.

(b). If~$\Psi = \Phi$ is the deterministic dynamics, the forward operators~$\F$ and~$\T$ are often called the Perron--Frobenius operator~\cite{LaMa94}. In this case we denote the normalized forward operator by~$\P$. Note that here the kernel~$k(x,y) = \delta(\Phi x -y)$ is only formally an~$L^2$ function, where~$\delta$ is the \emph{Dirac distribution}, satisfying~$\delta(u)=0$ for~$u\neq 0$, and~$\int \delta(u)du=1$. By~\eqref{eq:T*f} we have~$\P^*g(y) = g(\Phi y)$, which is called the \emph{Koopman operator}. We will denote the formal Koopman operator by~$\U$, given by~$\U g(x) = g(\Phi x)$, to decouple its definition from the function spaces in consideration. However, it always holds true that if~$\U$ is considered as an operator from~$L^2(\set{Y},\nu)$ to~$L^2(\set{X},\mu)$, where~$\mu,\nu$ are arbitrary measures such that~$\U$ is well-defined, then the adjoint of the Koopman operator,~$\U^*$, is the Perron--Frobenius operator~$\P$ describing the dynamical transport of~$\mu$-densities to~$\nu$-densities. That is,~$\langle \P f,g\rangle_{\nu} = \langle f, Ug\rangle_{\mu}$ for all~$f\in L^2(\set{X},\mu)$,~$g\in L^2(\set{Y},\nu)$.

(c). We will also consider~$\Psi = \Phi_{\eps} = \Phi + \sqrt{\eps} \bm \eta$, where~$\bm\eta$ is a standard normally distributed\footnote{The distribution of~$\bm\eta$ is cut off at some specific distance from the mean, such that we can work on bounded sets. This might necessitate the enlargement of~$\set{Y}$, which we tacitly assume has been done, and denote the result by~$\set{Y}$ again.} random variable. This implies
\[
k(x,y) = \frac{1}{Z_{\eps}}\exp\left(-\eps^{-1}\|\Phi x-y\|^2\right)=: \frac{1}{Z_{\eps}}k_{\eps}(\Phi x,y)\,,
\]
with~$Z_{\eps}$ being a normalizing constant, independent on~$x$, and~$\|\cdot\|$ being the Euclidean norm on~$\R^d$. The noisy dynamics~$\Psi$ is given by two steps (\emph{first} apply~$\Phi$, \emph{then} add noise), hence the associated forward operator is a concatenation of the forward operators of the two components:~$\T = \cDep\P$. Here,~$\P:L^2(\set{X},\mu)\to L^2(\set{Y},\nu_{\Phi})$ is the normalized Perron--Frobenius operator, where~$\nu_{\Phi}$ is the image of the distribution~$\mu$ under~$\Phi$. The diffusion operator~$\cDep:L^2(\set{Y},\nu_{\Phi})\to L^2(\set{Y},\nu)$ is the normalized forward operator of the noise, where~$\nu$ is the image of~$\nu_{\Phi}$ under the noise. If we denote~$q_{\nu}$ the density of~$\nu$ with respect to~$\ell$, i.e.~$d\nu(x) = q_{\nu}(x)dx$, then
\begin{equation*}
\cDep f(x) = \frac{1}{Z_{\eps}q_{\nu}(x)}\int k_{\eps}(x,y)f(y)d\nu_{\Phi}(y)\,.
\end{equation*}
The evaluation chain can now be represented as
\[
\T: L^2(\set{X},\mu) \stackrel{\P}{\longrightarrow} L^2(\set{Y},\nu_{\Phi}) \stackrel{\cDep}{\longrightarrow} L^2(\set{Y},\nu)\,.
\]
For later reference, we also define the formal diffusion operator~$\Dep$ (i.e.\ without the spaces it acts on) by~$\Dep f(x) = \frac{1}{Z_{\eps}}\int k_{\eps}(x,y)f(y)\,dy$, and note that formally, due to the symmetry of the kernel~$k_{\eps}$, we have~$\langle \cDep f,g\rangle_{\nu} = \langle f,\Dep g\rangle_{\nu_{\Phi}}$. Hence, viewed as an operator between the right spaces,~$\Dep = \cDep^*$.

(d). In Froyland's construction~\cite{Froyland2013},~$\T$ is the forward operator associated with a process where a small diffusion is applied \emph{both before and after} the deterministic dynamics takes place.\footnote{In our setting, this would mean~$\Phi_{\eps}x = \Phi(x+\sqrt{\eps}\bm\eta_1) + \sqrt{\eps}\bm\eta_2$, where~$\bm\eta_1,\bm\eta_2$ are independent random variables.} This assures that both the sets~$\set{A}$ and~$\set{B}$ are geometrically nice (cf.~Figure~\ref{fig:forward_backward_points}). This can be circumvented as follows. If one would like to have coherence at several time instances, it is natural to average the operators~$\T^*\T$ for all the different time instances, and compute the dominant eigenfunctions of the resulting operator~\cite{FrPa14,Froyland2015}.\\
To make this precise, let for arbitrary time instances~$s \le t$, the forward operator~$\T_{s,t}$ correspond to the deterministic dynamics from~$s$ to~$t$, plus random noise scaled by a parameter~$\eps$ (as above). Given a set of final time instances,~$I_T:=\{t_0,\ldots,t_{T-1}\}$, at which we would like to find coherent sets (now \emph{tuples}, instead of pairs), one can consider the dominant eigenfunctions of
\begin{equation}
\frac{1}{T}\sum_{t\in I_T} \T_{t_0,t}^*\T_{t_0,t}\,.
\label{eq:avgT*T}
\end{equation}
Note that~$t_0\in I_T$ in~\eqref{eq:avgT*T}, hence we automatically account for the geometrical smoothness of the sets at initial time too. Since we will adopt this construction in the current work, it suffices to take~$\T$ as the forward operator associated with the deterministic forward dynamics plus some small diffusion (at final time). Our data-based construction in section~\ref{sec:main} is going to approximate the operator~\eqref{eq:avgT*T}.

(e). So far, the choice of the small random perturbation (diffusion) which we apply to the dynamics was arbitrary. In practice, choosing the size of this perturbation does not have to be obvious. In~\cite{Froyland2015}, Froyland hence developed an~``$\eps$-free'' version of the notion of coherent pairs. In fact, he derives a first order perturbation expansion for~$\eps\to 0$ for the construction from~(d).\\
We summarize this for coherent pairs, i.e.\ where only two time instances are involved, say~$s$ and~$t$, $s<t$. Note that then~$\T_{s,s} = \cDep\P_{s,s} = \cDep$ corresponds to only diffusion, since~$\P_{s,s}=\mathrm{Id}$, the identity. Froyland shows that if the deterministic dynamics is volume-preserving for every time, and the noise has zero mean and covariance equal to the identity matrix, then\footnote{Equation~\eqref{eq:Tepsfree} follows from Froyland's result if we take two differences into account: (i) we take~$\eps$ as the \emph{variance} of the noise, and not the \emph{standard deviation}, as he does, and (ii) we apply noise only after the dynamic evolution, and not before and after.}
\begin{equation}
\T_{s,s}^*\T_{s,s} + \T_{s,t}^*\T_{s,t} = \mathrm{Id} + \frac{\eps}{2} (\Delta + \P^*\Delta\P) + o(\eps)\,,
\label{eq:Tepsfree}
\end{equation}
where~$\Delta$ denotes the Laplace operator on~$\set{X}$, and~$o(\eps)$ means some function such that~$o(\eps)/\eps\to 0$ as~$\eps\to 0$. Equation~\eqref{eq:Tepsfree} holds pointwise in~$x$, if the operators therein are applied to a sufficiently smooth function~$f$. Froyland calls~$(\Delta + \P^*\Delta\P)$ the \emph{dynamic Laplacian}, and its eigenfunctions yield coherent pairs. We give a data-based version of \eqref{eq:Tepsfree} at the end of section \ref{sec:main}. In section~\ref{sec:epsfree} we further elaborate on data-based approximations of the dynamic Laplacian.

\section{Diffusion in trajectory space} \label{sec:dmaps}

\subsection{Diffusion maps}

To set the stage, we give a brief review of the method of diffusion maps. For details and the proofs of the statements presented in this section, we refer to \cite{coifman2006} and the references therein.

The goal of diffusion maps is to learn global geometric information from point-cloud data by imposing local geometric information only. Suppose that we have~$m$ data points~$x^i\in \R^n$ which are i.i.d.\ realizations of random variables distributed according to an unknown density~$q$ on a likewise unknown submanifold~$\set{M}\subset \R^n$ of dimension~$\dim\set{M} = d$. We assume throughout that $\set{M}$ is compact and $C^\infty$ and $q\in C^3(\set{M})$. The idea behind diffusion maps is that the Euclidean distance in~$\R^n$ is a good local approximation for distances in~$\set{M}$. 
Now, a Markov chain is constructed on the data points by the following procedure: Fix a rotation-invariant kernel\footnote{Due to the apparent connection between the kernels~$k_{\eps}$ from section~\ref{sec:cohset}~(c) and from~\eqref{eq:keps}, we abuse notation by denoting these objects by the same symbol. From now on, this latter definition applies.}

\begin{equation}
k_\eps(x^i,x^j) = h\left(\frac{\|x^i-x^j\|^2}{\eps}\right).
\label{eq:keps}
\end{equation}
Here, $\eps>0$ is a scale parameter and we will always choose $h(x) = c_r \exp(-x) \mathbf{1}_{x\leq r}$ with some cutoff radius~$r$ and the constant~$c_r$ chosen such that $\int h(\|x\|^2)\dd x = 1$. We will comment more on the choice of $\eps$ and $r$ below, but in practice we choose~$r$ large enough such that the second moment~$\int h(\|x\|)x_1^2\dd x \approx 1/2$ up to reasonable precision.\footnote{This is in order to have an explicit factor $\tfrac14$ in~\eqref{eq:Dmaps_eps_convergence}. For our work here, the exact value of this factor is irrelevant; we assume in our theoretical considerations that the second moment of~$h$ is~$\frac12$.}
Now, we let
\begin{equation*}
k_\eps(x^i) = \sum_{j=1}^m k_\eps(x^i,x^j)
\end{equation*}
and form the new kernel
\begin{equation*}
k_\eps^{(\alpha)}(x^i,x^j) = \frac{k_\eps(x^i,x^j)}{k_\eps(x^i)^\alpha k_\eps(x^j)^\alpha}
\end{equation*}
for some $\alpha\in [0,1]$. Finally, the transition matrix $\mathbf{P}_{\eps,\alpha}$ of the Markov chain is constructed by row-normalising $k_\eps^{(\alpha)}$:
\begin{equation}
\label{eq:diffmaps}
\mathbf{P}_{\eps,\alpha}(i,j) := \frac{k_\eps^{(\alpha)}(x^i,x^j)}{d^{(\alpha)}_\eps(x^i)}, \qquad d_\eps^{(\alpha)}(x^i) = \sum_{j=1}^m k_\eps^{(\alpha)}(x^i,x^j).
\end{equation}
In the limit $m\rightarrow\infty$ of infinite data, the strong law of large numbers ensures that all discrete sums converge almost surely to integrals over $q$. In particular, for~$f:\set{M}\to\R$,
\begin{equation}
\lim_{m\rightarrow\infty} \frac{1}{m}\sum_{j=1}^m k_{\eps}(x,x^j)f(x^j) = \int_{\set{M}}k_{\eps}(x,y)f(y)q(y)\dd y\,.
\label{eq:lim_int}
\end{equation}
The matrix $\mathbf{P}_{\eps,\alpha}$ only exists on the data points, but using the kernel function $h$ it is straightforward to extend the kernel $p_{\eps,\alpha}(x,x^j) = k_\eps^{(\alpha)}(x,x^j)/d^{(\alpha)}_\eps(x)$ to all $x\in\R^d$. We now define the operator $P_{\eps,\alpha}$ by
\begin{equation}\label{eq:Peps}
P_{\eps,\alpha} f(x) := \lim_{m\rightarrow\infty}\sum_{j=1}^m p_{\eps,\alpha}(x,x^j)f(x^j)
\end{equation}
and let $L_{\eps,\alpha} = \eps^{-1}\left(P_{\eps,\alpha} - \mathrm{Id}\right)$ be the corresponding generator. The error in (\ref{eq:Peps}) for finite~$m$ is of order~$\mathcal{O}(\eps^{-d/4}m^{-1/2})$ \cite{Hein2005,Singer2006}. In~\cite{coifman2006} Coifman et al showed that\footnote{The additional factor~$1/4$ comes from the fact that they choose~$h$ to have second moment equal to~$2$, we chose it to have~$\tfrac12$.}
\begin{equation}
\label{eq:Dmaps_eps_convergence}
\lim_{\eps\rightarrow 0}L_{\eps,\alpha} f = \frac{\Delta (f q^{1-\alpha})}{4q^{1-\alpha}} - \frac{\Delta (q^{1-\alpha})}{4q^{1-\alpha}}f
\end{equation}
holds uniformly on the space spanned by the first $K$ eigenfunctions of $\Delta$, for any fixed $K>0$. Here~$\Delta = \mathrm{div}\circ\mathrm{grad}$ is the (negative semi-definite) Laplace-Beltrami operator on~$\set{M}$. In particular, for $\alpha = 1$ one has~$\lim_{\eps\rightarrow 0}L_{\eps,1} f = \frac14\Delta f$.
In other words, the random walk generated by~$\mathbf{P}_{\eps,1}$ on the data points converges to Brownian motion on~$\set{M}$ as~$m\rightarrow\infty$ and~$\eps\rightarrow 0$. The method now proceeds by analysing the dominant spectrum of~$\mathbf{P}_{\eps,1}$. The number~$\Lambda$ of leading non-trivial eigenvalues is an estimator of the dimension of $\set{M}$, and the corresponding eigenfunctions~$\xi_i$, which converge in probability to those of~$\Delta$ for~$m\rightarrow\infty$ and~$\eps\rightarrow 0$, are good global intrinsic coordinates on~$\set{M}$. The~$\xi_i$ are the so-called \emph{diffusion maps} since they provide a map~$x^i\mapsto (\xi_1(x^i),\ldots, \xi_{\Lambda}(x^i))$ from~$\set{M}$ to the embedding space~$\set{E} = \mbox{span}\{\xi_1,\ldots, \xi_{\Lambda}\}$ \cite{coifman2006}. Figure~\ref{fig:dmap_example} gives an example.

\begin{figure}[h]
\centering
\includegraphics[width=0.35\textwidth]{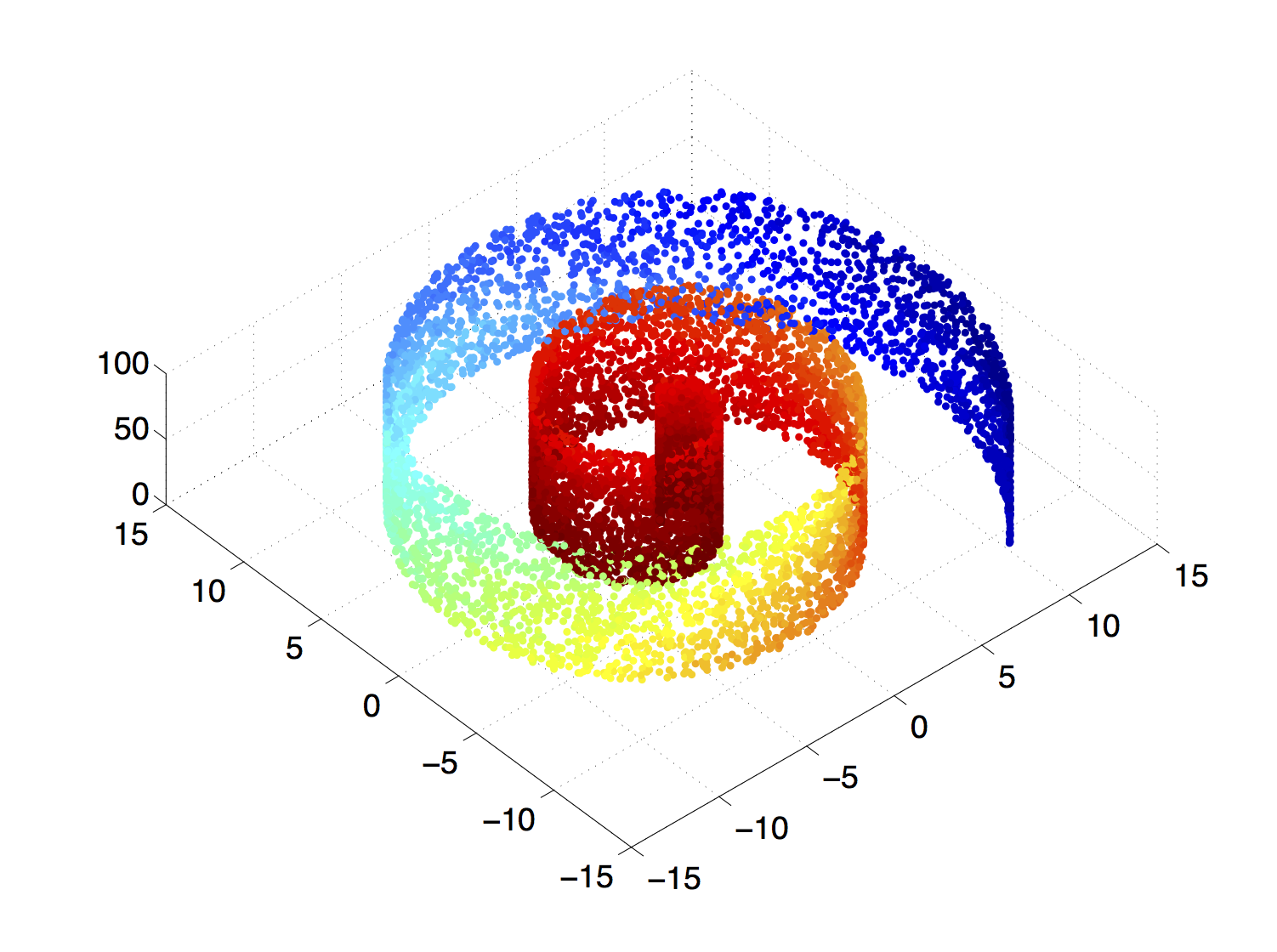}
\includegraphics[width=0.3\textwidth]{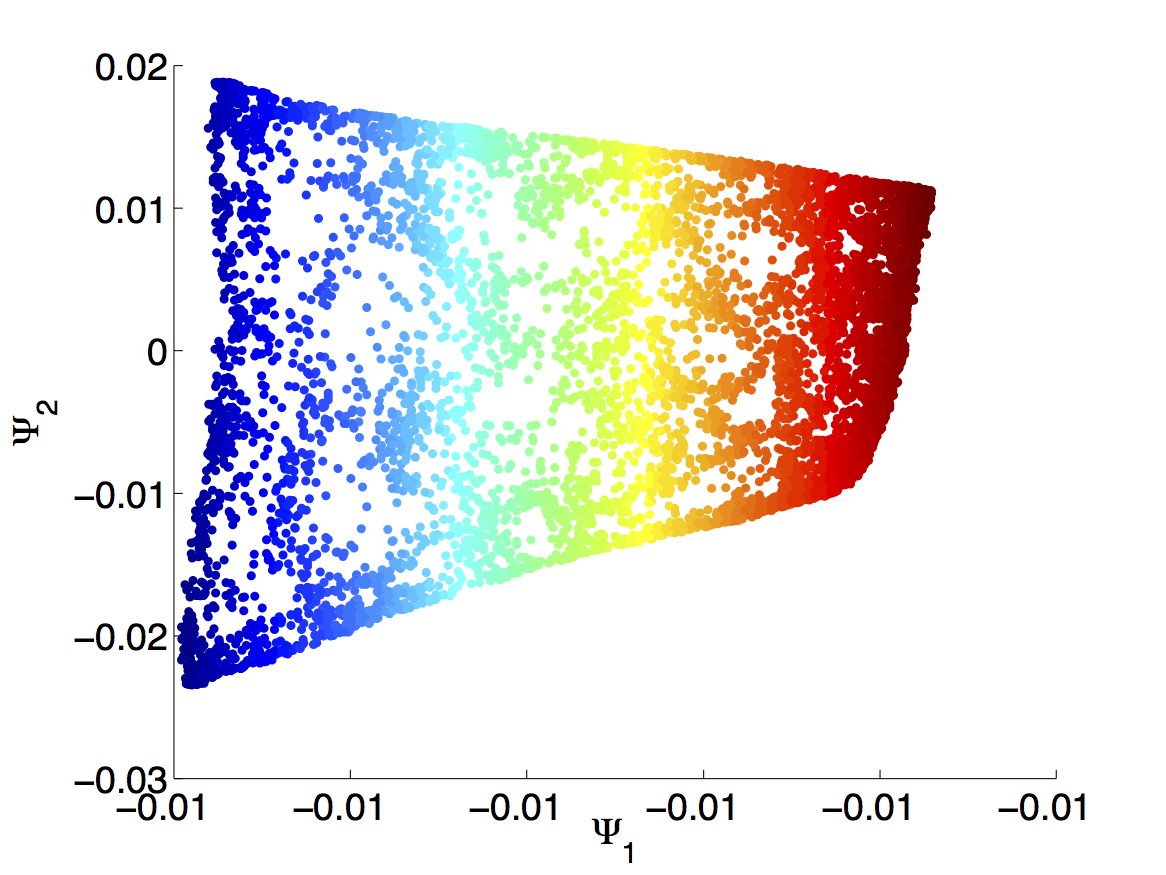}
\includegraphics[width=0.3\textwidth]{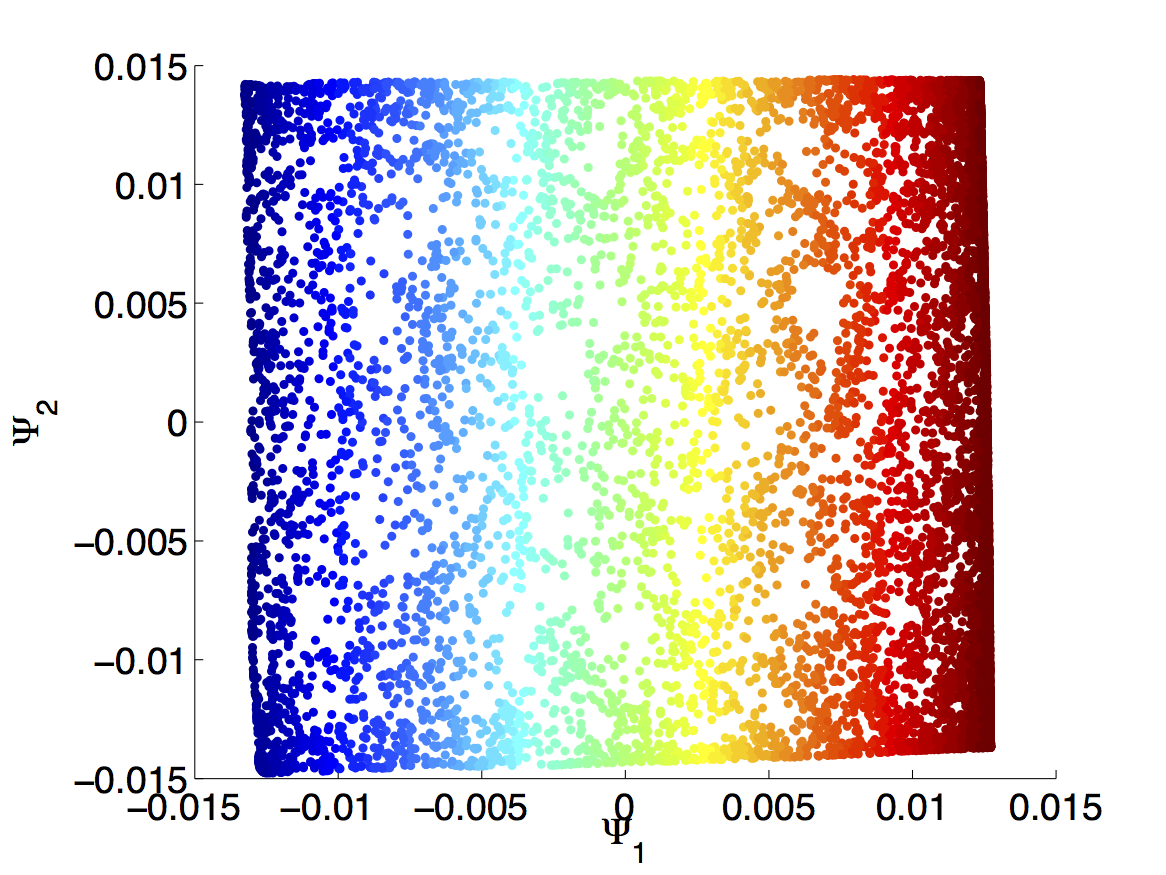}
\caption{Diffusion map example. Left: $m=10.000$ data points sampled from a rectangular strip embedded in~$\R^3$. The sampling density is non-uniform and decreases with distance from the origin. Color according to~$\xi_1$. Middle: Diffusion map embedding for~$\alpha=0$. Right: Diffusion map embedding for $\alpha=1$.}
\label{fig:dmap_example}
\end{figure}

\paragraph{Forward-backward diffusion maps.} 
Let us recast the diffusion operators from section~\ref{sec:cohset} (c) to our current setting, by defining
\begin{equation}\label{eq:diffusion_operators}
D_\eps f = \frac{1}{\eps^{d/2}}\int_{\set{M}}k_\eps(\cdot,y)f(y)\dd y, \quad \mathcal{D}_\eps f = \frac{D_\eps(qf)}{D_\eps q}\,.
\end{equation}
Further, for a matrix~$\mat{A}\in\R^{m\times m}$, let~$\mathbf{A} f(x^i) := \sum_j \mat{A}(i,j)f(x^j)$, and $\mathbf{A} f := (\mathbf{A} f(x^1),\ldots,\linebreak[3]\mathbf{A} f(x^m))$.

In order to make contact with the forward-backward dynamics developed in section~\ref{sec:cohset}, we will need a forward-backward version of diffusion maps. To minimize technical difficulties, we present this in Lemma \ref{lemma:fbdiffmaps} below for the case where $\set{M}$ has no boundary. If $\set{M}$ does have a boundary, then all the statements below hold uniformly on the set $\set{M}_\eps$ of points with distance at least $\eps^{\gamma}$ from the boundary for a fixed $0 < \gamma < \frac12$, while for points in $\set{M}\setminus \set{M}_\eps$ the presence of first-order derivatives in Taylor expansions of \eqref{eq:diffusion_operators} results in slightly worse asymptotics \cite{coifman2006}. Note that $\set{M}\setminus \set{M}_\eps$ is a set of measure $\mathcal{O}(\eps^\gamma)$, so this only has a mild effect.

\begin{lemma}\label{lemma:fbdiffmaps}
For $q$-distributed data points~$\{x^i\}_{i=1}^m$ on~$\set{M}\subset\R^n$ and~$\mathbf{P}_{\eps,0}$ as in (\ref{eq:diffmaps}), let us define the forward-backward diffusion matrix
\begin{equation}\label{eq:Bmatrix}
\mathbf{B}_\eps = (b_{\eps}(x^i,x^j))_{i,j=1}^m = \left(\mbox{diag}(\mathbf{P}^T_{\eps,0} \mathbf{1})\right)^{-1}\mathbf{P}_{\eps,0}^T\mathbf{P}_{\eps,0}\,,
\end{equation}
with~$\mbox{diag}(v)$ denoting the diagonal matrix with entries given by vector~$v$ on the diagonal.
By extending the kernel~$b_{\eps}$ from the data points to~$\R^d$ analogously as described after~\eqref{eq:lim_int}, we define the operator~$\mathcal{B}_\eps : L^2(\R^n,\mu) \rightarrow L^2(\R^n,\mu)$, where~$d\mu(x) = q(x)dx$, by
\begin{equation}\label{eq:Boperator}
\mathcal{B}_\eps f(x) := \lim_{m\rightarrow\infty} \sum_{j=1}^m b_\eps(x,x^j)f(x^j)\,,
\end{equation}
and denote its adjoint in~$L^2(\R^n,\mu)$ by~$\mathcal{B}_\eps^*$. Let $\set{M}$ have no boundary and let $f$ be bounded on $\set{M}$. Then we have the following properties:
\begin{enumerate}[(i)]
\item $\mathbf{B}_\eps \mathbf{1} = \mathbf{1}$ and $\mathcal{B}_\eps \mathbf{1} = \mathbf{1}$.
\item $|\mathbf{B}_\eps f(x^i) - \mathcal{B}_\eps f(x^i)| = \mathcal{O}(\eps^{-d/4}m^{-1/2})$.
\item $\mathcal{B}_\eps f = D_\eps\mathcal{D}_\eps f + \mathcal{O}(\eps^2)$ uniformly on $\set{M}$. As a consequence, $\mathcal{B}_\eps$ is almost self-adjoint: $|\mathcal{B}_\eps f(x) - \mathcal{B}^*_\eps f(x) | = \mathcal{O}(\eps^2)$ uniformly on $\set{M}$.
\item $\mathbf{B}_\eps$ is almost symmetric: $\|\mathbf{B}_\eps - \mathbf{B}_\eps^T\| \leq \mathcal{O}(\eps^2) + \mathcal{O}(\eps^{-d/4}m^{-1/2})$ for any compatible matrix norm~${\|\cdot \|}$.
\item If~$D_\eps q = q$, then~$\mathcal{B}_\eps f = D_\eps\mathcal{D}_\eps f$.
\item If, additionally,~$f\in C^3(\set{M})$, then $\lim_{\eps\rightarrow 0} \frac{1}{\eps}(\mathcal{B}_\eps f - f) = \tfrac12 q^{-1}\nabla \cdot \left (q\nabla f\right)$ holds pointwise on $\set{M}$.
\end{enumerate}
\end{lemma}
\begin{proof}
See Appendix~\ref{app:fbdiffmaps}.
\end{proof}
Property (vi) shows that in the small $\eps$ limit, $\mathcal{B}_\eps f$ approximates the action of the $q$-weighted Laplace-Beltrami operator on $\set{M}$ \cite{Hein2005}. Note the following two special cases of property (vi):\\
(a) If~$q\equiv \mathrm{const}$ is the uniform distribution, then~$\lim_{\eps\rightarrow 0} \frac{1}{\eps}(\mathcal{B}_\eps f - f) = \tfrac12 \Delta f$.\\
(b) If~$q = e^{-V}$ with some potential energy function $V:\set{X}\rightarrow \R$, then~$\lim_{\eps\rightarrow 0} \frac{1}{\eps}(\mathcal{B}_\eps f - f) = \tfrac12 \left(\Delta f - \nabla V\cdot \nabla f\right)$. Up to a factor $\tfrac12$, this is the infinitesimal generator of the diffusion process given by the stochastic differential equation~$d\bm x_t = - \nabla V(\bm x_t)dt +  d\bm w_t$, where~$\bm w_t$ is a standard Wiener process (Brownian motion). Note that~$q$ is the invariant distribution of this process.

\begin{remark}[Limiting kernel] \label{rem:limitkernel}
Note that (\ref{eq:Boperator}) can be written as
\begin{equation}
\label{eq:Bkernel_infty}
\mathcal{B}_\eps f(x) = \int b_\eps^\infty(x,y)q(y)f(y)\dd y\,,
\end{equation}
with the~$m$-independent \emph{limiting kernel}
\begin{equation}
b_\eps^\infty(x,y) = \frac{1}{d^\infty_\eps(x)}\frac{1}{\eps^d}\int \frac{k_\eps(x,z)k_\eps(z,y)}{q^2_\eps(z)}q(z)\dd z,
\label{eq:bInf}
\end{equation}
where we introduced the shorthands
\begin{equation}\label{eq:dinfty}
q_\eps(x) := D_\eps q(x), \quad d^\infty_\eps(x) = \frac{1}{\eps^{d/2}}\int \frac{k_\eps(x,y)}{q_\eps(y)}q(y)\dd y.
\end{equation}
One can think of~$q_\eps$ as the best way to represent~$q$ with the kernel functions~$k_\eps(x,y)$, and of~$d^\infty_\eps$ as the best way to represent~$\mathbf{1}$. If~$q_\eps = q$ holds, then~$d^\infty_\eps = \mathbf{1}$, and~$b^\infty_\eps(x,y)$ reduces to a symmetric, doubly stochastic kernel quite similar to~$\kappa$; cf.~\eqref{eq:T*T}. Intuitively, this observation will allow us to connect our diffusion maps construction below to the analytic framework of coherence, introduced in section~\ref{sec:cohset}.
\end{remark}

\subsection{Space-time Diffusion maps}\label{sec:main}

\paragraph{General setting.}

In this section, we combine the geometric ideas of diffusion maps with dynamics. Let~$\Phi_{s,t} : \X_s \rightarrow \X_t$ for~$s,t\in \R$ be the unknown, possibly non-autonomous flow map from~$\X_s\subset\R^d$ to~$\X_t \subset \R^d$,~$s,t\in\R$. The data set we have at our disposal consists of~$m$ trajectories evaluated at~$T\in\N$ time slices~$I_t = \{t_0,\ldots, t_{T-1}\}$. To simplify notation, we set~$\Phi_t:=\Phi_{t_0,t}$. That is, we have access to the data set
\[
X = \{x^i_t := \Phi_t x^i\: :\: i=1,\ldots,m;\; t\in I_t\}
\]
with initial points~$x^i\in \X$ that we assume to be i.i.d.\ realizations of random variables distributed according to the distribution~$q_0$. We call~$q_t$ the distribution of the points~$x_t^i$ at time~$t$. At every timeslice~$t\in I_t$, we can construct diffusion map matrices $\mathbf{P}_{\eps,\alpha,t}$ and a forward-backward diffusion matrix~$\mathbf{B}_{\eps,t}$ via (\ref{eq:Bmatrix}) by using the~$m$ data points~$\left\{\Phi_tx^i\right\}_{i=1}^m$. Then~$\mat{B}_{\eps,t}(i,j) = b_{\eps,t}(\Phi_t x^i,\Phi_t x^j)$, where
\begin{equation}
\label{eq:bt}
b_{\eps,t}(x,y) = \frac{1}{d_{\eps,t}(x)}\sum_{i=1}^m \frac{k_\eps(x,\Phi_tx^i)k_\eps(\Phi_tx^i,y)}{k_{\eps,t}(\Phi_t x^i)^2}, \quad d_{\eps,t}(x) := \sum_{i=1}^m \frac{k_\eps(x,\Phi_t x^i)}{k_{\eps,t}(\Phi_t x^i)}\,,
\end{equation}
and $k_{\eps,t}(\Phi_t x^i) := \sum_j k_\eps(\Phi_t x^i,\Phi_t x^j)$.
Now we construct a Markov chain on the trajectories by specifying the following \emph{Spacetime Diffusion Map} transition matrix~$\mathbf{Q}_{\eps} \in \R^{m\times m}$:
\begin{equation}
\label{eq:sptdm_matrix}
\mathbf{Q}_{\eps}(i,j) = \frac{1}{T}\sum_{t\in I_t}\mathbf{B}_{\eps,t}(i,j) = \frac{1}{T}\sum_{t\in I_t}b_{\eps,t}(\Phi_tx^i, \Phi_tx^j)\,.
\end{equation}
We will show in Theorem~\ref{thm:main} below, that~\eqref{eq:sptdm_matrix} is a data-based version of the time-averaged forward-backward transfer operators from~\eqref{eq:avgT*T}.
The transition matrix (\ref{eq:sptdm_matrix}) describes jumps between trajectories in the following manner: Starting at trajectory~$i$, first one of the timeslices~$I_t$ is selected uniformly at random. Then, the forward-backward diffusion map transition matrix (\ref{eq:Bmatrix}) \emph{at the selected timeslice} is used to jump to a new trajectory~$j$.

\paragraph{Connection to coherence.} The connection between the transition probabilities prescribed by (\ref{eq:sptdm_matrix}) and the notion of coherence is now intuitively clear: \emph{Coherent sets are tight bundles of trajectories.}
That is, if there is a subset~$I_b = \{i_1,\ldots,i_b\}$ of trajectories such that~$\|\Phi_t x^i-\Phi_tx^j\|$ is small for all~$i,j\in I_b$ and all~$t\in I_t$, we would like to see these trajectories as a part of a coherent set.
For such a tight bundle of trajectories, all transition probabilities assigned by~$\mathbf{Q}_{\eps}$ between~$i,j\in I_b$ will be large, and we should be able to identify~$I_b$ by clustering~$\mathbf{Q}_{\eps}$.

%
%
%

Our main result is the following theorem, which links the transition matrix~$\mathbf{Q}_{\eps}$ with the analytical coherence framework.
\begin{theorem}
\label{thm:main}
With~$\mathbf{Q}_{\eps}$ as in~\eqref{eq:sptdm_matrix}, we have for fixed~$\eps>0$
\begin{equation}
\label{eq:main_alpha0}
\lim_{m\rightarrow\infty}\mathbf{Q}_{\eps}f(x^i) = \frac{1}{T}\sum_{t\in I_t} \P^*_t \mathcal{D}_{\eps,t}^*\mathcal{D}_{\eps,t}\P_tf(x^i) + \mathcal{O}(\eps^2),
\end{equation}
where~$\mathcal{D}_{\eps,t}f:= D_{\eps}(fq_t)/D_{\eps}q_t$, cf~\eqref{eq:diffusion_operators}.
Convergence in \eqref{eq:main_alpha0} is a.s.\ as~$m\rightarrow\infty$. The pointwise error for finite~$m$ is~$\mathcal{O}(\eps^{-d/4}m^{-1/2})$.
\end{theorem}
\begin{proof}
See Appendix~\ref{app:mainthm} for the full proof. The idea, though, can be sketched with Figure~\ref{fig:Qexample} as follows.

Let us think of the values of the function~$f$ in the data points~$x^i$ as statistical weights. The collection of these weighted point measures approximates the distribution~$fq_0$. Fixing a time slice~$t$, if we assign the weight~$f(x^i)$ to the points~$x_t^i$, respectively, they will approximate the distribution~$(\P_t f)q_t$. Application of the matrix~$\mat{B}_{\eps,t}$ to this latter data vector redistributes the statistical weights like diffusion. Pulling back the new statistical weights to the data points at initial time (the new weight of each data point~$x_t^i$ is assigned to~$x^i$) approximates the application of~$\P_t^*$.
\end{proof}

\begin{figure}[ht]
\centering
\includegraphics[width=0.7\textwidth]{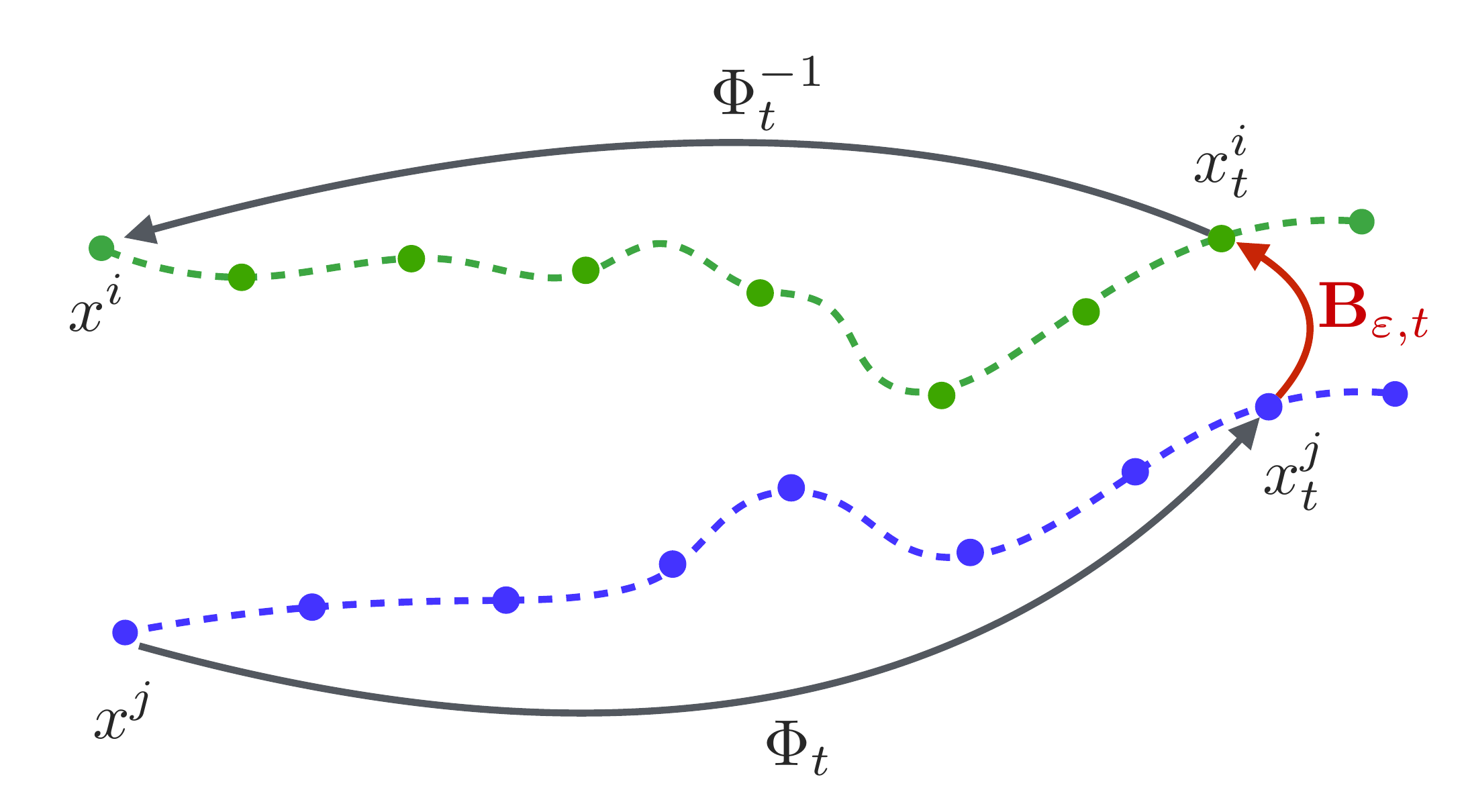}
\caption{To implement ``forward-diffuse-backward'', we start at $x^j$ at time $t_0$, then push the statistical weight of this data point,~$f(x^j)$, along the~$j^{\rm th}$ trajectory (shown in blue) to the data point~$x^j_t$, then use~$\mat{B}_{\eps,t}$ to redistribute the weights between the data points (this is diffusion, shown in red)  and finally transport the new weights along the~$i^{\rm th}$ trajectory back to initial time, to arrive at~$x^i$.}
\label{fig:Qexample}
\end{figure}

Theorem~\ref{thm:main} says that in the data-rich limit we are approximating the very analytical object that was designed to identify coherent pairs (tuples). In particular, the dominant eigenfunctions $\{\Xi_1,\ldots, \Xi_\Lambda\}$ of $\mathbf{Q}_\eps$ approximate those of the operator $\frac{1}{T}\sum_{t\in I_t} \P^*_t \mathcal{D}_{\eps,t}^*\mathcal{D}_{\eps,t}\P_t$, and have thus a significance for the dynamical system $\Phi_t$ which is similar to the significance of the diffusion map eigenfunctions $\{\xi_1,\ldots, \xi_\Lambda\}$ of $\mathbf{P}_{\eps,\alpha}$ for the purely geometrical problem. In fact $\{\Xi_1,\ldots, \Xi_\Lambda\}$ encode both dynamical and geometrical properties, and for this reason we call them \emph{spacetime diffusion maps}.

A few comments are in order:

\begin{enumerate}
\item[(a)] Although the operator~$\T^*\T$ is \emph{symmetric (self-adjoint)} and \emph{stochastic},~$\mat{Q}_{\eps}$ is merely stochastic by construction. It is, however, by Lemma~\ref{lemma:fbdiffmaps}~(iv),~$\mathcal{O}(\eps^2)$ close to a symmetric matrix, and this estimate is getting better as~$m\to\infty$, cf.~\eqref{eq:main_alpha0}. In all our numerical studies, the dominant spectrum of~$\mat{Q}_{\eps}$ was real-valued.

\item[(b)] The break of symmetry of~$\mat{Q}_{\eps}$ comes from the row-normalization by~$d_{\eps,t}$ in~\eqref{eq:bt}. As discussed in Remark~\ref{rem:limitkernel},~$d_{\eps,t}$ is our best approximation of the constant one function, hence the more data is available, the closer we get to this normalization not having any effect.

\item[(c)] Formally, we can apply our method, and construct the space-time diffusion matrix~$\mat{Q}_{\eps}$ also, if the trajectories are generated by \emph{non-deterministic dynamics}. In this case,~$\Phi_t x$ is a random variable for any fixed~$x$, and the Koopman operator of this dynamics is defined by~$\U_t g(x):=\expect [f(\Phi_t x)]$, where the expectation~$\expect[\cdot]$ is taken with respect to the law of~$\Phi_t x$.
Nevertheless, coherent sets can still be extracted from the dominant eigenmodes of~$\T^*_t\T_t$~\cite{FrSaMo10,DeJuMa15}, where~$\T_t$ is the (normalized) forward operator associated with the non-deterministic dynamics~$\Phi_t$. We study non-deterministic dynamics in a future publication.
\end{enumerate}

\begin{remark}[A data-based dynamic Laplacian] \label{rem:dynLap}
Suppose that the dynamics~$\Phi$ is a diffeomorphism\footnote{A diffeomorphism is a differentiable, everywhere invertible map, with a differentiable inverse.}, e.g.\ the solution of an ODE. Then~$\P^*_t\mathcal{P}_t = \mathrm{Id}$,\footnote{This is a consequence of~\cite[Corollary~3.2.1]{LaMa94}, by noting that the operator~$P$ therein plays the role of our~$\mathcal{F}$.} and using property~(vi) of Lemma~\ref{lemma:fbdiffmaps} it can be readily seen that
\begin{equation}\label{eq:dynLap}
\lim_{m\rightarrow\infty}\mathbf{Q}_{\eps}f(x^i) = \mathrm{Id} + \frac{\eps}{2} \frac{1}{T}\sum_{t\in I_t} \P^*_t \Delta_{q_t} \mathcal{P}_t f(x^i) + \mathcal{O}(\eps^2)
\end{equation}
with the~$q$-weighted Laplacian~$\Delta_q f := q^{-1}\nabla\cdot (q\nabla f)$. In other words, the operator~$\mathbf{L}_\eps = \eps^{-1}(\mathbf{Q}_\eps - \mat{I})$, with~$\mat{I}$ being the identity matrix matching the size of the data, is a data-based approximation of the operator~$\frac{1}{2T}\sum_{t\in I_t} \P^*_t \Delta_{q_t} \mathcal{P}_t$, which can be seen as a generalization of the dynamic Laplacian introduced in~\cite{Froyland2015} for non-uniform densities~$q$.
\end{remark}

\subsection{Clustering with Space-time diffusion maps}

A consequence of Theorem~\ref{thm:main} is that we may compute sets which are coherent for all times $t\in I_t$ by searching for a subset~$I_b = \{i_1,\ldots, i_b\}$ of trajectories which is \emph{metastable under~$\mathbf{Q}_\eps$}. This reduces the problem of computing coherent sets, which involves both geometry and dynamics, to the problem of clustering a graph. In this article, we use spectral clustering~\cite{von2007tutorial, Lambiotte09} on~$\mathbf{Q}_\eps$ to solve this problem, since
\begin{enumerate}[(i)]
\item spectral clustering can identify metastable sets~\cite{Schuette1999, Deuflhard2000}, and
\item eigenvectors of a diffusion maps transition matrix yield good coordinates to the intrinsic geometry of the data set; cf~Figure~\ref{fig:dmap_example}. We will see in section~\ref{sec:numerics} that the eigenvectors of~$\mat{Q}_{\eps}$ give natural ``\emph{transport coordinates}''.
\end{enumerate}
Alternative clustering methods, e.g.\ to save computational time, are of course possible and deserve further exploration, see for example \cite{schaeffer2007graph}. Any spectral clustering algorithm proceeds in the following three steps:
\begin{enumerate}
\item For some not too large $N$, compute the $N$ largest eigenvalues $\lambda_i$ of $\mathbf{Q}_\eps$. Identify $\Lambda$ such that $\lambda_{\Lambda} - \lambda_{\Lambda+1}$ is large (this is known as \emph{spectral gap}).
\item Compute the $\Lambda$ largest eigenfunctions $\Xi_1,\ldots, \Xi_\Lambda$.
\item Postprocessing: Extract $\Lambda$ clusters $C_1,\ldots, C_\Lambda$ from\footnote{$\Xi_1 = \mathbf{1}$ is the constant function.} $\Xi_2,\ldots, \Xi_\Lambda$.
\end{enumerate}


The justification for this approach is that a spectral gap after $\Lambda$ dominant eigenvalues indicates $\Lambda$ metastable sets, and that the eigenfunctions $\Xi_2,\ldots, \Xi_\Lambda$ are almost constant on the metastable sets \cite{Schuette1999, Deuflhard2000}. A number of different algorithms exist, depending on how the postprocessing step is handled, and whether hard or soft clusters are being sought. For example, the algorithms by Shi and Malik \cite{shi2000normalized} and Ng et al \cite{ng2002spectral} compute $\Lambda$ hard clusters which form a full partition of the state space $V = \{1,\ldots, m\}$ of $\mathbf{Q}_\eps$ by performing $k$-means on $\Xi_2,\ldots, \Xi_\Lambda$. We do the same in this paper\footnote{We normalize the $\Xi_i$ such that $\|\Xi_i\|_2 = 1$.}, mostly for reasons of simplicity and ease of implementation. However, we note that enforcing a full partition into metastable sets is often too strict. In some cases, it might be desirable to use fuzzy membership functions instead. We refer the reader to~\cite{deuflhard2005robust, bezdek2013pattern, sarich2014modularity} for more information.


\begin{remark}[Frame-independence]
The eigenfunctions~$\Xi_1,\ldots,\Xi_{\Lambda}$ are purely functions of the Euclidean distances~$\|x^i-x^j\|$ between Lagrangian observers. Since these are independent under a possibly time-dependent affine-linear transformation with orthogonal linear part, the values of the eigenfunctions in the data points are also independent of the transformation.
Thus, the algorithm is independent of the frame of reference, i.e.~objective.
\end{remark}

\subsection{Algorithmic aspects}

We describe an algorithm for extracting coherent sets from data, which we assume to be given as a~$d\times m \times T$ array of~$mT$ time-ordered data points in~$\R^d$. The algorithm has two stages:

\begin{enumerate}
\item Compute~$\mathbf{Q}_\eps$. The computational cost of this is dominated by the $m^2T$ distance computations between the~$mT$ data points. In practice, for any given point~$x^i$ only the distances to points within the cutoff radius~$r$, that is, only distances that satisfy~$\| x^i - x^j\|^2 \leq r\eps$, need to be computed and stored. This is a typical nearest neighbor search problem~\cite{arya1994optimal}, and an efficient implementation is readily available in many software packages. To illustrate, we provide a Matlab pseudocode that uses the \texttt{rangesearch} function, which solves this problem using k-d trees.

\begin{lstlisting}[frame=single]
Q = sparse(m,m);
range = sqrt(r*eps)
for t=1:T
	% retrieve data points in timeslice t
    data = pts(:,:,t);
    % compute all distances within range
    [idx, D] = rangesearch(data',data',range);
    % reshape output of rangesearch into sparse matrix
    K = assemble_sim_matrix(idx, D, eps);
    % compute diffusion map matrix B
    q = sparse(1./sum(K,2));
    Peps = diag(q)*K;
    deps = sparse(1./sum(Peps,1));
    B = diag(deps)*(transpose(Peps))*Peps;
    % add up to Q
    Q = Q + B;
end;
% normalize
Q = 1/T*Q;
\end{lstlisting}

Besides the distance computations, the only other computationally expensive task is the~$T$ sparse matrix multiplications that are needed to compute the matrices~$\mathbf{B}_{\eps,t}$. The cost for this can be estimated as~\cite{bulucc2012} $\mathcal{O}(b^2m T)$ in the best and~$\mathcal{O}(bm^2T)$ in the worst case, where~$b$ is the typical number of nonzero elements in any row of~$\mathbf{P}_{\eps,0,t}$. In many practical cases, one can avoid this cost by computing a simplified version of~$\mathbf{Q}_\eps$, e.g.\
\begin{equation}
\label{eq:tildeQ}
\mathbf{\tilde Q}_\eps = \frac{1}{T}\sum_{t\in I_t} \mathbf{P}_{\eps,\alpha,t}
\end{equation}
which requires no matrix multiplication. If~$\alpha = 1/2$ is chosen, then~$\mathbf{Q}_\eps$ and~$\mathbf{\tilde Q}_{2\eps}$ agree up to first order in~$\eps$ for large~$m$. This can be seen from comparing~(\ref{eq:Dmaps_eps_convergence}) for~$\alpha=1/2$ with Lemma~\ref{lemma:fbdiffmaps}~(vi).
If the densities $q_t$ are uniform for all~$t\in I_t$ then there will be no difference at all between~$\mathbf{Q}_\eps$ and $\mathbf{\tilde Q}_{2\eps}$. In our numerical experiments, $\mathbf{Q}_\eps$ and $\mathbf{\tilde Q}_{2\eps}$ always produced very similar results.

\item Run the spectral clustering algorithm. This requires the computation of the leading eigenvectors of~$\mathbf{Q}_\eps$, which is challenging for large~$m$ with a worst-case complexity of~$\mathcal{O}(m^3)$ even for sparse matrices~\cite{chen2011}. We do not discuss large-scale spectral clustering in here, since we have shown that our method ``converges'' to the analytical method in the data-rich limit. In that case other, Galerkin projection-based methods, are available, see~\cite{FrSaMo10,FrPa14,DeJuMa15,WRR15}, and~\cite{KlKoSch15} for an overview of methods.
For more information about fast spectral clustering algorithms we refer to~\cite{chen2006,fowlkes2004,liu2007}.
\end{enumerate}

\paragraph{Choice of parameters.} The cutoff radius~$r$ is used to tune the shape of the kernel function~$h$. For diffusion maps,~$r$ should be smaller then the scalar curvature of~$\set{M}$, which determines the length scale at which $\set{M}$ no longer looks locally flat \cite{Hein2005}. However, the scalar curvature is typically not known. We will choose~$r = 2$, which corresponds to a cutoff at $\exp(-r) \approx 0.1$. We found that increasing $r$ resulted in heavier computations due to the reduced sparsity\footnote{The sparsity of a matrix~$\mat{A}$ is the number of of nonzero entries of~$\mat{A}$ devided by the total number of entries.} of $\mathbf{Q}_\eps$, while having no significant effect on the results.

How should one choose~$\eps$ for a given amount of data, that is, for a given~$m$? There are two error terms present in~\eqref{eq:main_alpha0}, a variance term scaling as~$\mathcal{O}(\eps^{-d/4}m^{-1/2})$ and a bias term scaling as~$\mathcal{O}(\eps^2)$. This represents a trade-off: If~$\eps$ is reduced then the bias is decreased but the variance is increased. $\mathbf{Q}_\eps$ inherits this behavior from diffusion maps. The optimal choice of~$\eps$ for diffusion maps was investigated in~\cite{Singer2006} and found to be
\begin{equation}
\label{eq:epsopt}
\eps = \frac{C(\set{M})}{m^{1/(3+d/2)}}.
\end{equation}
Here~$C(\set{M})$ is an unknown constant which depends on the manifold~$\set{M}$. Equation~\eqref{eq:epsopt} tells us that if twice the amount of data is available, then we can reduce~$\eps$ by a factor of~$2^{-1/(3+d/2)}$. In practice, there are two good indicators for choosing $\varepsilon$: (i) One could choose $\eps$ such that the sparsity of~$\mathbf{Q}_\eps$ is between~$1\%$ and~$5\%$ (typically used values for sparse matrices), (ii) one may compute the dominant spectrum of $\mathbf{L}_\eps = \eps^{-1}(\mathbf{Q}_\eps - \mathbf{I})$ (see Remark \ref{rem:dynLap}) for different~$\eps$, and choose~$\eps$ based on minimal sensitivity of the eigenvalues, see section~\ref{sec:numerics} for more details.

\paragraph{Missing data.} Many real-world data sets are incomplete. For example, not all of the data points~$\left\{\Phi_t x^i\right\}_{t\in I_t}$ of any given trajectory might be available, but only some of them. Our algorithm can handle this naturally. We adopt the following convention: Whenever the distance~$\|\Phi_t x^i - \Phi_t x^j\|$ cannot be computed because e.g.~$\Phi_t x^i$ is missing, we set~$\|\Phi_t x^i - \Phi_t x^j\| = \infty$. This convention is easily implemented and leads to the~$i$-th row of~$\mathbf{B}_{\eps,t}$ being equal to $\delta_{ij}$. Hence, from the point of view of the Markov chain induced by~$\mathbf{Q}_\eps$, a missing data point~$\Phi_t x^i$ means that at the time slice~$t$, the Markov chain cannot leave or jump to trajectory~$i$.

\section{Is there an~$\eps$-free construction?} \label{sec:epsfree}

As already noted in section~\ref{sec:cohset} (e), the parameter~$\eps$ is in general artificial, and it is not immediate what are natural choices for it. Also, one could apply noise that is not Gaussian. We recall equation \eqref{eq:Tepsfree}, which provides a perturbation expansion in~$\eps$ \cite{Froyland2015} if the dynamics~$\Phi$ is a volume-preserving diffeomorphism:
\begin{equation}
\T^*\T f(x) = f(x) + \frac{\eps}{2}\P^*\Delta\P f(x) + o(\eps)\,,
\label{eq:T*Texpansion}
\end{equation}
for~$f\in C^3(\set{M})$.
Equation~\eqref{eq:T*Texpansion} also holds if the noise is not Gaussian, but has mean zero and covariance matrix~$I$. This result allows to extract coherent sets from the eigenfunctions of the dynamic Laplacian, which is an~$\eps$-free operator.

In Remark~\ref{rem:dynLap}, we extended this result to the non-volume-preserving case, where the Laplace operator has to be replaced by the~$q$-Laplacian~$\Delta_q = q^{-1}\nabla\cdot(q\nabla)$. Can one obtain a \emph{purely data-based} $\eps$-free construction that mimics the dynamic Laplacian? Remark~\ref{rem:dynLap} readily suggests to take~$\lim_{\eps\to 0}\frac{1}{\eps}(\mat{Q}_{\eps}-\mat{I})$. Note, however, that~$\eps\mapsto k_{\eps}(x,y)$ is for~$x\neq y$ an infinitely smooth function, with its derivatives \emph{of any order} being zero at~$\eps=0$. This renders the~$\eps$-derivatives of~$\mat{Q}_{\eps}$ of any order also zero. This holds true if the kernel base function~$h$ in~\eqref{eq:keps} is replaced by any compactly supported function.

It seems like something went wrong here. It turns out, we cannot exchange the limits~$m\to\infty$ and~$\eps\to 0$. Lemma~\ref{lemma:fbdiffmaps}~(ii) already indicates this: the~$\mathcal{O}(\eps^{-d/4}m^{-1/2})$ estimate diverges as~$\eps\to 0$ for finite~$m$. One can mimic~\eqref{eq:T*Texpansion} by finite data, but 
statements like~$\mat{Q}_{\eps} = \mat{I} + \eps \mat{L}_{\eps} + \mathcal{O}(\eps^2)$ only hold for~$\eps>\eps(m)$, where~$\eps(m)\to 0$ as~$m\to\infty$.

So far, we chose the exponential kernel base function~$h(x) = c_r\exp(-x) \mathbf{1}_{x\leq r}$ for our computations, since it gives rise to explicit diffusion operators~$P_{\eps,\alpha}$ via Lemma~\ref{lemma:diffusion_finite_eps}. It turns out this choice is not necessary for~(\ref{eq:Dmaps_eps_convergence}), and thus~(\ref{eq:dynLap}), to hold. In Ref.~\cite{Hein2005}, it was shown that~(\ref{eq:Dmaps_eps_convergence}) can be established with virtually any kernel base function~$h: \R^+ \rightarrow \R$, it merely has to satisfy some mild conditions, including sufficient smoothness, boundedness, and having a compact support.
Observe now that the choice~$h(x) = x^{-a/2}$ with $a>0$ would lead to $k_\eps(x,y) = \eps^{a/2}\|x-y\|^{-a}$, and when we compute $\mathbf{P}_{\eps,\alpha}$ with this kernel the $\eps^{a/2}$ factor cancels due to the row normalization in (\ref{eq:diffmaps}), apparently leading to $\mathbf{P}_{\eps,\alpha}$ being $\eps$-independent. But $h(x) = x^{-a/2}$ is neither bounded nor compact, hence we must introduce cutoff and saturation values, i.e.~$h(x) = \min\{h_{\rm max}, x^{-a/2}\mathbf{1}_{x\leq r}\}$ (actually, a mollified version of this, such that~$h\in C^2(\R^+)$). This reintroduces the $\eps$-dependence of $\mathbf{P}_{\eps,\alpha}$ via
\[
\mat{P}_{\eps,\alpha}(i,j) \neq 0\quad\Leftrightarrow\quad \|x^i-x^j\|^2 \le \eps r\,.
\]
which essentially means that we trade in the ``noise variance'' parameter~$\eps$ for a ``proximity'' parameter~$\eps r$. Due to computational reasons, one introduces such a cutoff parameter in practice anyway, since a large data set would render manipulation with fully occupied matrices impossible. But the message here is that such a parameter is actually necessary for mathematical reasons: in order for $\eps^{-1}(\mathbf{P}_{\eps,\alpha} - \mathbf{I})$ to converge to the scaled Laplace operator, as in~\eqref{eq:Dmaps_eps_convergence}, the proximity parameter must be scaled to~$0$ as~$m\rightarrow \infty$, such that~$m\eps^{d+4}/\log m\to\infty$, cf~Theorem~3 in Ref.~\cite{Hein2005}.

\begin{remark} \label{rem:trajmeths}
In contrast to our approach (\emph{first} compute a graph Laplacian for every time slice, \emph{then} perform temporal averaging), the approach in Ref.~\cite{HaEtAl15} computes a \emph{dynamical distance}~$r_{ij}$ by performing a temporal average of the Euclidean distance first. Using the dynamical distance and the weights~$r_{ij}^{-1}$, they construct a single graph Laplacian and perform spectral clustering with it. These weights would correspond to a kernel base function\footnote{To achieve boundedness at the origin, the authors in Ref.~\cite{HaEtAl15} set the diagonal terms to~$r_{ii}^{-1} = K$, where~$K\gg 1$ is some large constant.}~$h(x) = x^{-1/2}$ in our case (note that we insert the squared Euclidean distances into~$h$). Their method uses a cutoff parameter as well. The main difference to our method is that they perform time-averaging before setting up the graph Laplacian, and these two operations do not commute.
Froyland and Padberg-Gehle~\cite{FrPa15} use a fuzzy c-means clustering method on the dynamical distances directly. Hereby, their dynamical distances are squared Euclidean distances between the trajectories embedded into the high-dimensional space~$\R^{dT}$, where~$d$ is the data dimension and~$T$ is the number of time slices.
\end{remark}

\section{Numerical examples} \label{sec:numerics}

\subsection{Double gyre} \label{ssec:doublegyre}

We consider the non-autonomous system~\cite{FrPa14}
\begin{equation}
\begin{aligned}
\dot x &= -\pi A \sin\left(\pi f(t,x)\right) \cos(\pi y) \\
\dot y &= \pi A \cos\left(\pi f(t,x)\right) \sin(\pi y)\frac{df}{dx}(t,x),
\end{aligned}
\label{eq:DoubleGyre}
\end{equation}
where $f(t,x)=\alpha \sin(\omega t)x^2+(1-2\alpha\sin(\omega t))x$. We fix the parameter values $A=0.25$, $\alpha=0.25$ and~$\omega=2\pi$. The system preserves the Lebesgue measure on~$\set{X} = [0,2]\times [0,1]$.
Equation~\eqref{eq:DoubleGyre} describes two counter-rotating gyres next to each other (the left one rotates clockwise), with the vertical boundary between the gyres oscillating periodically. The period of revolution of the gyres varies with the distance from the ``center'', and is, on average, about 5 time units.

First, we consider a data-rich case. We simulate 20000 trajectories, with initial states from a~$200\times 100$ grid of~$\set{X}$, with position information obtained every~$0.1$ time instances from initial time~$0$ to final time~20. Thus~$d=2$,~$m=20000$, and~$T=201$.

We construct the space-time diffusion matrix~$\tilde{\mat{Q}}_{\eps}$ for various values of~$\eps$, and show the dominant spectrum of~$\mat{L}_{\eps} = \eps^{-1}(\tilde{\mat{Q}}_{\eps}-\mat{I})$ in Figure~\ref{fig:DG_evals}.

\begin{figure}[htb]
\centering
\includegraphics[width = 0.5\textwidth]{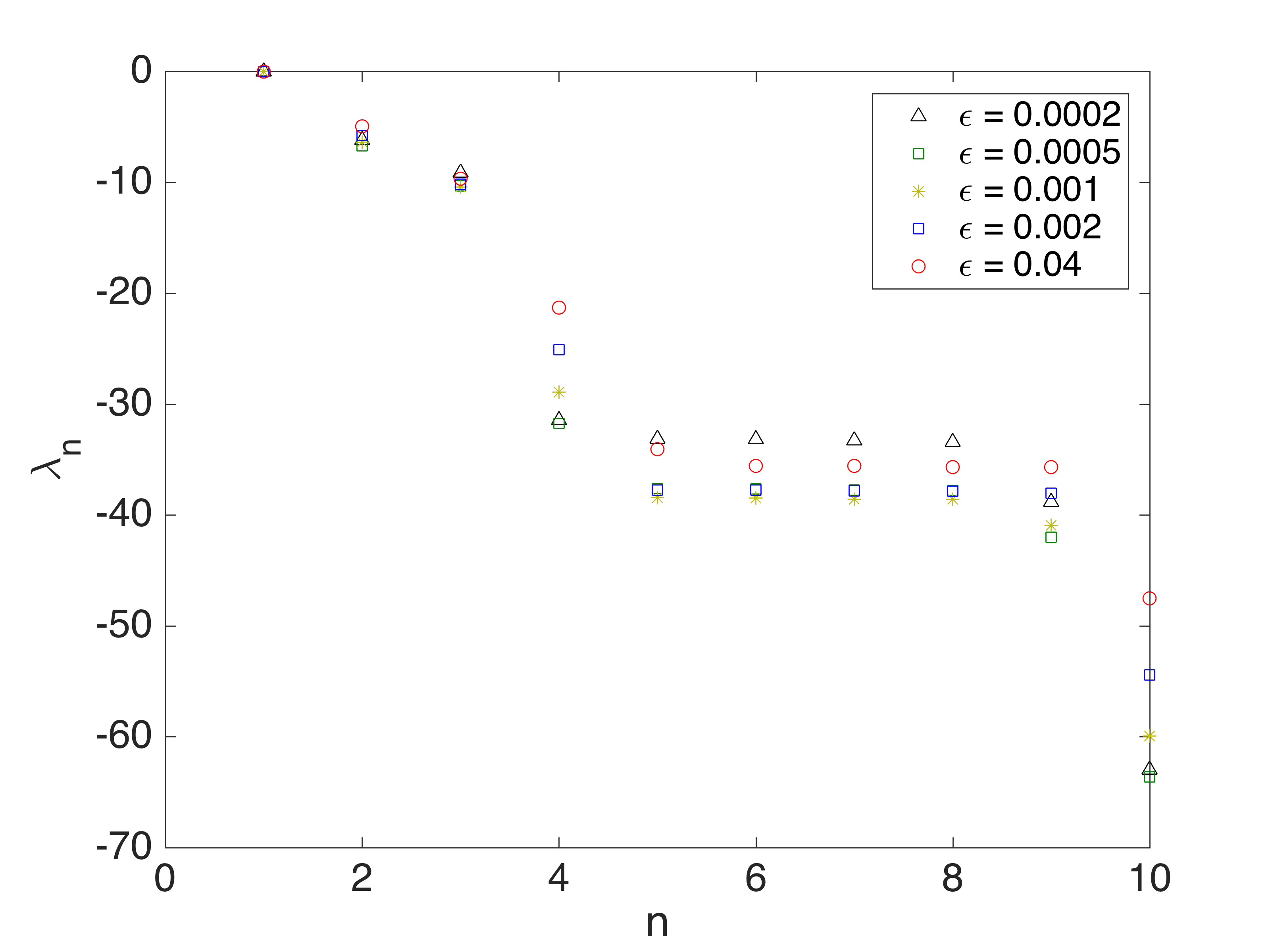}
\caption{Scaled eigenvalues of the space-time diffusion matrix.}
\label{fig:DG_evals}
\end{figure}

We can identify a gap after three eigenvalues, and expect to find~$\Lambda=3$ coherent sets. Extracting~three clusters yields for every~$\eps$ the coherent sets shown in Figure~\ref{fig:DG_3rich}.

\begin{figure}[htb]
\centering
\includegraphics[width = 0.49\textwidth]{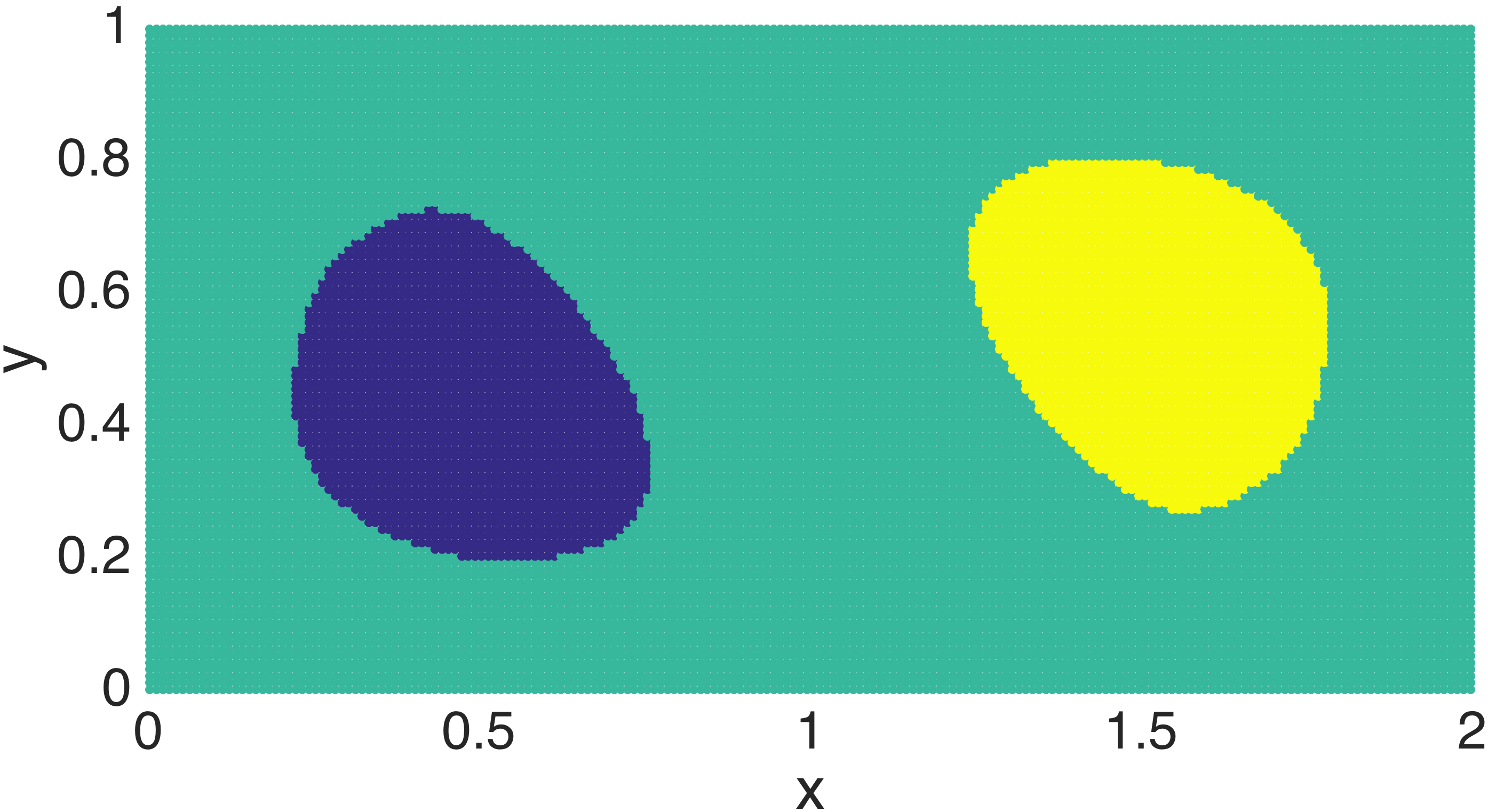}
\hfill
\includegraphics[width = 0.49\textwidth]{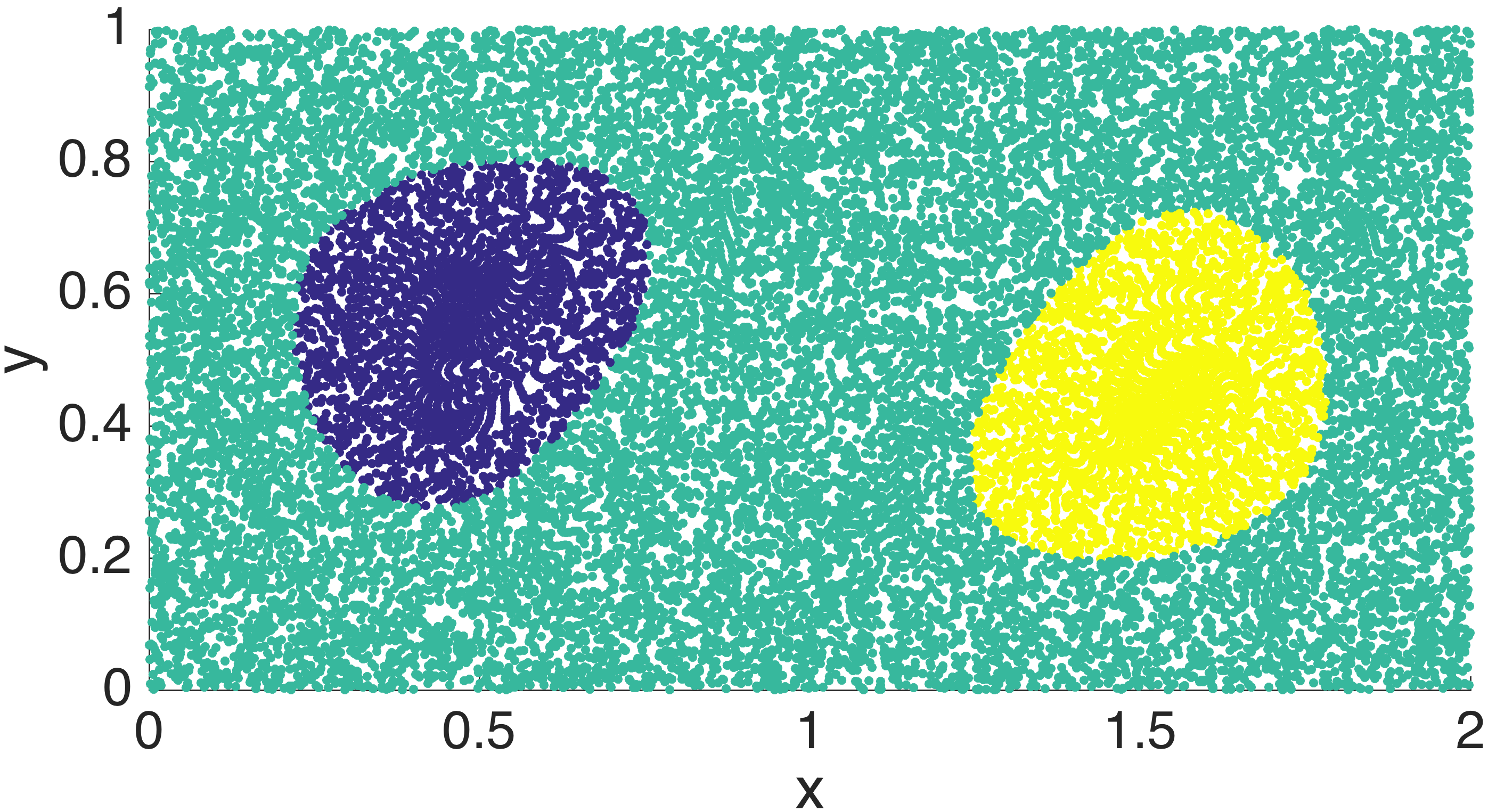}
\caption{Result of 3-clustering the double gyre trajectory data, shown at initial time ($t=0$; $1^{\rm st}$ time slice), and half a period before final time ($t=19.5$; $196^{\rm th}$ time slice). Multimedia view online.}
\label{fig:DG_3rich}
\end{figure}

We observe an interesting ``bifurcation'' in the 2-clustering of the~$2^{\rm nd}$ eigenvector~$\Xi_2$, when decreasing~$\eps$. Figure~\ref{fig:DG_2clust_diffeps} shows the eigenvectors and corresponding 2-clusterings for~$\eps=0.004$ and~$\eps = 0.0002$.

\begin{figure}[htb]
\centering
\includegraphics[width = 0.49\textwidth]{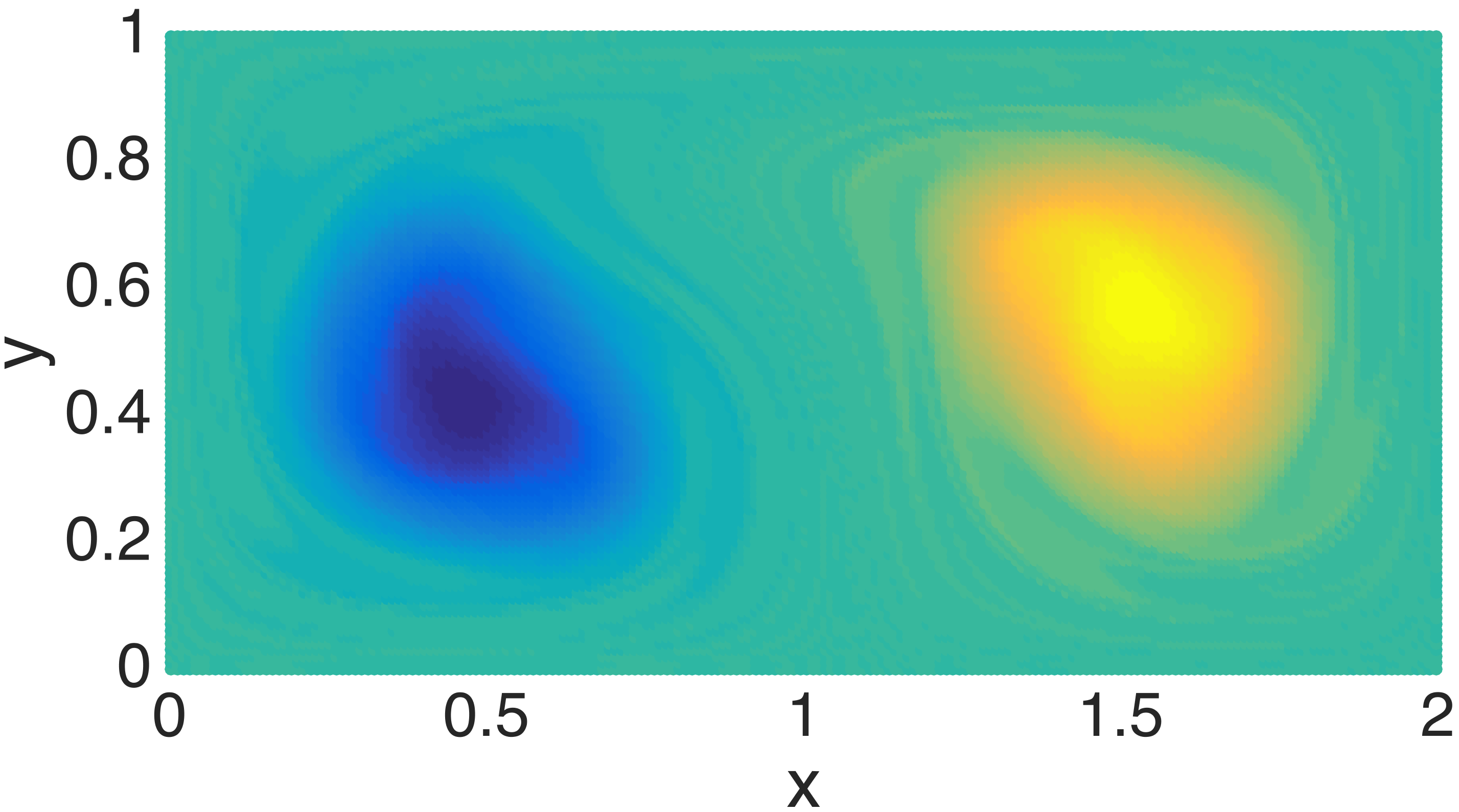}
\hfill
\includegraphics[width = 0.49\textwidth]{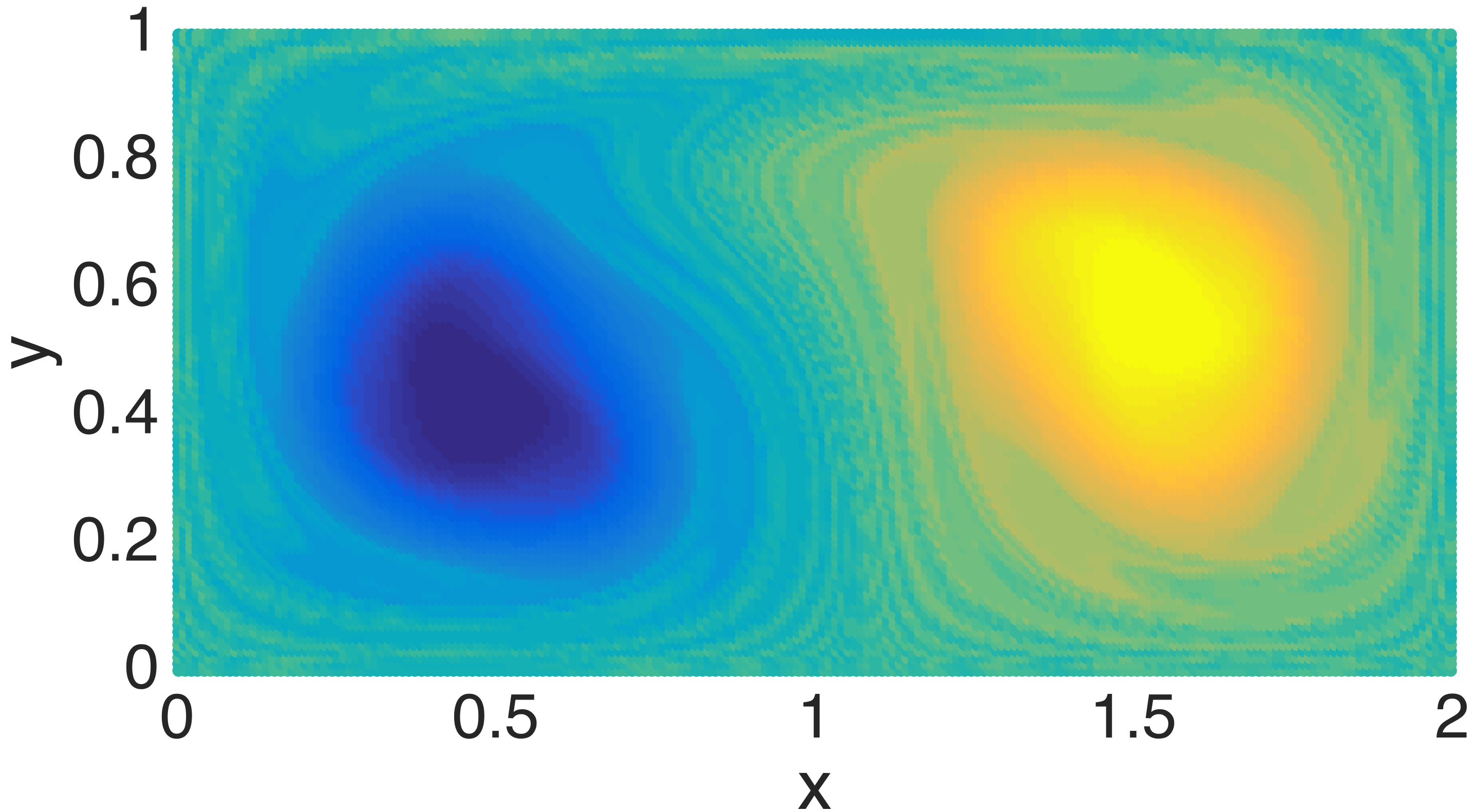}
\\
\includegraphics[width = 0.49\textwidth]{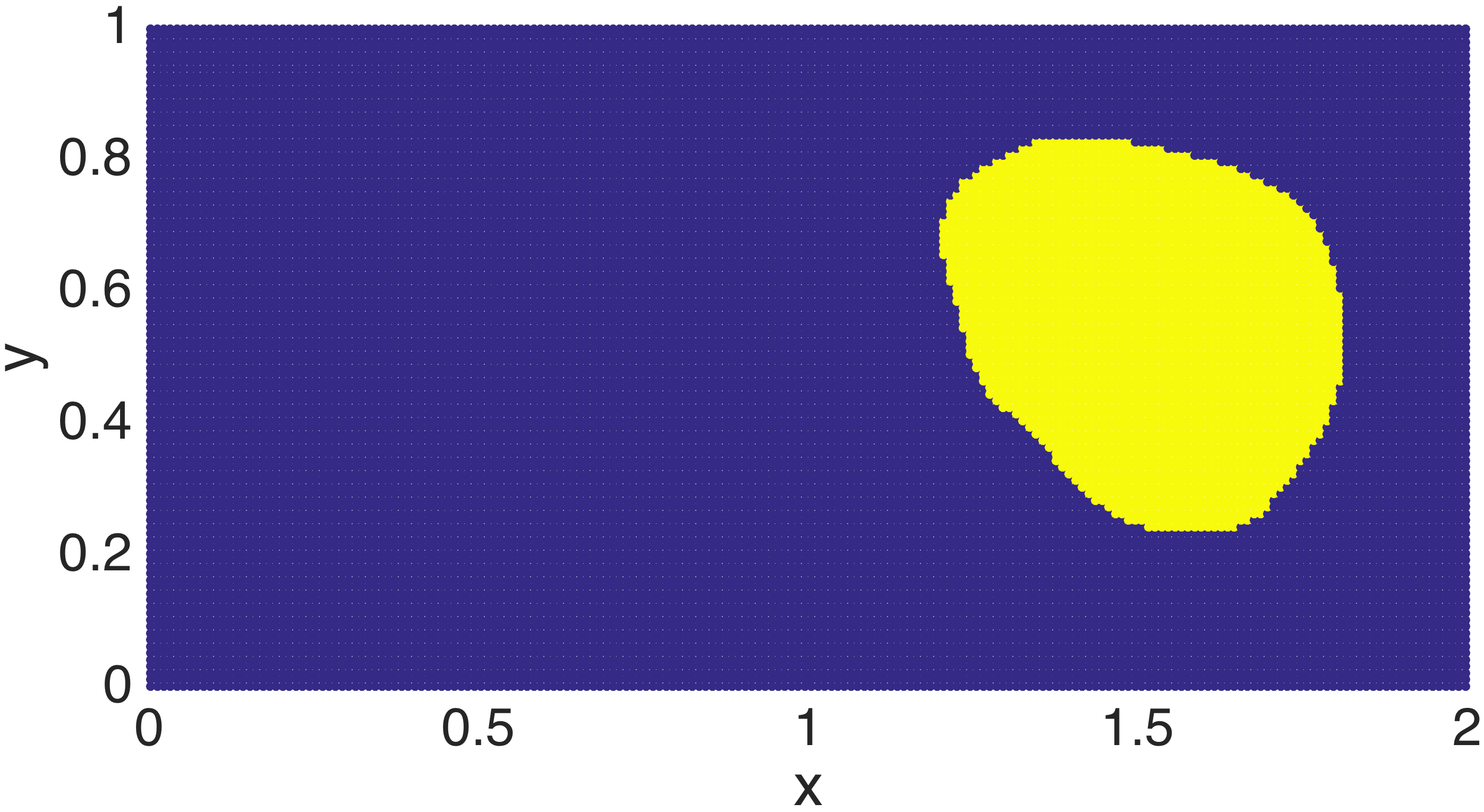}
\hfill
\includegraphics[width = 0.49\textwidth]{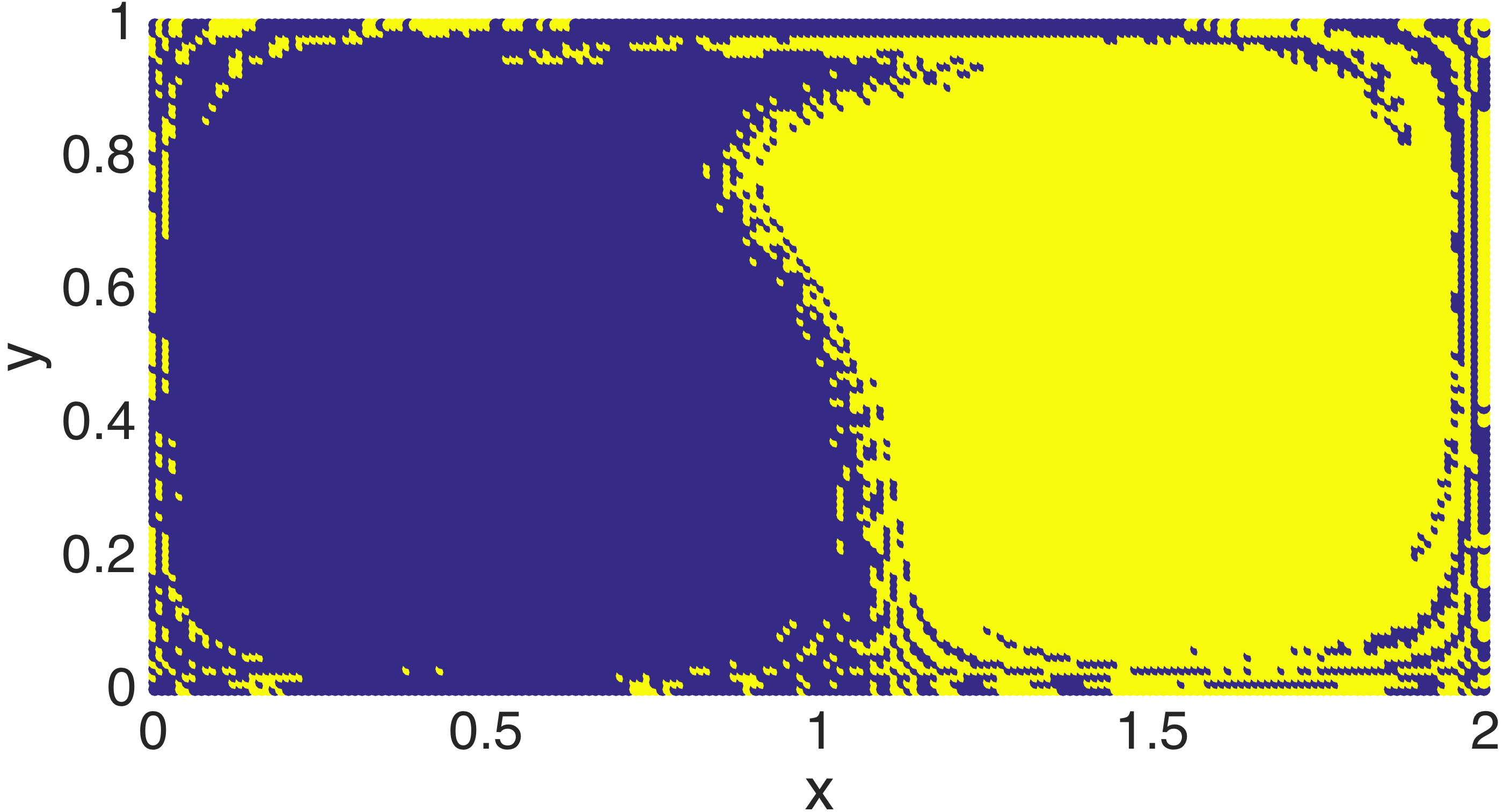}
\caption{Top: second eigenfunction~$\Xi_2$ of~$\tilde{\mat{Q}}_{\eps}$ for~$\eps = 0.0002$ (left) and~$\eps = 0.004$ (right) at initial time $t=0$. Bottom: corresponding 2-clusterings. Multimedia view online.}
\label{fig:DG_2clust_diffeps}
\end{figure}

For the smaller diffusion value, one of the gyres gets separated from the rest of phase space to yield the most coherent splitting. For the larger diffusion, however, the separation is along the stable manifold of the hyperbolic periodic orbit on the~$\{y=0\}$ boundary of~$\set{X}$. This latter case has been observed on different occasions both with transfer operator based methods and Lagrangian drifter-based techniques~\cite{FrPa14,FrPa15}. The transition between the clusters for changing~$\eps$ has been reported for almost invariant sets in~\cite[pp~27]{FrPa14}. The reason for this bifurcation is the following. If no diffusion is present, the central parts of the two gyres (from now on ``gyre cores'') are regular regions of the flow (invariant tori of the time-1 flow map, cf~\cite[Figure~1]{FrPa14}), hence they are perfectly coherent. Meanwhile, convective transport between the regions~$\{x\le 1\}$ and~$\{x\ge 1\}$ occurs along the unstable manifold of the periodic orbit on~$\{y=1\}$ (which is the image of the stable manifold of the periodic point on~$\{y=0\}$ under point reflection with respect to the point~$(1,0.5)$), which, close to~$\{y=0\}$, meanders back and forth between the two regions. Thus, if convective transport dominates diffusion, the gyre cores are the most coherent sets. However, if we increase diffusion, trajectories can leave the gyre cores. Note that there is also diffusive transport across the separatrix~$\{x=1\}$, but it is less than transport across the gyre core boundaries, because these latter boundaries are longer than the separatrix (this we can see with the naked eye). Hence, if~$\eps$ is large enough that diffusive transport dominates convective transport, the ``left-right'' separation of~$\set{X}$ reveals the most coherent sets. Of course, if diffusion is that large, there is less determinism in the fate of the single trajectories. A hard clustering of the complete state space might not be sensible, and a soft clustering shall be used instead~\cite{FrPa15}. We remark, that for e.g.\ metastability analysis of molecular dynamics, where diffusion plays a decisive role, the concept of ``soft membership'' in dynamical coarse graining is well established~\cite{SchSa13}.

We note that techniques based on Galerkin-type projections of transfer operators were always reported to reveal the ``left-right'' separation as most coherent splitting~\cite{FrPa14,WRR15}. This is most likely due to the smoothing effect of the projection onto basis functions: they introduce numerical diffusion~\cite{FrPHD,FrJuKo13}, which seems to be over our ``bifurcation threshold'' at~$\eps\approx 0.004$.

While the computational cost of Galerkin projection methods decreases with decreasing number of basis functions (which, in general leads to \emph{increased} numerical diffusion), the computational effort of our method decreases with decreasing~$\eps$, since less points are~$\mathcal{O}(\sqrt{\eps})$-close to each other, and this sparsifies~$\mat{Q}_{\eps}$. While for~$\eps = 0.004$ around~$11\%$ of the entries of~$\tilde{\mat{Q}}_{\eps}$ are nonzero, for~$\eps = 0.0002$ the fraction of nonzeros is~$0.6\%$.

Looking for additional coherent sets, we cluster the eigenvector data into~$4$ clusters, shown in Figure~\ref{fig:DG_4rich}. In the large-diffusion case we find the gyre cores along with the left-right separation according to the stable manifold as coherent sets. In the small-diffusion case we find further subdivision of the regular region. The interested reader may compare this result with that in~\cite{FrKo15}.

\begin{figure}[thb]
\centering
\includegraphics[width = 0.49\textwidth]{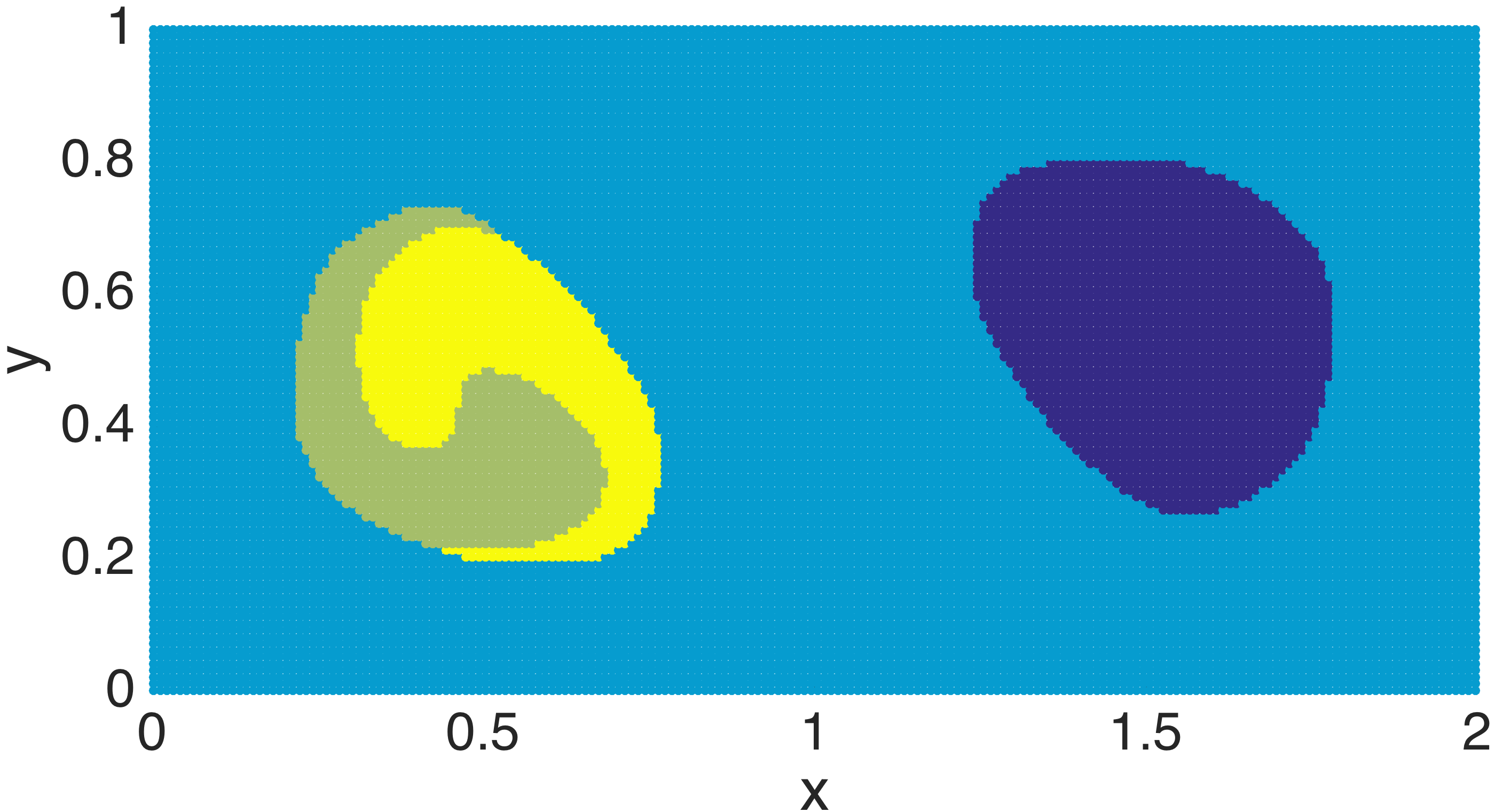}
\hfill
\includegraphics[width = 0.49\textwidth]{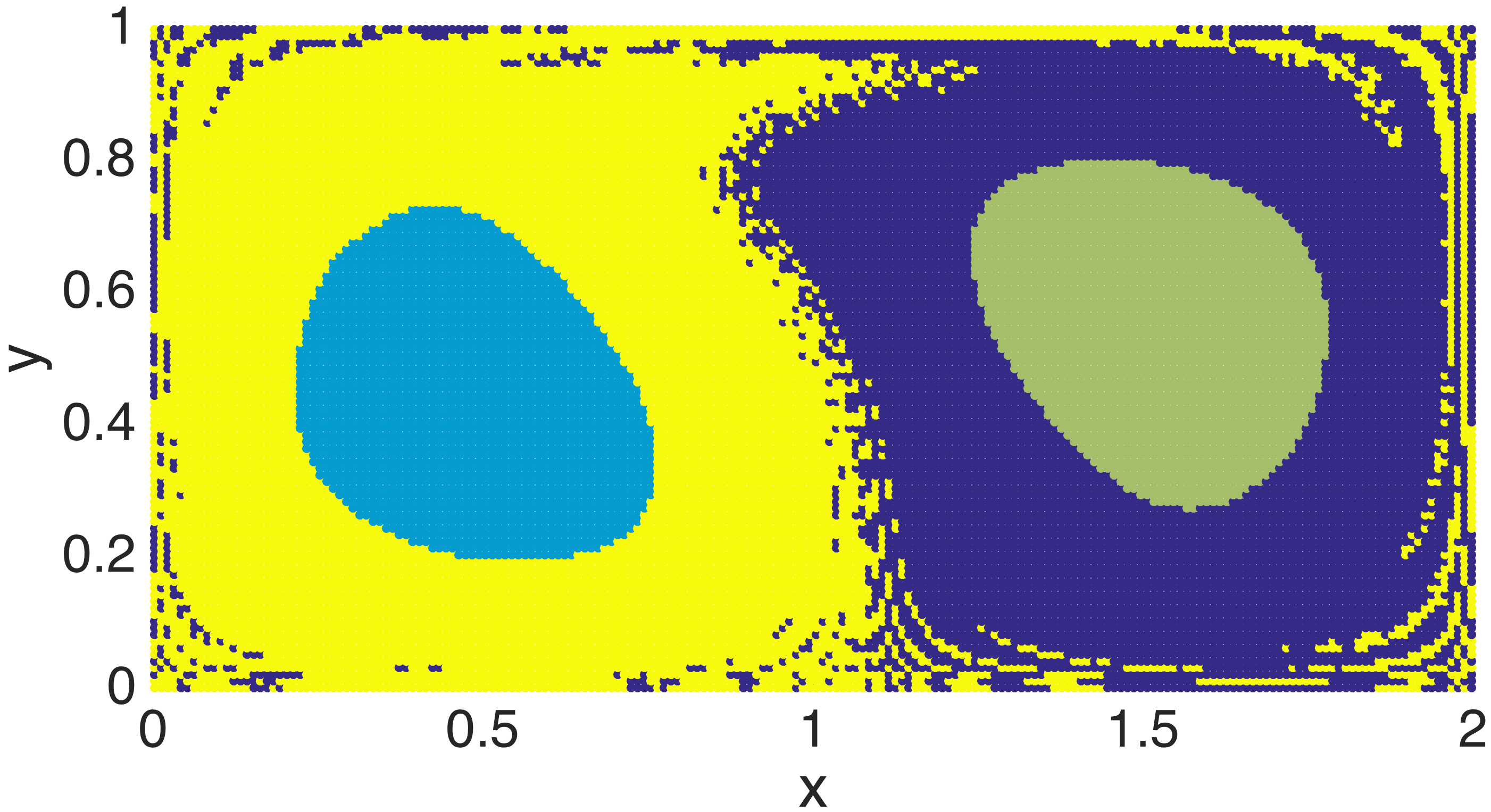}
\caption{Result of 4-clustering the double gyre trajectory data, shown at initial time ($t=0$; $1^{\rm st}$ time slice). Left: $\eps = 0.0005$, right: $\eps = 0.004$.}
\label{fig:DG_4rich}
\end{figure}

We turn now to a sparse, incomplete data case; cf also~\cite{FrPa15}. We take our previous data set, and pick~$m=500$ trajectories randomly, and discard the rest. Then, we destroy~$80\%$ of the remaining data, by setting randomly (both in time and space) entries to~\texttt{NaN} (``Not a Number'' in Matlab). To balance the sparsified neighborhoods due to the loss of data, we set~$\eps = 0.01$. Then we assemble the space-time diffusion matrix, and carry out the clustering of its eigenvectors for~2 and~3 clusters, respectively. The results are shown in Figure~\ref{fig:DG_sparse}.

\begin{figure}[htb]
\centering
\includegraphics[width = 0.49\textwidth]{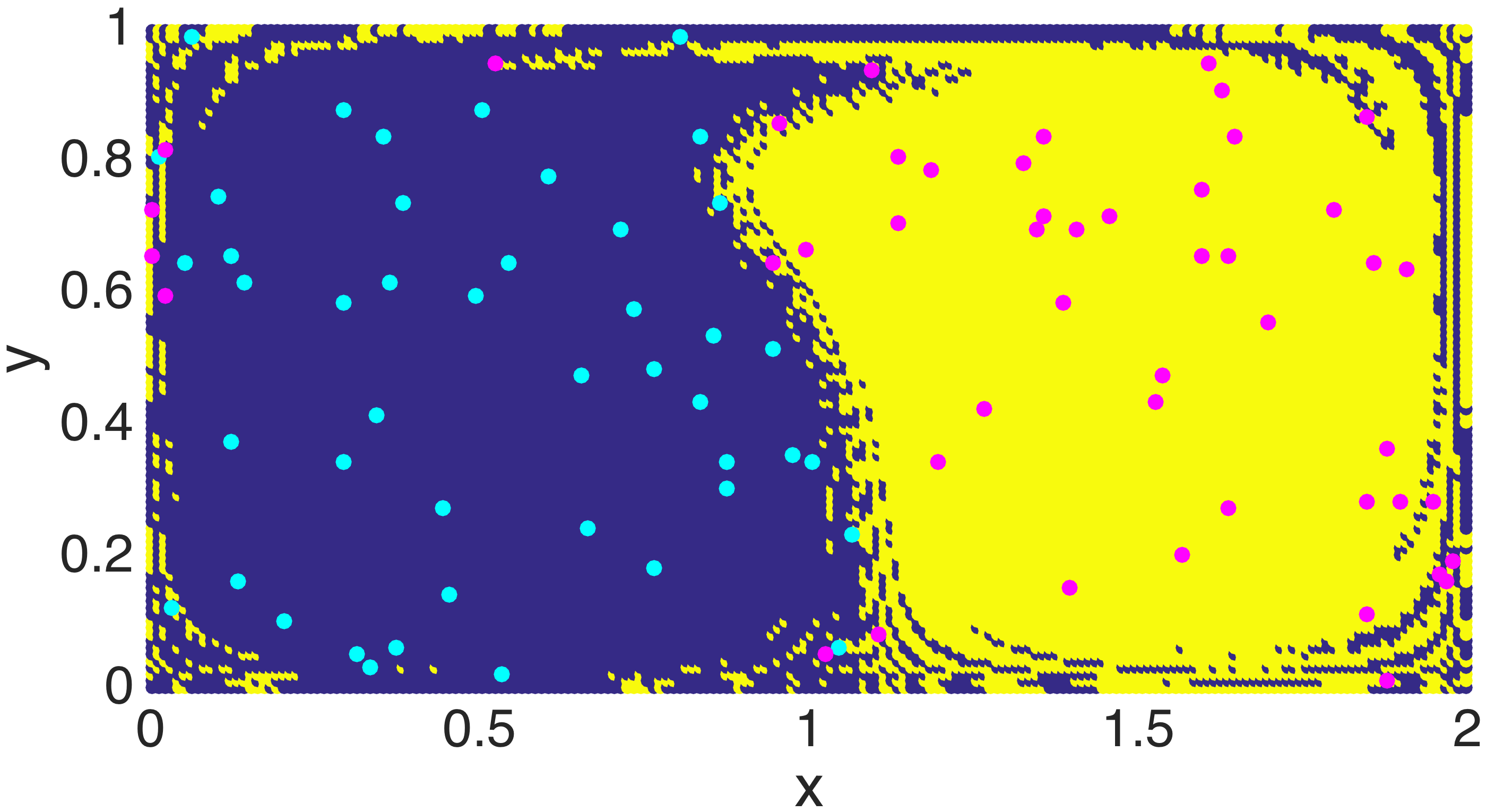}
\hfill
\includegraphics[width = 0.49\textwidth]{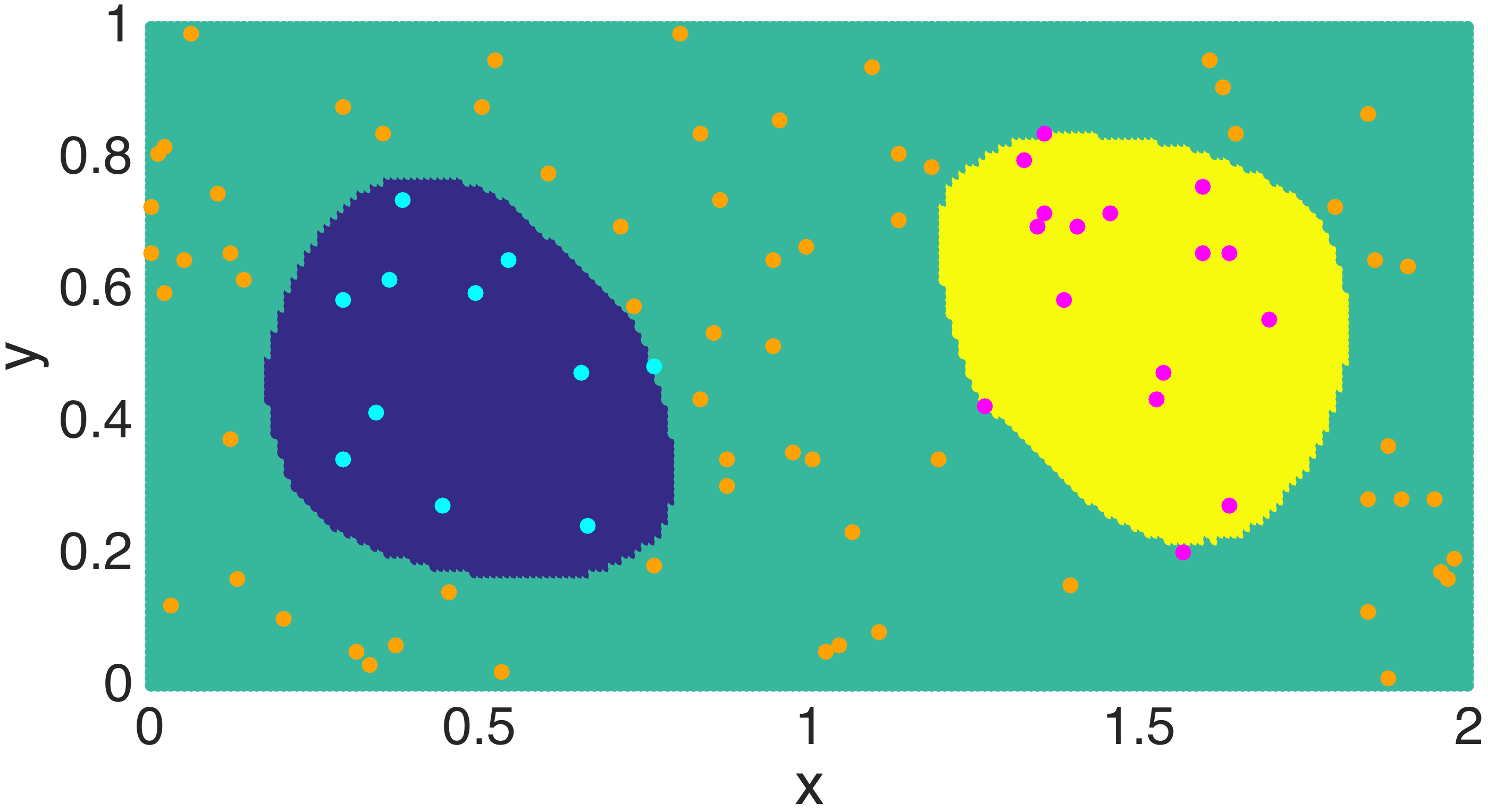}
\caption{Results for the sparse incomplete data set, compared with the results of the full data case from before. In the sparse case, $97.5\%$ of the previous trajectories is discarded, and~$80\%$ of the remaining data is destroyed. The sparse incomplete data clusters are represented by the colors cyan and magenta (2-clustering, left figure), and cyan, magenta, orange (3-clustering, right figure), respectively.}
\label{fig:DG_sparse}
\end{figure}

There, the original 2- and 3-clusterings in the large diffusion case are overlayed by the clustering of the sparse incomplete data. We observe an excellent agreement; note that in some cases even the filaments of the one cluster reaching well into the other are correctly identified.

\subsection{Bickley jet}

We consider a perturbed Bickley jet as described in~\cite{RypEtAl07}. This is an idealized zonal jet approximation in a band around a fixed latitude, assuming incompressibility, on which three traveling Rossby waves are superimposed.
The dynamics is given by $(\dot x, \dot y) = (-\frac{\partial\Psi}{\partial y}, \frac{\partial\Psi}{\partial x})$, with stream function $\Psi(t,x,y) = -U_0 L \tanh\big(y/L\big)
 + U_0L\,\mathrm{sech}^2\big(y/L\big) \sum_{n=1}^3 A_n\cos\left(k_n\left(x- c_n t\right)\right)$. The constants are chosen as in Section~4 in Ref.~\cite{RypEtAl07}, the length unit is Mm (1 Mm = $10^6$ m), the time unit is days. In particular, we set~$k_n = 2n/r_e$ with~$r_e = 6.371$, $U_0 = 5.414$, and~$L = 1.77$. The phase speeds~$c_n$ of the Rossby waves are~$c_1 = 0.1446U_0$, $c_2 = 0.2053U_0$, $c_3 = 0.4561U_0$, their amplitudes~$A_1 = 0.0075$, $A_2 = 0.4$, and~$A_3 = 0.3$. The system is usually considered on a state space which is periodic in the~$x$ coordinate with period~$\pi r_e$; we will, however, not make any use of this knowledge in our computations.

We advect $m= 12000$ particles with initial conditions at $t_0=0$ on a uniform grid inside the domain $[0,20]\times [-3,3]$. We save the positions of the particles at the $T=401$ time frames $I_t = \{0, 0.1,\ldots, 40\}$, during which each particle traverses the cylinder~$\sim 5$ times. With this data as input, we compute $\mathbf{Q}_\eps$ according to \eqref{eq:sptdm_matrix}. The dominant spectrum of $\mathbf{L}_\eps = \eps^{-1}(\mathbf{Q}_\eps - \mathbf{I})$ for different values of $\eps$ is shown in Figure \ref{fig:bickley_evals} on the left. The $\lambda_n$ for $n\leq 9$ are stable for $0.01\leq \eps \leq 0.05$. We choose~$\eps = 0.02$, yielding a sparsity of $4.5\%$. The unifying features of the spectra are large spectral gaps after the 2nd, 3rd and 9th eigenvalue, which indicates that clusterings with $\Lambda=2,3$ or $9$ are all possible. The eigenfunctions $\Psi_2, \Psi_3$ and $\Psi_4$ are shown in Figure \ref{fig:bickley_sets} on the right at time~$t=20$. Clearly,~$\Psi_2$ and~$\Psi_3$ pick out the meandering jet stream region in the middle, which constitutes the strongest dynamical boundary in this system, and the six vortices.~$\Psi_4$ distinguishes between two of the six vortices,~$\{\Xi_5,\ldots, \Xi_9\}$ distinguish between the others.

\begin{figure}[h]
\centering
\includegraphics[width = 0.49\textwidth]{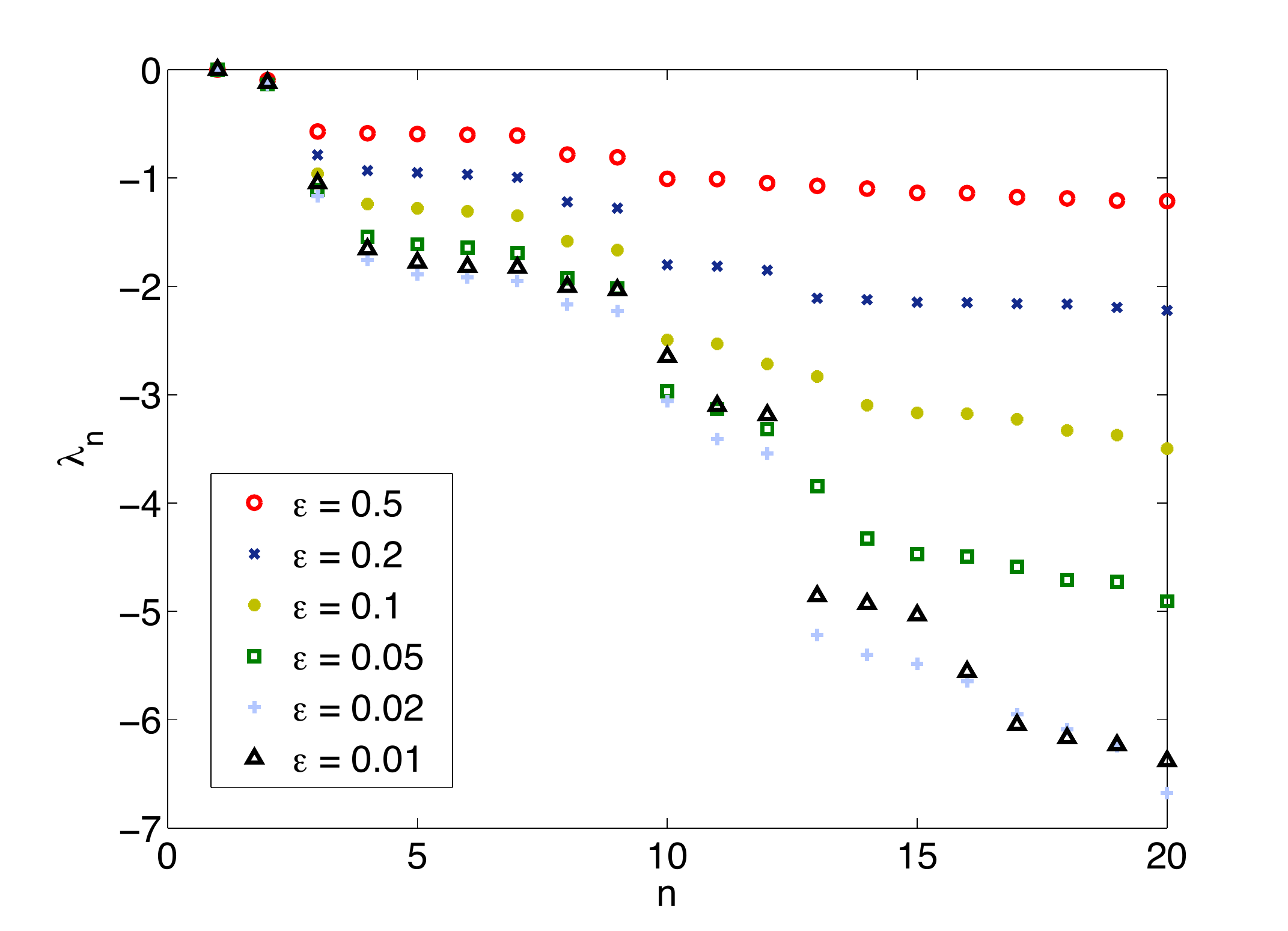}
\hfill
\includegraphics[width = 0.49\textwidth]{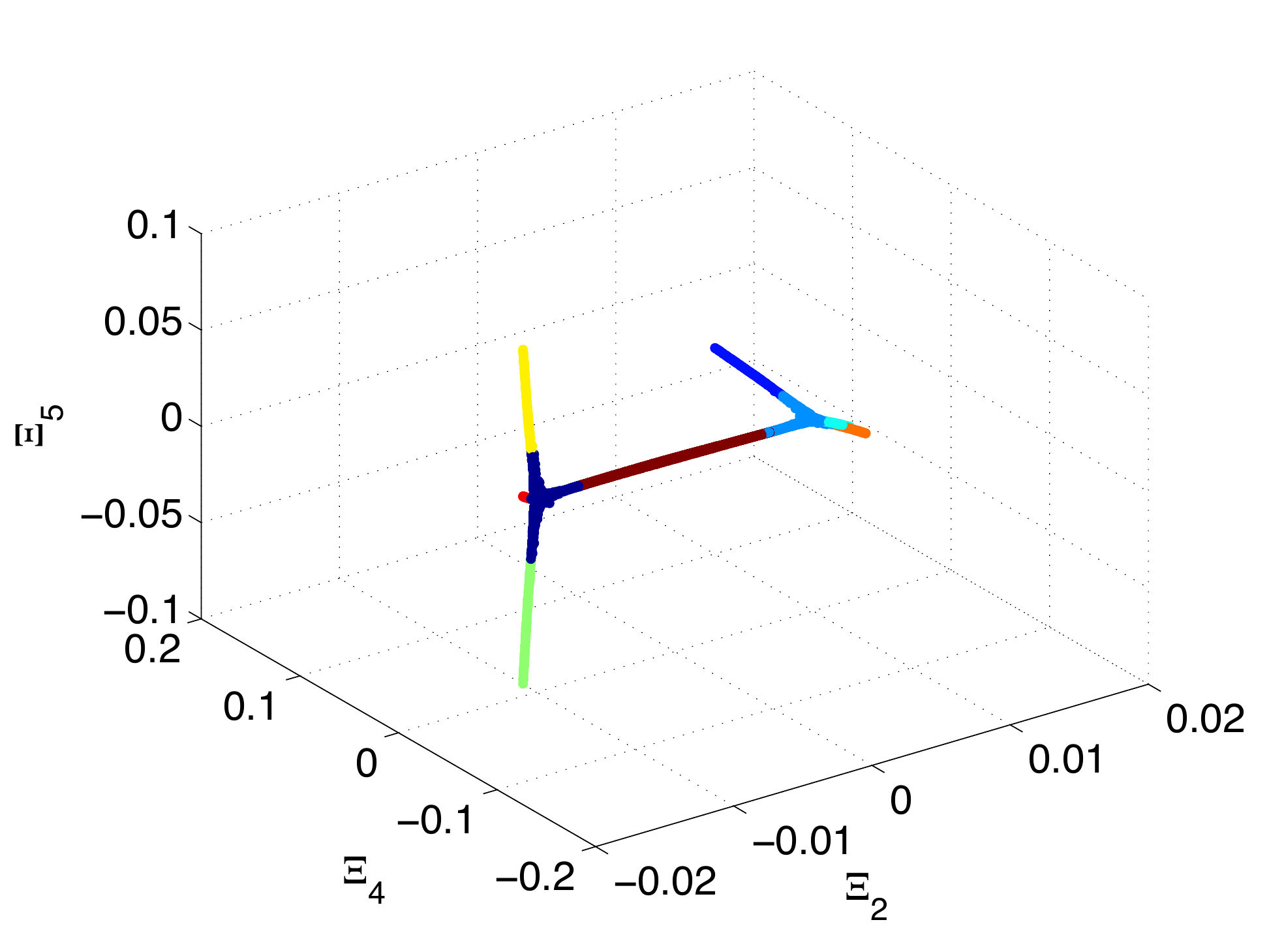}
\caption{Bickley jet, eigenvalues (left) and embedding using the eigenfunctions $\Xi_2$, $\Xi_4$ and $\Xi_5$ (right).}
\label{fig:bickley_evals}
\end{figure}


\begin{figure}[htb]
\centering
\includegraphics[width = 0.52\textwidth]{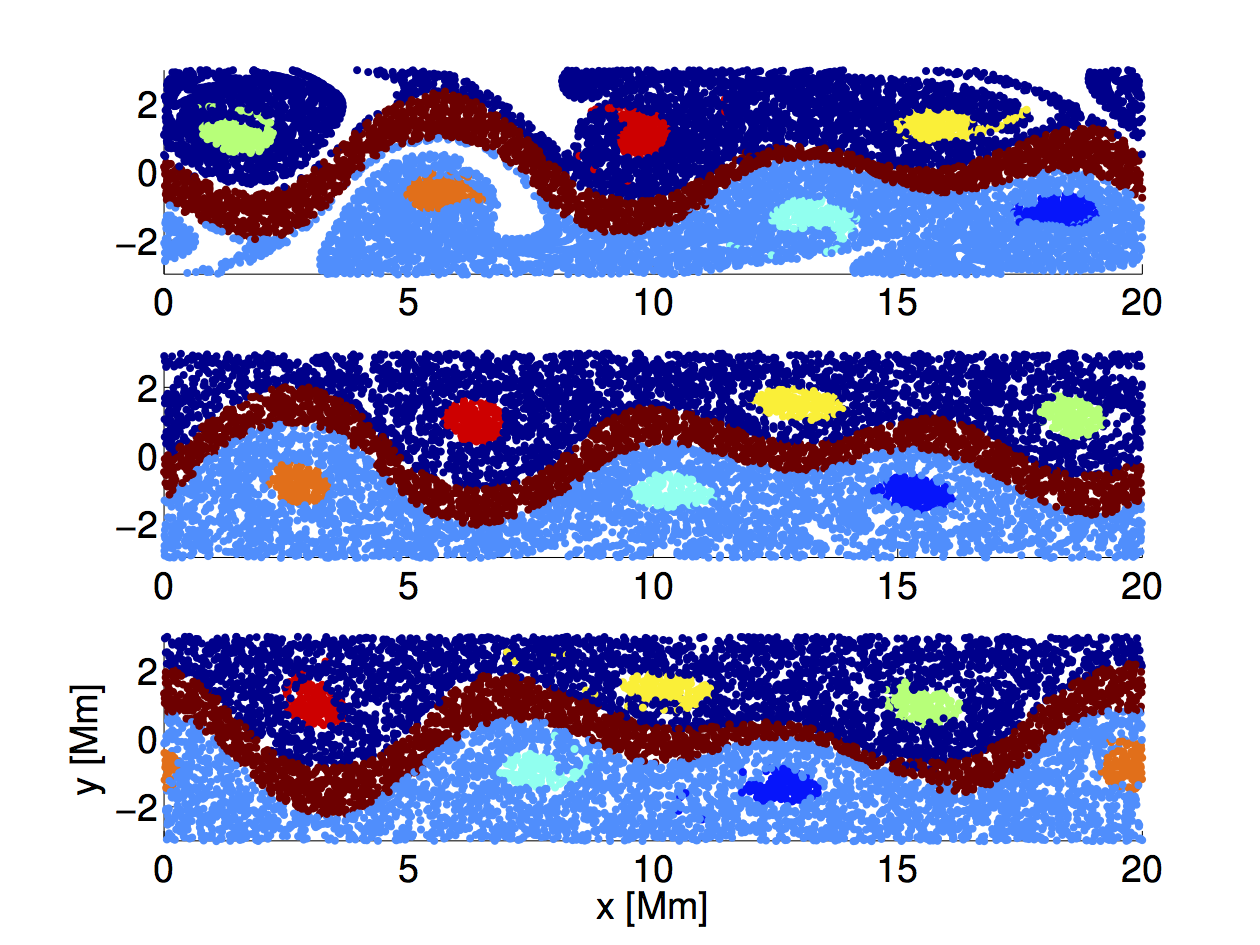}
\hfill
\includegraphics[width = 0.46\textwidth]{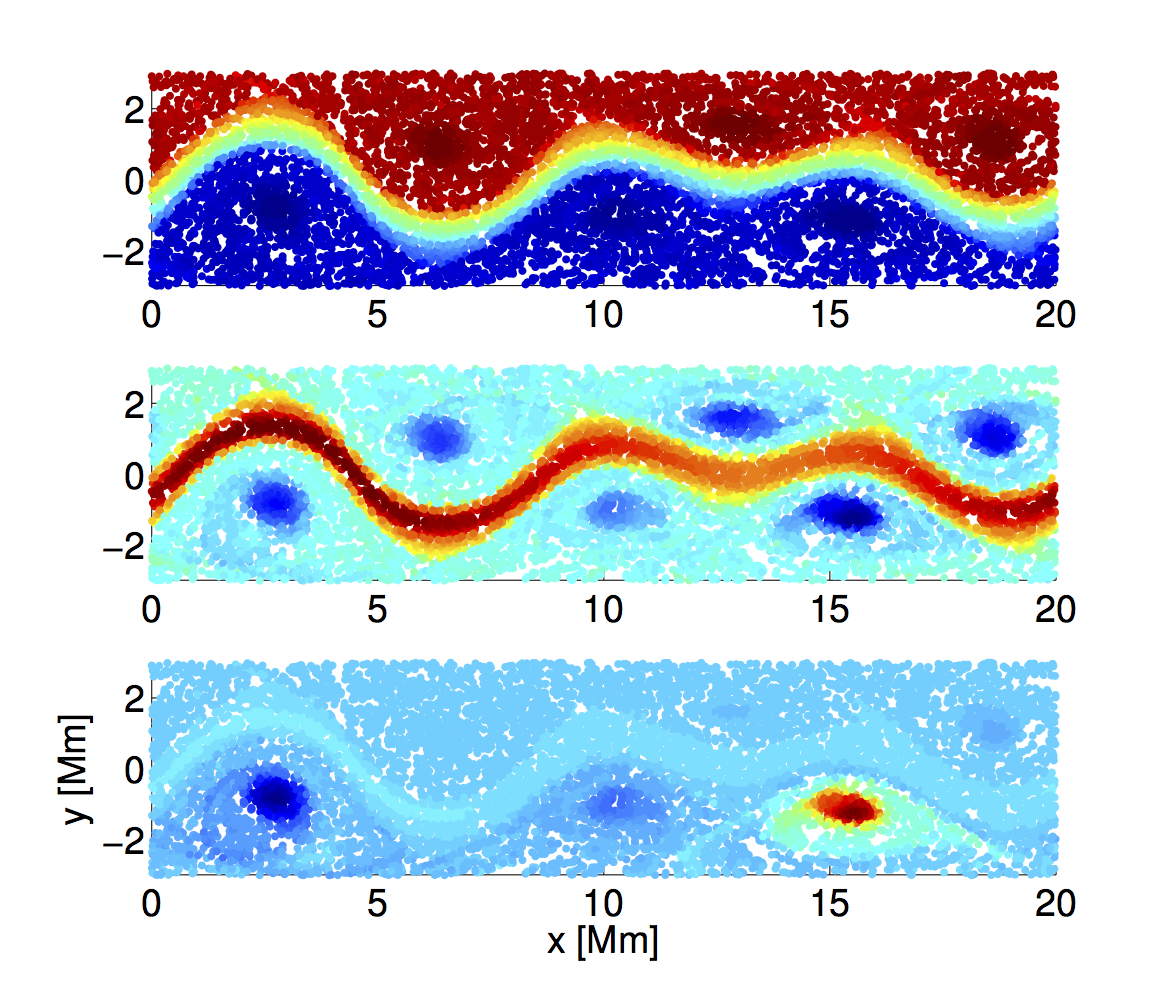}
\caption{Left, top to bottom: Bickley jet, clusters at times $t=5$, $t=20$ and $t=35$, for~$\Lambda = 9$ (multimedia view online). Right: top to bottom: Eigenfunctions~$\Xi_2,~\Xi_3$ and~$\Xi_4$ at~$t=20$.}
\label{fig:bickley_sets}
\end{figure}

The clustering for $\Lambda = 9$ is shown in Figure \ref{fig:bickley_sets} on the left at times $t=5$, $t=20$ and $t=35$. The long and narrow cluster in the jet stream region stays perfectly coherent for the whole time interval, while the six clusters in the vortex region loose some mass. This is in perfect agreement with the eigenvalue structure in Figure \ref{fig:bickley_evals}. For~$\Lambda = 3$, the six clusters in the vortex region merge with the corresponding background cluster (not shown). A movie showing the full time evolution can be found in the supporting information. In Figure~\ref{fig:bickley_evals} on the right, the~$m=12000$ trajectories are embedded as points in~$\mbox{span}\{\Xi_2, \Xi_4, \Xi_5\}$ and coloured according to the clustering in Figure~\ref{fig:bickley_sets}. This embedding highlights the connectivity structure of the clusters.

Note that we only use the Euclidean metric as input; no information about the global cylindrical geometry of the state space is given. The fact that our method extracts the jet stream region clearly shows that it learns the cylindrical geometry and highlights dynamical features that are encoded in the time-ordering of the data.
A purely geometrical, heuristic method based on the Euclidean metric alone will always struggle to identify the long, narrow and meandering clusters that we find in the data.

\subsection{Ocean drifter data set}

To test our method on real world data, we consider a dataset of ocean drifters from the Global Ocean Drifter Program available at AOML/NOAA Drifter Data Assembly Center (\url{http://www.aoml.noaa.gov/envids/gld/}). We focus on the years 2005-2009 and restrict to those drifters that have a minimum lifespan of one year within this timespan. We record the position of these 2267 drifters every month, i.e.\ our trajectories have~60 time frames. This is the same dataset which has been studied in~\cite{FrPa15}.

The drifter data is sparse: The average lifetime of a drifter is only~23 months, and there are also gaps in observations where a drifter location failed to be recorded. On average, only~$38\%$ of the drifters are available at any given time instant. The dataset is also extremely sparse spatially, with only~2267 drifters covering the global ocean, it serves therefore as a good test case for our method. Additionally, we do not use any metric that is adapted to the sphere. We simply consider the drifters as data points in~$\R^3$ and take the Euclidean metric in~$\R^3$. Hereby we scale all distances such that radius of the Earth is equal to one.

\begin{figure}[h]
\centering
\includegraphics[width = 0.49\textwidth]{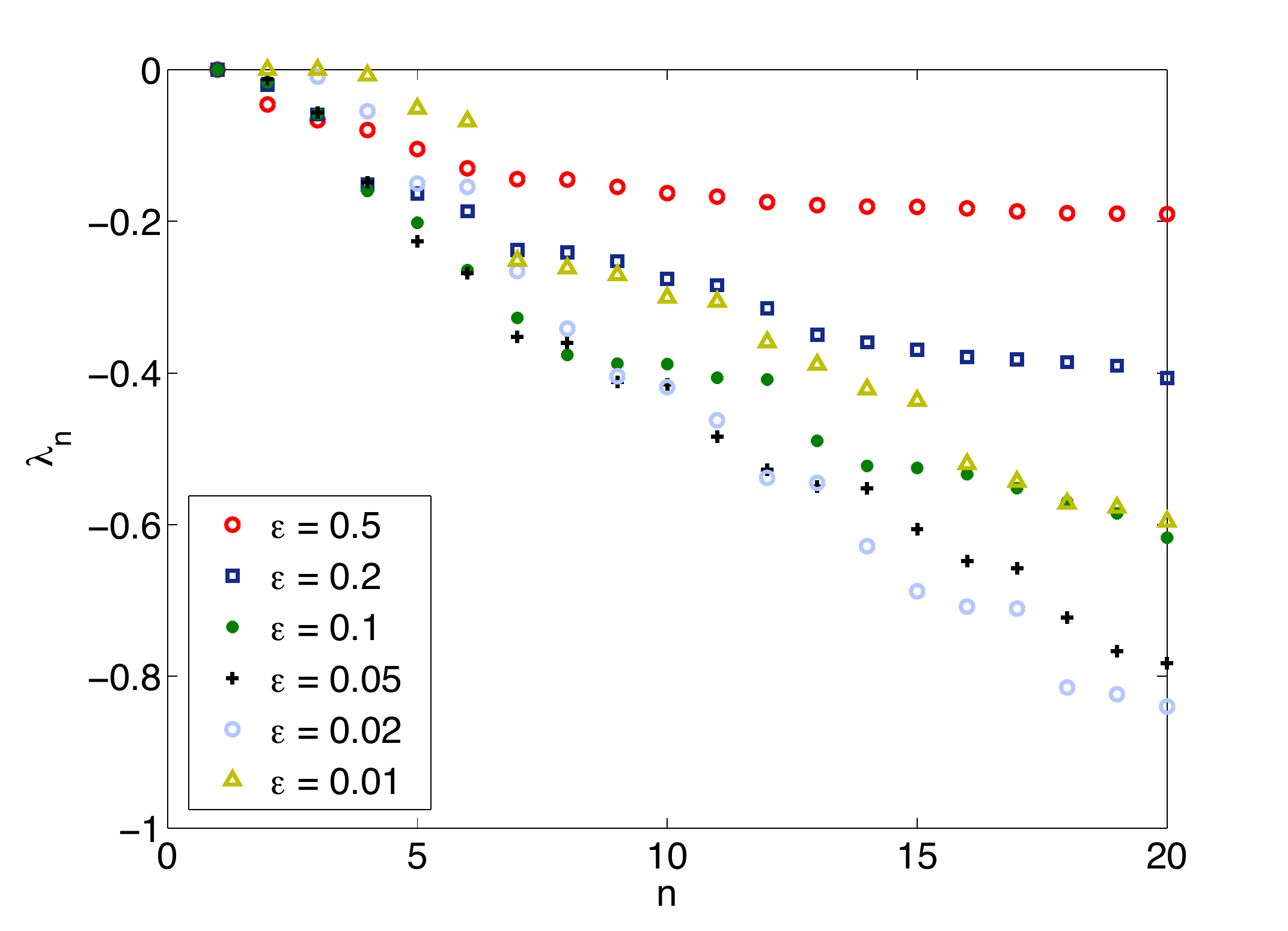}
\hfill
\includegraphics[width = 0.49\textwidth]{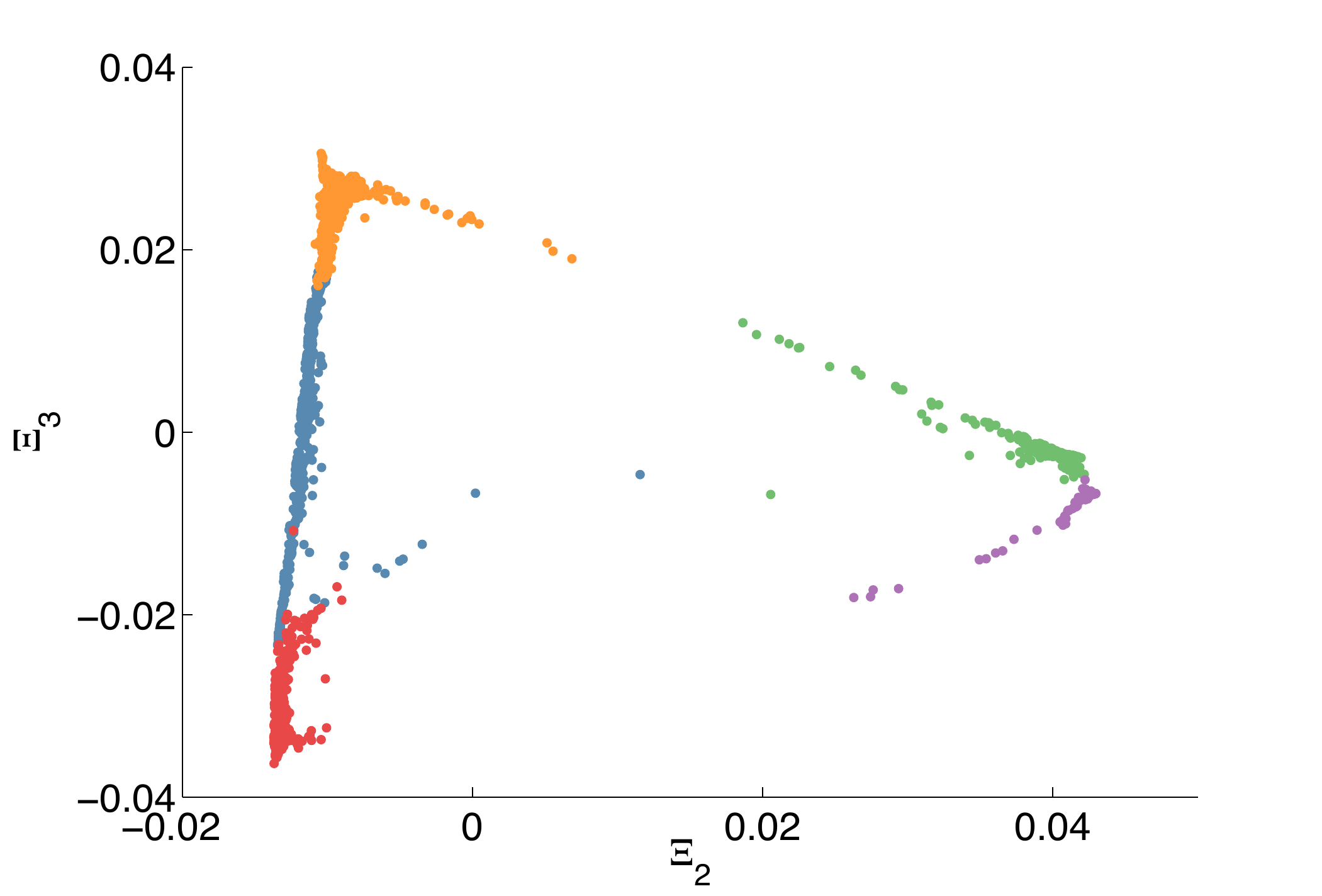}
\caption{Ocean drifter data. Left: Eigenvalues for different $\eps$. Right: Embedding using the eigenfunctions $\Xi_2$ and $\Xi_3$ for $\eps = 0.1$ with coloring according to the clusters Figure \ref{fig:drifter_sets} (red: Northern Pacific, blue: Southern Pacific, Yellow: Southern Atlantic/Indian Ocean, green: Northern Atlantic, purple: Arctic).}
\label{fig:drifter_evals}
\end{figure}

To set~$\eps$, we compute~$\mathbf{Q}_\eps$ for a range of values for $\eps$, the result is shown in Figure~\ref{fig:drifter_evals} on the left. Because the data is so sparse, the spectra show some variation with changing~$\eps$. For~$0.05 \leq \eps \leq 0.2$ they are reasonably close, indicating an optimal balance between the variance and bias terms in \eqref{eq:dynLap}. We choose~$\eps = 0.1$, which leads to a sparsity of~$\mathbf{Q}_\eps$ of $18\%$. There is no clear spectral gap in the data, so we choose~$\Lambda = 5$, as did the authors in~\cite{FrPa15}. The resulting clusters are shown in Figure~\ref{fig:drifter_sets}. To display as much information as possible, we divide the full time span into the four time intervals Jan 2005 -- Mar 2006, Apr 2006 -- Jun 2007, Jul 2007 -- Sep 2008 and Oct 2008 -- Dec 2009. For every time interval, we plot all drifter locations in a single plot and color-code time in each of the plots by color saturation (the darker the color, the ``later'' the drifter location). A movie showing all~60 frames can be found in the supplementary information.

\begin{figure}[th]
\centering
\includegraphics[width = 0.99\textwidth]{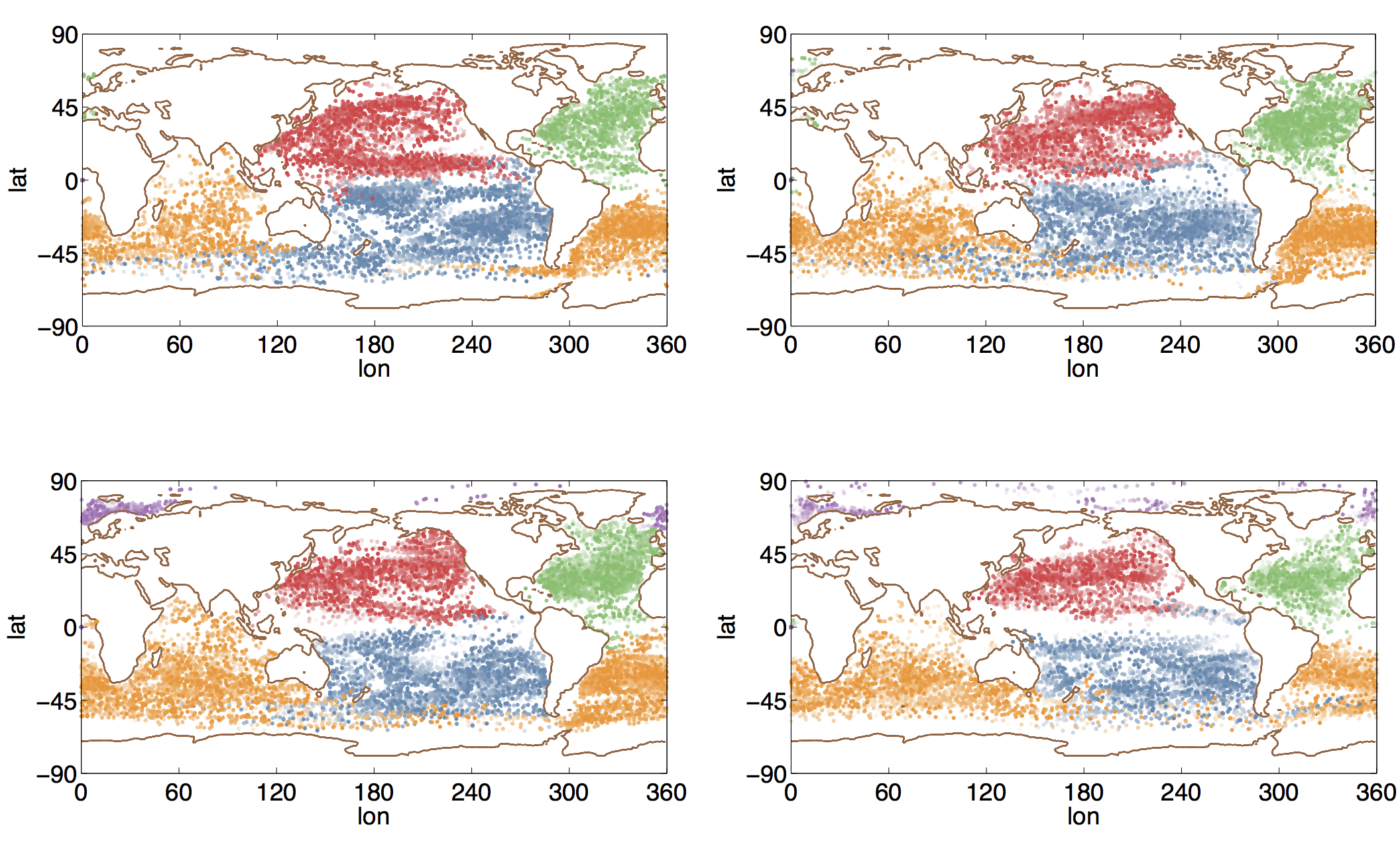}
\hfill
\caption{Ocean drifter data, clusters. Top left: Jan 2005 -- Mar 2006. Top right: Apr 2006 -- Jun 2007. Bottom left: Jul 2007 -- Sep 2008. Bottom right: Oct 2008 -- Dec 2009. Color saturation is proportional to time in the respective~$15$ month time window. Multimedia view online.}
\label{fig:drifter_sets}
\end{figure}

The five clusters we find may be described broadly as the Northern Pacific, the Southern Pacific, the Northern Atlantic, the Southern Atlantic together with the Indian Ocean, and the Arctic Ocean. Boundaries between clusters are in locations where continents and islands form bottlenecks (for example, the boundary between the green and purple cluster is a line between Great Britain and Iceland) and at the equator. In Figure \ref{fig:drifter_evals} on the right, we show the embedding of the 2267 drifters produced by~$\Xi_2$ and~$\Xi_3$. We see that~$\Xi_2$ separates the Arctic and Northern Atlantic from the rest, while~$\Xi_3$ distinguishes between the Northern Pacific, the Southern Pacific and the Southern Atlantic/Indian Ocean. We can also infer connectivity patterns from this plot. Note that there is no connection between the red, the yellow, and the purple clusters, showing that none of the drifters passed trough the Bering Strait and the Indonesian Archipelago, respectively. A few isolated data points hint at a possible connection between the blue and green clusters, this could be due to the vicinity of drifters across the Panama Strait.

The main difference to the result of Froyland and Padberg-Gehle~\cite{FrPa15} is that we do not separate the Indian Ocean from the Southern Atlantic, but instead separate the Arctic from the Northern Atlantic. A possible explanation for this is that fuzzy clustering, used by them, has a tendency to produce clusters of similar size, and although this is equally true for the~$k$-means algorithm we use, we measure size in terms of the geometry given by the diffusion coordinates~$\Xi_1, \ldots, \Xi_\Lambda$. As a result, we do produce clusters of different sizes as long as their dynamical separation is strong.

We note that the Southern Atlantic, the Southern Pacific and the Indian Ocean are dynamically well connected through the Antarctic Circumpolar Current, this can be seen by the substantial overlap between the blue and yellow clusters close to the Antarctic. As a result, drifters in this region are difficult to classify. By contrast, the Arctic is well separated from the Northern Atlantic, and the Arctic drifters are actually only available for the last 30 of the 60 months.
\subsection{The ABC-flow}

As a last, three-dimensional example, we consider the steady Arnold--Beltrami--Childress flow (short: ABC flow)~\cite{arnol1966topology}, generated by the ODE
\begin{equation*}
\begin{aligned}
\dot{x} &= A\sin(z)\, +\, C\cos(y) \\
\dot{y} &= B\sin(x)\, +\, A\cos(z) \\
\dot{z} &= C\sin(y)\, +\, B\cos(x) 
\end{aligned}
\label{eq:ABCode}
\end{equation*}
on~$\set{X} = [0,2\pi]^3$ (with periodic boundary conditions), with the ``usual'' set of parameters,~$A = \sqrt{3}$, $B=\sqrt{2}$, and~$C=1$. This autonomous system with this set of parameters yields six three-dimensional vortices, which are invariant under the dynamics~\cite{Dombre1986,FrPa09,BlHa14}.
Thus, they are also coherent sets.

The trajectory data we use consists of initial states building a~$40\times 40\times 40$ uniform grid of~$\set{X}$, integrated on a time window of length~$40$, and sampled uniformly in time every~$0.2$ time instances. Thus~$d=3$, $m=64000$, and~$T=201$.

We build the space-time diffusion map transition matrix~$\tilde{\mat{Q}}_{\eps}$ for~$\eps = 0.02$, and extract~$7$ clusters from its six subdominant eigenvectors. The spectrum of~$\tilde{\mat{Q}}_{\eps}$ does not show a clear spectral gap after six eigenvalues for any values of~$\eps$. This is because on the considered time interval, parts of the respective vortices are also coherent.

\begin{figure}[htb]
\centering
\includegraphics[width = 0.4\textwidth]{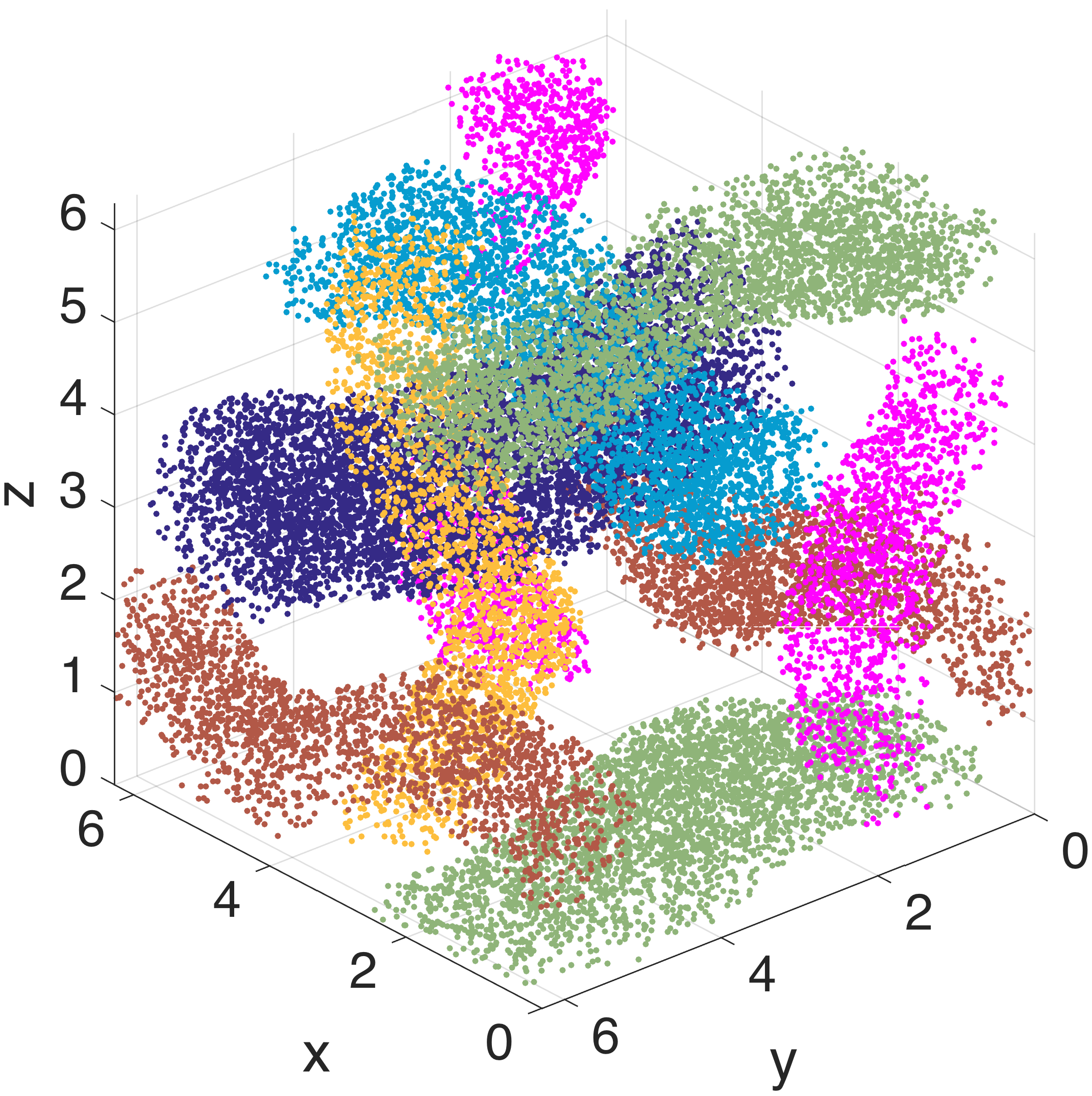}
\qquad
\includegraphics[width = 0.4\textwidth]{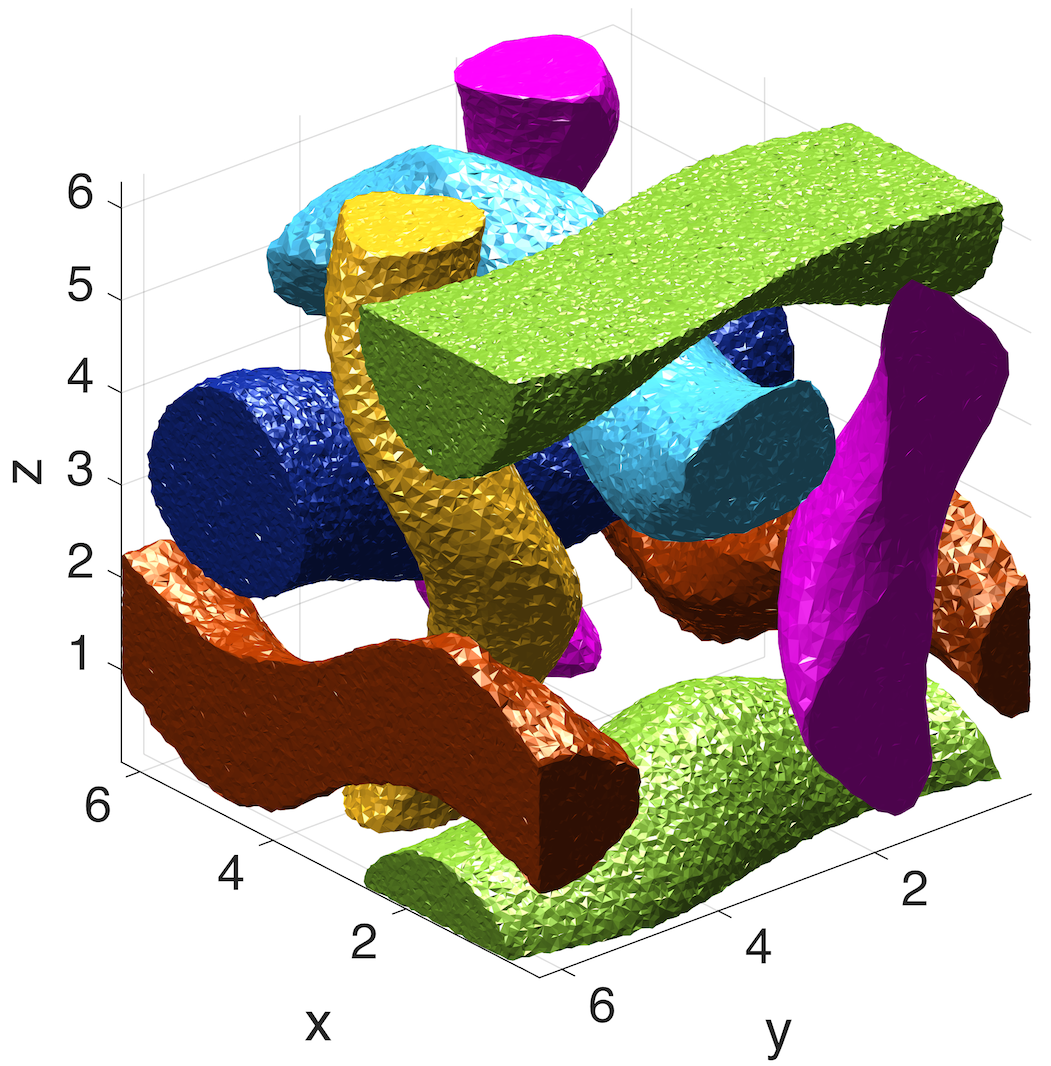}
\caption{Left: the six coherent vortices extracted by a~$7$-clustering of the eigenvectors, using the data points at final time. The seventh cluster, the region between the vortices, is not shown. Right: the boundary of the same six coherent data point sets, computed by Matlab's \texttt{boundary} function.}
\label{fig:ABCclusts}
\end{figure}

Figure~\ref{fig:ABCclusts} shows the clusters that indicate the six invariant vortices. Note that the vortices in this autonomous system do not move in space, hence the clusters look the same at every time slice. The right-hand side of Figure~\ref{fig:ABCclusts} shows the boundaries of the clusters computed by using the data points from all time slices.

\begin{figure}[htb]
\centering
\includegraphics[width = 0.49\textwidth]{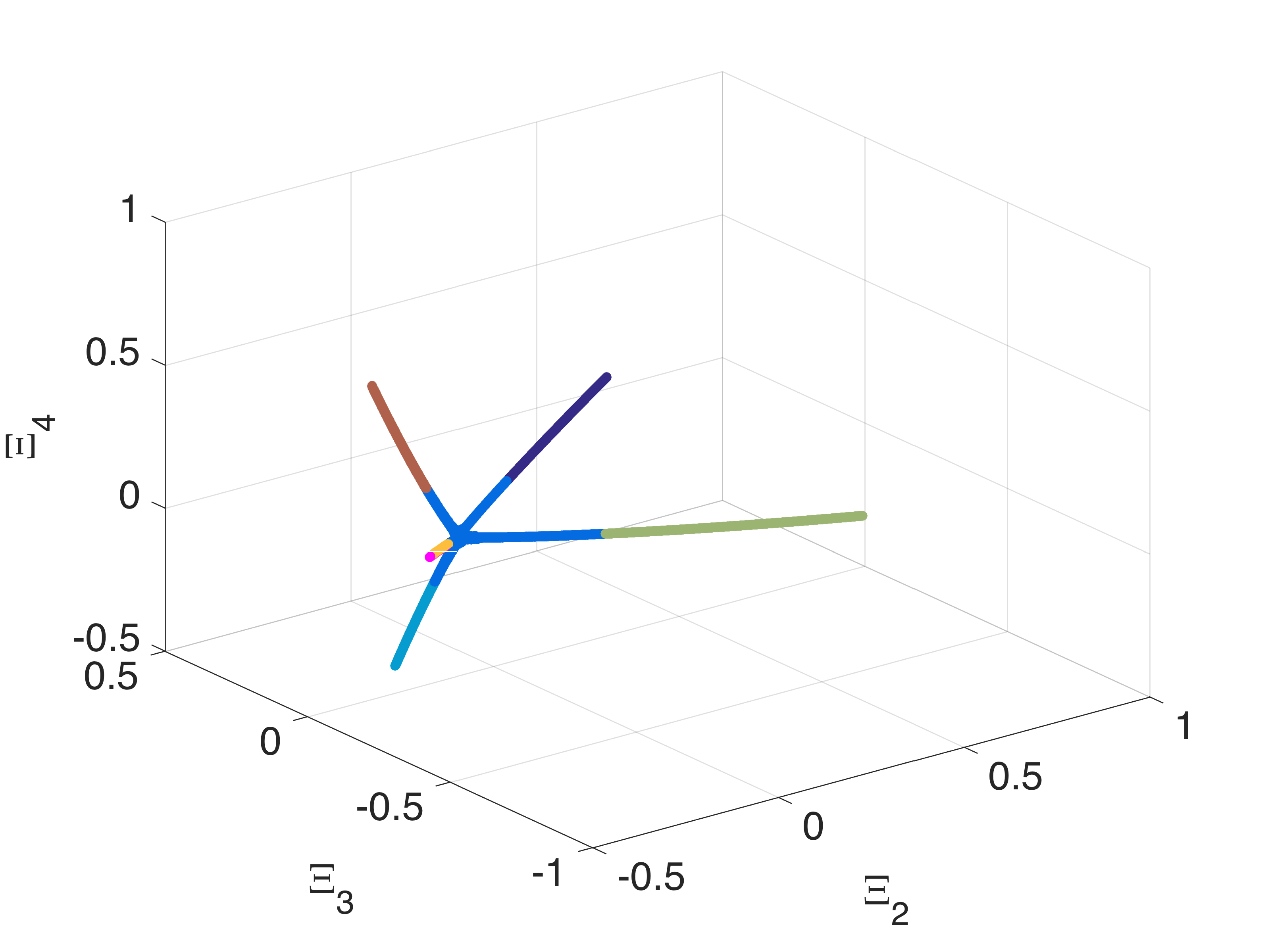}
\hfill
\includegraphics[width = 0.49\textwidth]{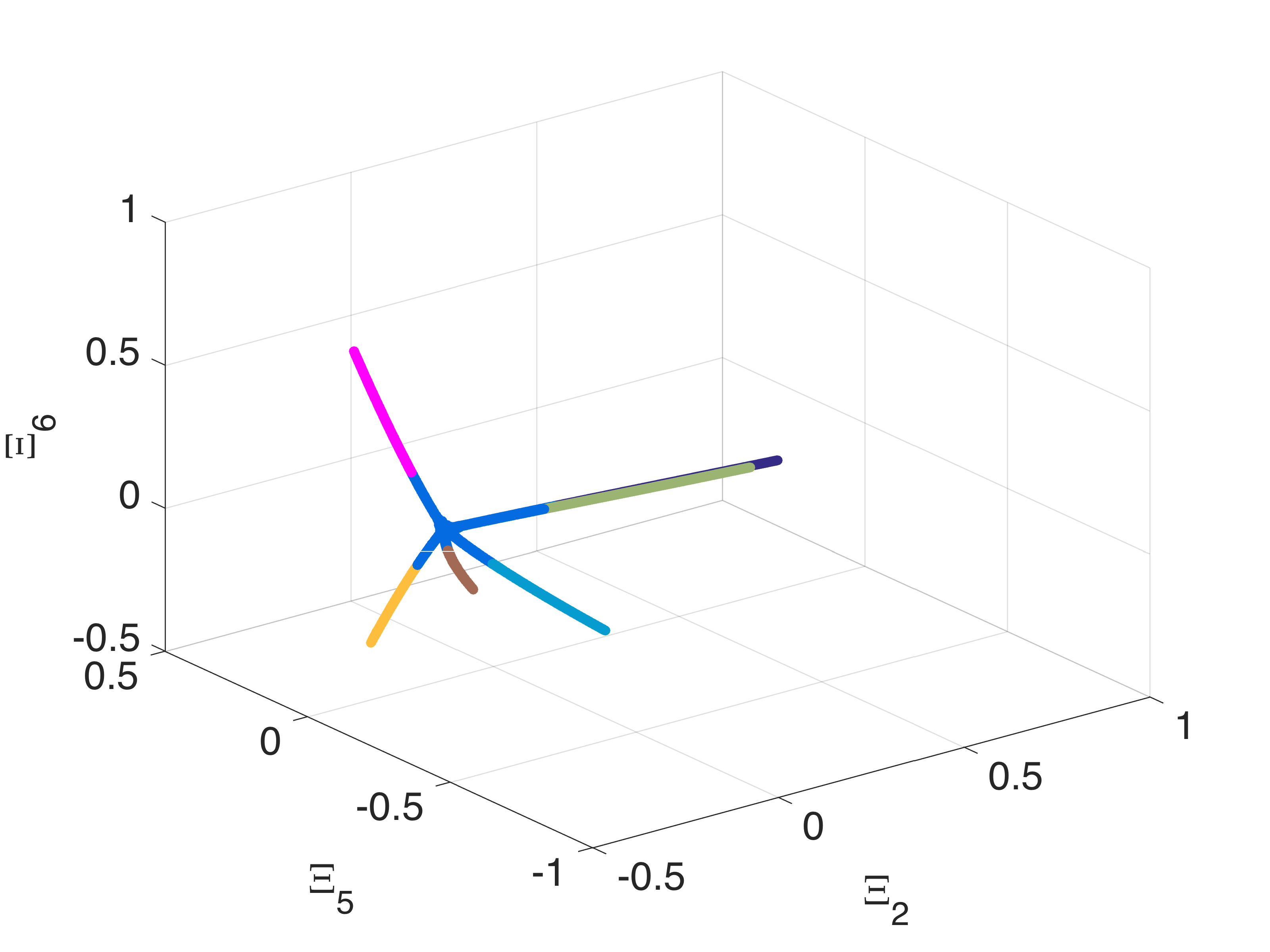}
\caption{Eigenvector-embedding of the data into~$\R^3$ (left:~$\Xi_2$, $\Xi_3$, and~$\Xi_4$; right:~$\Xi_2$,~$\Xi_5$, and~$\Xi_6$), with colors identical to those of the clusters in Figure~\ref{fig:ABCclusts}.}
\label{fig:ABCevals}
\end{figure}

Figure~\ref{fig:ABCevals} shows the embedding of the data  by three different eigenvectors, respectively. Note that the star-shaped geometry indicates that transport between the vortices can only occur through the ``transition region'' between the vortices. This was similar in the Bickley jet example, but with more than one single transition region. However, for the ocean drifters, the topology of continents and ocean basins resulted in a quite different dynamical connectivity pattern; cf~Figure~\ref{fig:drifter_evals}.

\section{Conclusion}

In this article, we provided a data-driven method for the detection of coherent sets. Our main result is Theorem~\ref{thm:main}, which establishes a connection between our method and the ``forward-diffuse-backward'' transfer operator~$\mathcal{T}^*\mathcal{T}$ studied within the analytical framework of coherence~\cite{Froyland2013}. This allows us to give meaning to the dominant eigenfunctions~$\Xi_i$ of~$\mathbf{Q}_\eps$, which represent our main computational output: They are approximations of the respective eigenfunctions of a time-averaged version of~$\mathcal{T}^*\mathcal{T}$. We use the~$\Xi_i$ in two ways: (i) To detect coherent sets via spectral clustering, and (ii) as ``dynamical coordinates'' which can be used to reveal the intrinsic low-dimensional organization of the trajectory data, such as the connectivity structure between clusters.

Coherent sets do not have to have crisp boundaries, e.g.~in the form of some transport barriers. One can enlarge of reduce these sets, by potentially taking small fluctuations in the amount of transport between the set and its surrounding into account. Moreover, in the case of sparse data, or where the underlying dynamics has a stochastic component, ``hard'' assignment of trajectories to coherent sets may not even make sense. In these situations, a ``soft'', fuzzy assignment seems to be more appropriate~\cite{FrPa15}. In future work on non-deterministic systems, we shall analyze this thoroughly.

We based our method on diffusion maps, but we expect that the techniques presented herein can be used to analyze other methods based on the construction of similarity graphs between trajectories in a similar manner. However, we found that using diffusion maps has several important advantages:
\begin{itemize}
\item The diffusion kernel function~$h(x) = \exp(-x)\mathbf{1}_{x\leq r}$ is numerically very well behaved.
\item We need very little a priori knowledge about the system at hand. In fact, we only need a distance function $\|\cdot\|$ which is a good approximation of the actual intrinsic distance \emph{locally}. No knowledge about invariant measures, the global geometry of the state space, or a good set of basis functions is assumed. In all of our numerical examples we just used the Euclidean distance, even though the state spaces considered included spherical and toroidal geometries.
\item Only a single scale parameter~$\eps$ needs to be tuned, and we provided criteria for doing so.
\end{itemize}
These properties indicate that space-time diffusion maps are well suited as an analysis tool for trajectory data generated from a ``black box'' dynamical system. Our numerical experiments suggest robust results even for very sparse and incomplete data.

There are further two possible directions to extend this research. First, one could consider stochastic dynamics. The transfer operator framework used here incorporates this case but the pointwise assertions of Theorem~\ref{thm:main} will have to be replaced by suitable local averages due to the noise in the dynamics. Second, one could consider noisy, incomplete, and even corrupted observations. For example, the data might be of the form $y_t = Ox_t + \eta_t$, where $x_t$ is the true state of the system, $O$ is some linear operator and $\eta_t$ is additive noise. There is evidence that diffusion maps is robust under additive noise \cite{coifman2006}, but if observations are incomplete then Euclidean distances between observations will not represent distances between the underlying states even locally, and one has to resort to other techniques.

\section*{Acknowledgments}
RB was supported by EPSRC Grant No.\ EP/K039512/1. PK was supported by the Einstein Center for Mathematics Berlin (ECMath), project grant CH7. The authors would like to thank Gary Froyland for helpful comments.

\appendix

\section{Proofs}

\subsection{Reversed dynamics and adjoint operator}\label{sec:adjoint}

To see that $\T^*$ is the forward operator of the backward dynamics, let~$\bm x, \bm y$ be random variables with distributions~$\bm x\sim \mu$ and~$\bm y\sim k(\bm x,\cdot)d\ell$; that is, we think of~$\bm y$ as the image of~$\bm x$ under the non-deterministic dynamics~$\Psi$. Then, for any~$\mu$-measurable set~$\set{B}$, and~$\nu$-measurable set~$\set{A}$, we have by Bayes' law for the probability that the time-reversed dynamics ends up in~$\set{B}$ after one step, provided it started in~$\set{A}$, that
\begin{eqnarray*}
\prob [\bm x\in \set{B}\,\vert\, \bm y\in \set{A} ] & = & \frac{\prob [\bm x\in \set{B} ,\ \bm y\in \set{A}]}{\prob [\bm y\in \set{A}]} = \frac{\mu(\set{B})\int_{\set{A}} \T\tfrac{\mathbf{1}_{\set{B}}}{\mu(\set{B})}d\nu}{\nu(\set{A})} = \frac{1}{\nu(\set{A})}\int \mathbf{1}_{\set{A}} \T\mathbf{1}_{\set{B}}\,d\nu\\
& = & \frac{1}{\nu(\set{A})}\int \mathbf{1}_{\set{B}} \T^* \mathbf{1}_{\set{A}}\,d\mu
 =  \int_{\set{B}} \T^* \frac{\mathbf{1}_{\set{A}}}{\nu(\set{A})}\,d\mu\,.
\end{eqnarray*}

\subsection{Diffusion map and diffusion operators} \label{app:diffusion_finite_eps}

We are going to need the following finite $\eps$ version of the result (\ref{eq:Dmaps_eps_convergence}):

\begin{lemma}\label{lemma:diffusion_finite_eps}
Recall from~\eqref{eq:diffusion_operators} that
\[
D_\eps f = \frac{1}{\eps^{d/2}}\int_{\set{M}}k_\eps(\cdot,y)f(y)\dd y, \quad \mathcal{D}_\eps f = \frac{D_\eps(qf)}{D_\eps f}
\]
Assume that $\set{M}$ has no boundary and that $f$ is bounded on $\set{M}$. Then
\begin{itemize}
\item[(i)] $P_{\eps,0}f = \cDep f$.
\item[(ii)] If further $f,q \in C^3(\set{M})$, then
\begin{equation}
\label{eq:diffmaps_finiteeps}
P_{\eps,\alpha}f = \frac{D_\eps\left(q^{1-\alpha}f\right)}{D_\eps q^{1-\alpha}} + \mathcal{O}\left(\eps^2\right).
\end{equation}
In particular,~$P_{\eps,1}f = D_\eps f + \mathcal{O}\left(\eps^2\right)$.
\end{itemize}
\end{lemma}

\begin{proof}
On the one hand, from~\eqref{eq:Dmaps_eps_convergence} we have
\[
P_{\eps,\alpha}f(x) = f(x) + \eps \left(\frac{\Delta (f q^{1-\alpha})}{q^{1-\alpha}} - \frac{\Delta (q^{1-\alpha})}{q^{1-\alpha}}f\right) + \mathcal{O}(\eps^2).
\]
On the other hand, by using $D_\eps f = f + \eps\Delta f + \mathcal{O}(\eps^2)$ and the product rule, we have
\[
\frac{D_\eps(q^{1-\alpha}f)}{D_\eps q^{1-\alpha}} = f + \eps \left(\frac{\Delta (f q^{1-\alpha})}{q^{1-\alpha}} - \frac{\Delta (q^{1-\alpha})}{q^{1-\alpha}}f\right) + \mathcal{O}(\eps^2).
\]
Thus the left-hand side and right-hand side of~\eqref{eq:diffmaps_finiteeps} are equal up to~$\mathcal{O}(\eps^2)$. To show~$P_{\eps,0}f = \mathcal{D}_\eps f$, note that almost surely
\begin{equation}
\label{eq:lem1_step1}
P_{\eps,0}f(x) = \lim_{m\rightarrow\infty} \frac{\frac{1}{m\eps^{d/2}}\sum_{j=1}^m k_\eps(x,x^j)f(x^j)}{\frac{1}{m\eps^{d/2}}\sum_{k=1}^m k_\eps(x,x^k)} = \frac{D_\eps(fq)(x)}{D_\eps q(x)} = \mathcal{D}_\eps f(x).
\end{equation}
This finishes the proof.
\end{proof}

\subsection{Proof of Lemma \ref{lemma:fbdiffmaps}} \label{app:fbdiffmaps}

\begin{enumerate}[(i)]
\item follows directly from (\ref{eq:Bmatrix}).
\item Recall that by Lemma \ref{lemma:diffusion_finite_eps} (i), $P_{\eps,0}f = \mathcal{D}_\eps f$. Now note that we can write $\mathcal{B}_\eps f(x) = \frac{1}{d_\eps^\infty(x)}P^*_{\eps,0}P_{\eps,0}f(x)$ where
\[
P^*_{\eps,0}f(x) = \frac{1}{\eps^{d/2}}\int k_\eps(x,y)\frac{q(y)}{q_\eps(y)} f(y)\dd y 
\]
is the adjoint of $P_{\eps,0}$. This follows by inspecting (\ref{eq:Bkernel_infty}) and (\ref{eq:bInf}). Further note that $d_\eps^\infty = P^*_{\eps,0} \mathbf{1}$. Set $h = \mathbf{P}_{\eps,0}f$ and estimate
\begin{align*}
\left| \mathbf{B}_\eps f(x^i) - \mathcal{B}_\eps f(x^i)\right| &\; = \left|\left(\mbox{diag}(\mathbf{P}^T_{\eps,0} \mathbf{1})\right)^{-1}\mathbf{P}_{\eps,0}^T\mathbf{P}_{\eps,0}f(x^i) - \frac{1}{d_\eps^\infty(x^i)} P_{\eps,0}^* P_{\eps,0}f(x^i) \right| \\
& \; \leq \left|\left(\mathbf{P}^T_{\eps,0} \mathbf{1}(x^i)\right)^{-1} - (P_{\eps,0}^*\mathbf{1}(x^i))^{-1}\right| \mathbf{P}_{\eps,0}^T\mathbf{P}_{\eps,0}f(x^i)  \\
 & \quad + \frac{1}{P_{\eps,0}^*\mathbf{1}(x^i)} \left| \mathbf{P}_{\eps,0}^T h(x^i) - P_{\eps,0}^* h(x^i)\right| + \frac{1}{P_{\eps,0}^*\mathbf{1}(x^i)} \left| P^*_{\eps,0}( \mathbf{P}_{\eps,0}f - P_{\eps,0}f)(x^i)\right|
\end{align*}
In \cite{Hein2005}, it was shown that $|\mathbf{P}_{\eps,0}f(x^i) - P_{\eps,0}f(x^i)| = \mathcal{O}(\eps^{-d/4}m^{-1/2})$ holds uniformly. $|\mathbf{P}_{\eps,0}^Tf(x^i) - P^*_{\eps,0}f(x^i)| = \mathcal{O}(\eps^{-d/4}m^{-1/2})$ follows, and since $\mathbf{P}_{\eps,0}$ and $P_{\eps,0}$ are bounded, we have
\[
\left| \mathbf{B}_\eps f(x^i) - \mathcal{B}_\eps f(x^i)\right| = \mathcal{O}(\eps^{-d/4}m^{-1/2}) + \frac{1}{P^*_\eps\mathbf{1}(x^i)} \mathcal{O}(\eps^{-d/4}m^{-1/2}).
\]
 (ii) follows by using the Taylor expansion $q_\eps = q + \mathcal{O}(\eps)$, which holds uniformly on $\set{M}$ \cite{coifman2006} and gives $P_\eps^* \mathbf{1} = \mathbf{1} + \mathcal{O}(\eps)$.
\item As already noted in (ii), we can write $\mathcal{B}_\eps f(x) = \frac{1}{d_\eps^\infty(x)}P^*_{\eps,0}g(x)$ with $g := \mathcal{D}_\eps f$. We use the expansion $q^{-\alpha}_\eps = q^{-\alpha}\left(1 + \alpha\eps\left(\frac{\Delta q}{q} - \omega\right)\right) + \mathcal{O}(\eps^2)$ which holds uniformly on $\set{M}$ with $\omega:\R^n\rightarrow \R$ being a potential term depending on the embedding of $\set{M}$, see \cite{coifman2006}. With the shorthand $h_q := \frac{\Delta q}{q} - \omega$, (\ref{eq:dinfty}) yields
\[
d_\eps^\infty = \mathbf{1} + \eps D_\eps h_q + \mathcal{O}(\eps^2)
\]
uniformly on $\set{M}$. Using the same expansion on $P^*_{\eps,0}g$ gives $P^*_{\eps,0}g = D_\eps\left[ g + \eps h_q g\right] + \mathcal{O}(\eps^2)$, and finally
\[
\mathcal{B}_\eps f = \frac{D_\eps g + \eps D_\eps[h_q g]}{\mathbf{1} + \eps D_\eps h_q} + \mathcal{O}(\eps^2)
\]
uniformly on $\set{M}$, where the quotient is taken pointwise. Using $\frac{a+\eps c}{b + \eps d} = \frac{a}{b} + \eps \frac{c}{b} - \eps \frac{ad}{b^2} + \mathcal{O}(\eps^2)$ this can be rewritten as
\[
\mathcal{B}_\eps f = D_\eps g + \eps \left( D_\eps[h_q g] - D_\eps [h_q] D_\eps [g]\right) + \mathcal{O}(\eps^2)
\]
uniformly on $\set{M}$. Now note that since~$D_\eps = I + \mathcal{O}(\eps)$, where~$\mathcal{O}(\eps)$ depends on the first derivative of the argument of~$D_{\eps}$, the term~$D_\eps[h_q g] - D_\eps [h_q] D_\eps [g]$ is~$\mathcal{O}(\eps)$ uniformly on $\set{M}$ \cite[Lemma~8]{coifman2006}. This together with~$g = \mathcal{D}_\eps f$ shows (iii).

\item From (ii), we know $|\mathbf{B}_\eps f(x^i) - \mathcal{B}_\eps f(x^i)| = \mathcal{O}(\eps^{-d/4}m^{-1/2})$ and  $|\mathbf{B}^T_\eps f(x^i) - \mathcal{B}_\eps^* f(x^i)| = \mathcal{O}(\eps^{-d/4}m^{-1/2})$. From (iii), $|\mathcal{B}_\eps f(x^i) - \mathcal{B}^*_\eps f(x^i)| = \mathcal{O}(\eps^2)$. Then, for any bounded $f$,
\begin{align*}
\left| \mathbf{B}_\eps f(x^i) - \mathbf{B}_\eps^T f(x^i) \right| & \; \leq \left|\mathbf{B}_\eps f(x^i) - \mathcal{B}_\eps f(x^i)\right| + \left| \mathcal{B}_\eps f(x^i) - \mathcal{B}_\eps^* f(x^i)\right| + \left| \mathbf{B}^T_\eps f(x^i) - \mathcal{B}_\eps^* f(x^i)\right|\\
&\; = \mathcal{O}(\eps^{-d/4}m^{-1/2}) + \mathcal{O}(\eps^2)
\end{align*}
uniformly on $\set{M}$. Since this holds for any bounded $f$ and for any $x^i$, $i=1,\ldots, m$, the result follows for any compatible matrix norm $\| \cdot \|$.
\item If $q_\eps = q$, then $d_\eps^\infty = \mathbf{1}$, and $P^*_{\eps,0}f = D_\eps f$. Then $\mathcal{B}_\eps f(x) = \frac{1}{d_\eps^\infty(x)}P^*_{\eps,0}P_{\eps,0}f(x) = D_\eps\mathcal{D}_\eps f(x)$ follows for any $x\in\set{M}$.

\item By (iv), $\lim_{\eps \rightarrow 0}\frac{1}{\eps}(\mathcal{B}_\eps f - f) = \frac{d}{d\eps} (D_\eps \mathcal{D}_\eps f) |_{\eps = 0} = \frac{d}{d\eps} D_\eps f|_{\eps = 0} + \frac{d}{d\eps}\mathcal{D}_\eps f|_{\eps = 0}$. By Lemma \ref{lemma:diffusion_finite_eps} (ii) and (\ref{eq:Dmaps_eps_convergence}),
\[
\frac{d}{d\eps} D_\eps f|_{\eps = 0} = \frac{d}{d\eps} P_{\eps,1}f|_{\eps = 0} = \frac14\Delta f
\]
and
\[
\frac{d}{d\eps} \mathcal{D}_\eps f|_{\eps = 0} = \frac{d}{d\eps} P_{\eps,0}f|_{\eps = 0} = \frac{\Delta(fq)}{4q} - \frac{\Delta q}{4q}f = \frac14\Delta f + \frac12\nabla q \cdot \nabla f.
\]
Thus
\[
\lim_{\eps \rightarrow 0}\frac{1}{\eps}(\mathcal{B}_\eps f - f) = \frac12\Delta f +  \frac12\nabla q \cdot \nabla f = \frac12 q^{-1}\nabla\cdot (q\nabla f).
\]
\end{enumerate}

%

\subsection{Proof of Theorem \ref{thm:main}} \label{app:mainthm}

We have, by the definition of $\mathbf{Q}_{\eps}$,
\[
\sum_{j=1}^m \mathbf{Q}_\eps(i,j)f(x^j) = \frac{1}{T}\sum_{t\in I_t}\sum_{j=1}^m b_{\eps,t}(\Phi_tx^i,\Phi_tx^j)f(x^j).
\]
For one fixed $t\in I_t$, we label the images $\{\Phi_tx^1,\ldots, \Phi_t x^j\}$ with $\{y^1,\ldots, y^l\}$, taking into account possible duplicates since $\Phi_t$ is not assumed to be injective (if it is, then $l=m$). Define
\[
\hat g(y_k) := \frac{\sum_{j=1}^m \delta(y^k - \Phi_t x^j)f(x^j)}{\sum_{j=1}^m \delta(y^k - \Phi_t x^j)}.
\]
Then
\begin{equation}\label{eq:hatg}
\sum_{j=1}^m b_{\eps,t}(\Phi_t x^i,\Phi_tx^j)f(x^j) = \sum_{k=1}^l b_{\eps,t}(\Phi_t x^i , y^k) d(y^k)\hat g(y^k).
\end{equation}
where $d(y^k):= \sum_{j=1}^m \delta(y^k = \Phi_t x^j)$ counts the multiplicities of the $y^k$'s. Define further $g := \mathcal{P}_t f$. Now the proof has 3 steps:
\begin{enumerate}
\item For any test function $h$, we have $\left| \sum_k h(y^k)d(y^k)\hat g(y^k) - \langle h,g\rangle_{q_t}\right| = \mathcal{O}(m^{-1/2})$. This follows since
\[
\sum_{k=1}^l h(y^k)d(y^k)\hat g(y^k) = \sum_{k=1}^l \sum_{j=1}^m h(y^k) \delta(y^k - \Phi_tx^j) f(x^j) = \sum_{j=1}^m h(\Phi_tx^j)f(x^j),
\]
and the right-hand side is an unbiased Monte Carlo estimator of $\int h(\Phi_t x) f(x) q_0(x) \dd x = \langle U_t h,f\rangle_{q_0} = \langle h, \mathcal{P}_t f\rangle_{q_t}$ with variance $\mathcal{O}(m^{-1})$.
\item $\left|\langle h, g\rangle_{q_t} - \sum_k h(y^k)d(y^k) g(y^k)\right| = \mathcal{O}(m^{-1/2})$. This follows since the data points $\{y^k\}_{k=1}^l$, taking into account their multiplicities $d(y^k)$, are i.i.d.~$q_t$-distributed.
\item The kernel~$b_{\eps}$ from~\eqref{eq:Bmatrix} reads explicitly as
\begin{equation}\label{eq:Bkernel}
b_\eps(x,y) = \frac{1}{d_\eps(x)}\sum_{i=1}^m \frac{k_\eps(x,x^i)k_\eps(x^i,y)}{k_\eps(x^i)^2}, \quad d_\eps(x) := \sum_{i=1}^m \frac{k_\eps(x,x^i)}{k_\eps(x^i)}\,.
\end{equation}
Observe that~$b_{\eps,t}(x,y)$ from~\eqref{eq:bt} is exactly equal to~$b_\eps(x,y)$ in~\eqref{eq:Bkernel}, but with the~$q$-distributed samples~$x^i$ replaced by~$q_t$-distributed samples~$\Phi_t x^i$. By Lemma \ref{lemma:fbdiffmaps} (ii), we thus have
\begin{equation}
\label{eq:Bepst}
\left| \sum_k b_{\eps,t}(x, y^k)d(y^k)g(y^k) - \mathcal{B}_{\eps,t}g(x)\right| = \mathcal{O}(\eps^{-d/4}m^{-1/2})
\end{equation}
where $\mathcal{B}_{\eps,t}f(x)$ is defined as
\[
\mathcal{B}_{\eps,t}f(x) = \int_\X b^\infty_{\eps,t}(x,y)q_t(y) f(y) \dd y,
\]
and $b_{\eps,t}^\infty$ is given by the same formula as~$b_{\eps}^{\infty}$ in~\eqref{eq:bInf} and~\eqref{eq:dinfty}, with~$q$ replaced by~$q_t$ everywhere.
\end{enumerate}
Now for the test function $h(\cdot) = \mathcal{B}_{\eps,t}(x,\cdot)$, the steps 1.-3.\ allow the estimate
\begin{align*}
\left| \sum_k h(y^k)d(y^k)\hat g(y^k) - b_{\eps,t}g(x)\right| & \; \leq \left| \sum_k h(y^k)d(y^k)\hat g(y^k) - \langle h,g\rangle_{q_t}\right|  + \left|\langle h, g\rangle_{q_t} - \sum_k h(y^k)d(y^k) g(y^k)\right| \\
& \; + \left| \sum_k h(y^k)d(y^k)g(y^k) - \mathcal{B}_{\eps,t}g(x)\right| \\
& \; = \mathcal{O}(m^{-1/2}) + \mathcal{O}(m^{-1/2}) + \mathcal{O}(\eps^{-d/4}m^{-1/2}).
\end{align*}
But in view of the definition of $g$, we have $\mathcal{B}_{\eps,t} g(\Phi_tx^i) = U_t \mathcal{B}_{\eps,t}\mathcal{P}_t f(x^i)$. By using (\ref{eq:hatg}), we can rewrite the estimate above as
\[
\left| \sum_{j=1}^m b_{\eps,t}(\Phi_t x^i,\Phi_tx^j)f(x^j) - U_t \mathcal{B}_{\eps,t}\mathcal{P}_t f(x^i)\right| = \mathcal{O}(m^{-1/2}) + \mathcal{O}(\eps^{-d/4}m^{-1/2})
\]
Finally, Property (iii) of Lemma \ref{lemma:fbdiffmaps}, $\mathcal{B}_{\eps,t}f = D_\eps \mathcal{D}_{\eps,t}f + \mathcal{O}(\eps^2)$, finishes the proof. $\blacksquare$

\small

\bibliographystyle{plain}
\bibliography{References}

\end{document}